\documentclass[11pt]{amsart}

\usepackage[T1]{fontenc}
\usepackage[utf8]{inputenc}
\usepackage[letterpaper,margin=1in]{geometry}

\usepackage{amsmath,amssymb,amsfonts,amsthm}
\usepackage{mathtools}

\usepackage{graphicx}
\usepackage[shortlabels]{enumitem}
\usepackage{multicol}
\usepackage{dsfont}
\usepackage{mathrsfs}
\usepackage{xcolor}
\usepackage{appendix}
\usepackage{cancel}
\usepackage{tikz}
\usepackage[all]{xy}

\makeatletter

\renewcommand\section{\@startsection{section}{1}{\z@}%
  {-1.2\baselineskip plus -.2\baselineskip minus -.2\baselineskip}%
  {0.7\baselineskip}%
  {\normalfont\large\bfseries\centering}}

\renewcommand\subsection{\@startsection{subsection}{2}{\z@}%
  {-0.9\baselineskip plus -.2\baselineskip minus -.2\baselineskip}%
  {0.45\baselineskip}%
  {\normalfont\bfseries}}

\makeatother

\usepackage[hidelinks]{hyperref}

\newcommand{\mb}{\mathbb}

\newcommand{\on}{\operatorname}

\renewcommand{\r}{\mathbf{r}}
\renewcommand{\u}{\mathbf{u}}
\newcommand{\x}{\mathbf{x}}
\newcommand{\y}{\mathbf{y}}
\newcommand{\z}{\mathbf{z}}
\newcommand{\e}{\mathbf{e}}

\newcommand{\Supp}{\operatorname{Supp}}
\renewcommand{\P}{\mathbb{P}}

\newcommand{\GL}{\operatorname{GL}}
\newcommand{\mc}{\mathcal}

\renewcommand{\o}{\mathfrak{o}}

\newcommand{\m}{\mathfrak{m}}

\renewcommand{\bar}{\overline}
\providecommand{\Bar}{\overline}
\renewcommand{\Bar}{\overline}

\newtheorem{theorem}{Theorem}[section]
\newtheorem{proposition}[theorem]{Proposition}
\newtheorem{lemma}[theorem]{Lemma}
\newtheorem{corollary}[theorem]{Corollary}

\theoremstyle{definition}
\newtheorem{definition}[theorem]{Definition}
\newtheorem{example}[theorem]{Example}
\newtheorem{question}[theorem]{Question}

\theoremstyle{remark}
\newtheorem{remark}[theorem]{Remark}

\title{Counting subalgebras of $\mathfrak{o}^n$}

\author{Aaron Blas Pereda}
\address{CIEM-FAMAF, Universidad Nacional de C\'ordoba, 5000 C\'ordoba, Argentina}
\email{aaronpereda@mi.unc.edu.ar}

\author{Diego Sulca}
\address{CIEM-FAMAF, Universidad Nacional de C\'ordoba, 5000 C\'ordoba, Argentina}
\email{diego.a.sulca@unc.edu.ar}


\makeatletter
\def\l@section{\@tocline{1}{6pt}{1pc}{}{}}
\def\l@subsection{\@tocline{2}{0pt}{2pc}{5pc}{}}
\makeatother

\begin{document}

\begin{abstract}
Let $\mathfrak{o}$ be a compact discrete valuation ring with residue
field $\mathbb{F}_q$. We study the cotype zeta function of finite-index
subalgebras of $\mathfrak{o}^n$ by adapting an approach developed by Voll
for more general algebras.

For each $\emptyset\neq I\subseteq\{1,\ldots,n-1\}$, we express the generating series
counting subalgebras whose elementary divisor valuations, listed in
decreasing order, have jumps precisely at the positions in $I$ as a
finite weighted sum of $\mathfrak{o}$-adic integrals. We compute these
integrals in several cases, obtaining explicit formulas for the
generating series associated with every singleton set $I$ in arbitrary
rank. We also obtain the cotype zeta function of unital subalgebras of
$\mathfrak{o}^4$. As a by-product, we obtain a universal lower bound for
the abscissa of convergence of the subalgebra zeta function of any
$\mathfrak{o}$-algebra that is free of rank $n$ as an
$\mathfrak{o}$-module.

We also establish several results concerning the uniformity problem for the subalgebra zeta function of \(\mathfrak{o}^n\). To do so, we study two
systems of polynomial equations over the residue field and apply the results to derive the number of irreducible unital subalgebras of $\mathfrak{o}^n$
of index $q^e$, for $e\leq n+3$. Finally, we use the decomposition
of subalgebras into local factors to obtain the number of unital subalgebras
of $\mathfrak{o}^n$ of index $q^e$ for $e\leq 9$. For each fixed $n$, these subalgebra
counts are universal polynomials in $q$.
\end{abstract}

\maketitle

\tableofcontents

\section{Introduction}

We consider the ring $\mathbb{Z}^n$ equipped with componentwise addition and multiplication. We call a subring or subalgebra \emph{unital} if it contains the identity of the ambient ring. For each positive integer $k$, let $f_n(k)$ be the number of unital subrings of $\mathbb{Z}^n$ of index $k$. 
The sequence $\bigl(f_n(k)\bigr)_{k\ge1}$ and the asymptotic behavior of its partial sums have attracted considerable attention in recent years because of deep connections with the distribution of orders in number fields; see \cite{KaplanMarcinekTaklooBighash2015,AKKM2021,Isham2022,Isham2023,MishraRay2022}.

A natural generating function for the sequence $\bigl(f_n(k)\bigr)_{k\ge1}$ is the \emph{zeta function}
\[
   \zeta_{\mathbb{Z}^n}^{\scriptscriptstyle 1,<}(s) := \sum_{k=1}^\infty \frac{f_n(k)}{k^s}.
\]
More generally, let $L$ be a ring additively isomorphic to $\mb{Z}^n$. We do not assume that $L$ is commutative or associative.
Following
Grunewald, Segal, and Smith \cite{GrunewaldSegalSmith1988}, one considers the Dirichlet
series
\[
    \zeta_L^{\scriptscriptstyle <}(s) := \sum_{k=1}^\infty \frac{a_L^{\scriptscriptstyle <}(k)}{k^s},
\]
where $a_L^{\scriptscriptstyle <}(k)$ denotes the number of subrings (= subgroups stable under multiplication) of $L$ of index $k$. 
When $L$ has an identity, one may instead look at the series 
\[\zeta_L^{\scriptscriptstyle 1,<}(s)=\sum_{k=1}^\infty \frac{a_L^{\scriptscriptstyle 1,<}(k)}{k^{s}}\]
enumerating the finite-index unital subrings of $L$.

\medskip

The Dirichlet series $\zeta_L^{\scriptscriptstyle <}(s)$ (or $\zeta_L^{\scriptscriptstyle 1,<}(s)$) enjoys the following fundamental properties:
\begin{enumerate}
    \item 
    It admits an Euler factorization
    \[
        \zeta_L^{\scriptscriptstyle <}(s) = \prod_{p\ \mathrm{prime}} \zeta_{L_p}^{\scriptscriptstyle <}(s),
    \]
    where $L_p := L\otimes_{\mathbb Z}\mathbb Z_p$ ($\mb{Z}_p$ being the ring of $p$-adic integers) and
    \[
        \zeta_{L_p}^{\scriptscriptstyle <}(s)
        := \sum_{\Lambda \leq L_p} [L_p:\Lambda]^{-s}
        = \sum_{e=0}^\infty \frac{a_L^{\scriptscriptstyle <}(p^e)}{p^{es}},
    \] 
    the summation in the middle being over the finite-index subrings of $L_p$.
    Each local factor is known to be a rational function in $p^{-s}$
    \cite{GrunewaldSegalSmith1988}.
    
    \item   
    Unless it is constant, it has a rational abscissa of convergence, and it
    extends meromorphically beyond this abscissa
    \cite{duSautoyGrunewald2000}.
    
    \item   
    If $\alpha_L^{\scriptscriptstyle <}$ is the abscissa of convergence and the continued function has a pole of order $\beta_L^{\scriptscriptstyle <}$ at $s=\alpha_L^{\scriptscriptstyle <}$, then
    \[
        \sum_{k\le N} a_L^{\scriptscriptstyle <}(k)
        \sim C N^{\alpha_L^{\scriptscriptstyle <}} (\log N)^{\beta_L^{\scriptscriptstyle <}-1}
        \qquad (N\to\infty)
    \]
    for some constant $C>0$ \cite{duSautoyGrunewald2000}.
    
    \item 
    For almost all primes $p$, the local factor $\zeta_{L_p}^{\scriptscriptstyle <}(s)$ satisfies a functional equation
    under $p\mapsto p^{-1}$; see \cite{Voll2010}.
\end{enumerate}

When $L=\mathcal{O}_K$ is the ring of integers of a number field $K$, the zeta function $\zeta_{\mc{O}_K}^{\scriptscriptstyle{1,<}}(s)$ is also called the {\em order zeta function} of $K$, as the orders of $K$ are precisely the finite-index unital subrings of $\mathcal{O}_K$.

If $K$ is a quadratic field, one easily finds that $\zeta_{\mc{O}_K}^{\scriptscriptstyle{1,<}}(s)=\zeta(s)$, the Riemann zeta function. When $K$ is a cubic field, Datskovsky and Wright \cite{DatskovskyWright1988a} obtained
\[
    \zeta_{\mc{O}_K}^{\scriptscriptstyle{1,<}}(s)
    = \frac{\zeta_K(s)}{\zeta_K(2s)} \, \zeta(2s)\,\zeta(3s-1),
\]
where $\zeta_K(s)$ is the Dedekind zeta function of $K$.

When $K$ is a quartic field, Nakagawa \cite{Nakagawa1996} computed almost all the local factors of $\zeta_{\mc{O}_K}^{\scriptscriptstyle{1,<}}(s)$.
On the other hand, no formula for $\zeta_{\mc{O}_K}^{\scriptscriptstyle{1,<}}(s)$ is known if $[K:\mb{Q}]\geq 5$.
Nevertheless,
Kaplan, Marcinek, and Takloo-Bighash \cite{KaplanMarcinekTaklooBighash2015}
showed that for any number field $K$ of degree between $2$ and $5$, the abscissa of convergence of $\zeta_{\mc{O}_K}^{\scriptscriptstyle{1,<}}(s)$ is 1. They also obtained an expression for the order $\beta_K$ of the pole at $s=1$ in terms of the Galois group of the normal closure of $K$. As mentioned above, this suffices to obtain an asymptotic formula of the form
\[
    \sum_{k\le N} a_{\mc{O}_K}^{\scriptscriptstyle{1,<}}(k)
    \sim C_K N (\log N)^{\beta_K-1},
\]
for some $C_K>0$.
 
 If $[K:\mb{Q}]=n$, then the zeta functions $\zeta_{\mc{O}_K}^{\scriptscriptstyle{1,<}}(s)$ and $\zeta_{\mb{Z}^n}^{\scriptscriptstyle{1,<}}(s)$ share the same local factors at primes $p$ that split completely in $K$, since for those primes $\mc{O}_K\otimes_{\mathbb Z}\mb{Z}_p\cong\mb{Z}_p^n$ as rings.
The relationship was recently strengthened by the following result, originally conjectured in \cite{KaplanMarcinekTaklooBighash2015}:
\begin{theorem}[{\cite[Theorem E]{Sulca2023}}]
   If $K$ is a number field of degree $n$, then $\zeta_{\mc{O}_K}^{\scriptscriptstyle{1,<}}(s)$ and $\zeta_{\mathbb{Z}^n}^{\scriptscriptstyle{1,<}}(s)$ have the same abscissa of convergence.
\end{theorem}
Thus the analytic behaviour of $\zeta_{\mathbb{Z}^n}^{\scriptscriptstyle{1,<}}(s)$—in particular its abscissa of convergence—directly informs us about the distribution of orders in degree-$n$ number fields.

\subsection{On the subring zeta function of $\mb{Z}^n$}

\subsubsection{\textbf{Explicit formulas}}

Explicit formulas are known for $n\le4$:
\begin{align*}
\zeta_{\mathbb{Z}^2}^{\scriptscriptstyle{1,<}}(s)=& \zeta(s),\\
\zeta_{\mathbb{Z}^3}^{\scriptscriptstyle{1,<}}(s)=& \frac{\zeta(3s-1) \zeta(s)^3}{ \zeta(2s)^2},\\
 \zeta_{\mathbb{Z}^4}^{\scriptscriptstyle{1,<}}(s) =& \prod_p \frac{1}{\bigl(1 - p^{-s}\bigr)^2 \bigl(1 - p^{2} p^{-4s}\bigr) \bigl(1 - p^{3} p^{-6s}\bigr)}  \Bigl(1 + 4 p^{-s} + 2 p^{-2s} + (4p-3) p^{-3s} + (5p-1) p^{-4s}\\
&+ (p^2 - 5p) p^{-5s}  + (3p^2 - 4p) p^{-6s} - 2p^2 p^{-7s} - 4p^2 p^{-8s} - p^2 p^{-9s}\Bigr).
\end{align*}
The case $n=2$ is easy; the case $n=3$ is originally due to Datskovsky and Wright \cite{DatskovskyWright1988a}; and the case $n=4$ is due to Nakagawa \cite{Nakagawa1996}.
These cases were later rederived using combinatorial
methods by Liu \cite{Liu2007}. No explicit formula is known for $n\ge5$.

In contrast, the subgroup zeta function $\zeta_{\mb{Z}^n}(s)$ (enumerating all finite-index subgroups) is well-understood:
\[
    \zeta_{\mathbb{Z}^n}(s)
    = \zeta(s)\zeta(s-1)\cdots\zeta(s-n+1);
\]
 see \cite{LubotzkySegal2003} for several proofs.

\subsubsection{\textbf{On the abscissa of convergence}}
The above formulas for $\zeta_{\mb{Z}^n}^{\scriptscriptstyle{1,<}}(s)$ show that
$\alpha_{\mb{Z}^n}^{\scriptscriptstyle{1,<}}=1$ for $2\leq n\leq 4$. 
This is also true for $n=5$, while for $n\geq 6$ we have $1\leq \alpha_{\mb{Z}^n}^{\scriptscriptstyle{1,<}}\leq \frac{n}{2}-\frac{7}{6}$; see \cite[Theorem 6]{KaplanMarcinekTaklooBighash2015}.
A sharper lower bound for $n\ge7$ was obtained by Isham
\cite{Isham2022} following ideas of Brakenhoff \cite{Brakenhoff2009}:
\begin{theorem}[{\cite[Theorem 1.10]{Isham2022}}]
    \label{th:lower-bound}
    Let $n\ge7$. Then
    \[
        d_7(n):=\max_{0\le d \le n-1}
        \frac{d(n-1-d)}{n-1+d}
        < \alpha_{\mathbb{Z}^n}^{\scriptscriptstyle{1,<}}.
    \]
\end{theorem}
Strictly speaking, Isham's theorem as originally stated gives the lower bound $d_7(n)$ for the abscissa of convergence of each local factor $\zeta_{\mb{Z}_p^n}^{\scriptscriptstyle{1,<}}(s)$. 
However, the results of \cite[Section 4]{duSautoyGrunewald2000} imply that the abscissa of convergence of the global zeta function $\zeta_{\mb{Z}^n}^{\scriptscriptstyle{1,<}}(s)$ is strictly greater than the  abscissa of convergence of each of its local factors.

The inequality in the theorem implies that $\alpha_{\mb{Z}^n}^{\scriptscriptstyle{1,<}}>1$ for $n\geq 7$ and $\liminf_{n\to \infty}\frac{\alpha_{\mb{Z}^n}^{\scriptscriptstyle{1,<}}}{n}\geq 3-2\sqrt{2}$.

\subsubsection{\textbf{Uniformity}}

Although each $\zeta_{\mathbb{Z}_p^n}^{\scriptscriptstyle{1,<}}(s)$ is a rational function in $p^{-s}$, it is unknown how this rational expression depends on $p$.  We say that $\zeta_{\mathbb{Z}^n}^{\scriptscriptstyle{1,<}}(s)$ is \emph{uniform} if there exists a rational function $W_n(X,Y) \in \mathbb{Q}(X,Y)$ such that
\[
   \zeta_{\mathbb{Z}_p^n}^{\scriptscriptstyle{1,<}}(s) = W_n(p,p^{-s})
\]
for every prime $p$.

\begin{question}[{\cite[Section 7]{Liu2007},\ \cite[Question 3.7]{Voll2015}}]
   Is $\zeta_{\mathbb{Z}^n}^{\scriptscriptstyle{1,<}}(s)$ uniform?
\end{question}

The answer is yes for $n\leq 4$ (as the explicit formulas show), but remains open for $n\geq 5$. This question is equivalent to the following:

\begin{question}\label{question on the uniformity 2}
   Is $f_n(p^e)$ a polynomial in $p$ for every fixed $e\geq 1$?
\end{question}

For small $e$, the following result is due to Liu and to Atanasov, Kaplan, Krakoff, and Menzel.

\begin{theorem}[{\cite{Liu2007},\ \cite{AKKM2021}}]\label{th: fn(p,e) is polynomial for all e up to 8}
   For every $e\leq 8$, $f_n(p^e)$ is a polynomial in $n$ and $p$.
\end{theorem}
Mishra and Ray \cite[Theorem 1.4]{MishraRay2022} treated the case $e=9$ by counting irreducible subring matrices in Hermite normal form. Theorem~\ref{thm: f-index-nine} gives an independent calculation and corrects one coefficient in their formula.

A key ingredient in these results is Liu's notion of \emph{irreducible} subrings.
\begin{definition}\label{def: irreducibility}({\cite[\S3]{Liu2007}})
   A finite-index unital subring $\Lambda \leq \mathbb{Z}_p^n$ is called \emph{irreducible} if every element $(x_1,\dots,x_n)\in\Lambda$ satisfies $x_1\equiv \cdots \equiv x_n \pmod{p}$. We denote by $g_n(p^e)$ the number of irreducible unital subrings of $\mathbb{Z}_p^n$ of index $p^e$.
   \end{definition}
Equivalently, $\Lambda\leq\mb{Z}_p^n$ is irreducible if it is local. 
What we call $g_n(p^e)$, Liu calls $g_{n-1}(p^e)$.
We make this adjustment to be consistent with
\cite{AKKM2021,Isham2022,Isham2023}.
Every finite-index unital subring $\Lambda \leq \mathbb{Z}_p^n$ can be written as a direct product of irreducible unital subrings $\Lambda_i \leq \mathbb{Z}_p^{n_i}$ \cite[Theorem 3.4]{Liu2007}. This yields the following recursive formula 
\begin{proposition}[{\cite[Proposition 4.4]{Liu2007}; see also \cite[Proposition 1.7]{Isham2023}}]
   Set $f_0(p^e)=1$ if $e=0$ and $f_0(p^e)=0$ otherwise. Then
   \[
      f_n(p^e)=\sum_{i=0}^e \sum_{j=1}^n \binom{n-1}{j-1} f_{n-j}(p^{e-i})\, g_j(p^i).
   \]
\end{proposition}

Hence, computing $f_n(p^e)$ reduces to determining $g_j(p^i)$ for $1\leq j\leq n$ and $i\leq e$. One easily shows that $g_j(p^i)=0$ for $i<j-1$ and $g_j(p^{j-1})=1$, hence one only needs the values $g_j(p^i)$ for $j\leq \min\{n,e\}$ and $j\leq i\leq e$.
This observation also implies that Question \ref{question on the uniformity 2} has a positive answer  if the following question does:
\begin{question}\label{question on the uniformity 3}
   For $n\leq e$, is $g_n(p^e)$ a polynomial in $p$?
\end{question}

Liu \cite[Proposition 4.3]{Liu2007} gave a positive answer for $e=n$, Atanasov–Kaplan–Krakoff–Menzel \cite[Corollary 3.7]{AKKM2021} for $e=n+1$, and Isham \cite{Isham2023} studied $e=n+2$. We give a corrected formula for the last count in Theorem~\ref{thm: irreducible-excess-two-count}.

\subsubsection{\textbf{From $n+1$ to $n$}}
By \cite[Proposition 2.3]{Liu2007}, there is an index-preserving
correspondence between the lattice of finite-index unital subrings
of $\mb{Z}_p^{n+1}$ and the lattice of finite-index subrings of
$\mb{Z}_p^n$. Consequently,
\begin{align*}
 \zeta_{\mb{Z}_p^{n+1}}^{\scriptscriptstyle{1,<}}(s)
 &=\zeta_{\mb{Z}_p^n}^{\scriptscriptstyle{<}}(s)
 \qquad\text{for every prime }p,\\
 \zeta_{\mb{Z}^{n+1}}^{\scriptscriptstyle{1,<}}(s)
 &=\zeta_{\mb{Z}^n}^{\scriptscriptstyle{<}}(s).
\end{align*}
Moreover, under this correspondence, the irreducible unital
subrings of $\mb{Z}_p^{n+1}$ correspond to the subrings of
$\mb{Z}_p^n$ contained in $p\mb{Z}_p^n$.

\subsection{Results}

Let $\o$ be a compact discrete valuation ring with residue field of cardinality $q$, and consider the $\o$-algebra $\o^n$ with componentwise multiplication.
We study the Dirichlet series
\[
    \Xi_n(s)=\sum_{\Lambda\leq_s \o^n}[\o^n:\Lambda]^{-s},
\]
where the sum ranges over finite-index $\o$-subalgebras of $\o^n$
(i.e.\ $\o$-submodules closed under componentwise multiplication).
To do so, we closely follow an approach introduced by Voll \cite{Voll2010} to prove local functional equations for subring zeta functions.
Roughly speaking, we introduce multivariable zeta functions
\[
    \Xi_{n,I}\bigl(s_0,(s_\iota)_{\iota\in I}\bigr),
    \qquad \emptyset\neq I\subseteq [n-1]:=\{1,\ldots,n-1\},
\]
which keep track of the elementary-divisor types of subalgebras $\Lambda\leq_s \o^n$. We then express these series as finite weighted sums of $\o$-adic
integrals over suitable subsets of
$\o\times\o^I\times\on{L}_I(\o)$, where $\on{L}_I$ is a unipotent
group scheme of block lower unitriangular matrices determined by $I$.
The weights encode the combinatorics of partial flags over
$\mathbb F_q$ and arise from exploiting the symmetry of $\o^n$
under coordinate permutations; see
\S\ref{subsubsec: representations as integrals}.

In \S\ref{sec: lower bound for the abscissa}, we specialize our
integral representation to
$\Xi_{n,I}\bigl(s,(s)_{\iota\in I}\bigr)$ and obtain a lower bound
for its abscissa of convergence. Together with the decomposition
\[
 \Xi_n(s)=\frac{1}{1-q^{-ns}}
 +\sum_{\emptyset\neq I\subseteq[n-1]}
 \Xi_{n,I}\bigl(s,(s)_{\iota\in I}\bigr),
\]
this gives the lower bound
$\max_{1\leq\iota\leq n-1}\frac{\iota(n-\iota)}{n+\iota}$
for the abscissa of convergence of $\Xi_n(s)$.
For $\o=\mathbb Z_p$, this recovers Isham's lower bound
\cite[Theorem 1.10]{Isham2022} for the zeta function of unital
subalgebras of $\mathbb Z_p^{n+1}$.
More generally, the same bound holds for the abscissa of convergence
of the subalgebra zeta function of any $\o$-algebra that is free
of rank $n$ as an $\o$-module; see
Remark~\ref{rem: universal lower bound for the abscissa}.

In Section~\ref{sec: the case I of one element}, we obtain an
explicit formula for $\Xi_{n,\{\iota\}}(s_0,s_\iota)$ for arbitrary
$n\ge2$ and every $1\leq\iota\leq n-1$; see
Proposition~\ref{prop: formula for ZI}.
As an application, we derive the formulas
\begin{align*}
g_{n}(q^{n})=\binom{n-1}{1}_q\quad\text{and}\quad g_{n}(q^{n+1})=\binom{n}{2}q^{n-2}+\binom{n-1}{2}_q,
\end{align*}
where $g_n(q^e)$ denotes the number of irreducible unital subalgebras of $\o^n$ of index $q^e$, or equivalently, the number of subalgebras of $\o^{n-1}$ of index $q^e$ contained in $\pi\o^{n-1}$; 
see \S\ref{sec: Application II}.
For $\o=\mathbb Z_p$, this gives direct derivations of the formulas
for $g_n(p^n)$ and $g_n(p^{n+1})$.
We also compute $\Xi_2(s)$, which, for $\o=\mathbb Z_p$, yields
the local unital subring zeta function
$\zeta_{\mathbb Z_p^3}^{\scriptscriptstyle 1,<}(s)$.

In Section~\ref{sec: the case I of two elements}, we describe a strategy for computing $\Xi_{n,\{i_1,i_2\}}(s_0,s_{i_1},s_{i_2})$ for general $n$ and any $\{i_1,i_2\}\subseteq [n-1]$.
We carry out this computation explicitly in the case $n=3$ and $\{i_1,i_2\}=\{1,2\}$; see Proposition~\ref{prop: formula for Xi for I=12}.

Petrogradsky \cite{Petrogradsky2007} introduced the multiple zeta function
that counts finite-index subgroups of $\mathbb Z^d$ by the invariant factors
of their quotients. Chinta, Kaplan, and Koplewitz \cite{ChintaKaplanKoplewitz2023}
used his formulas to study the distribution of cotypes of sublattices.
Chinta defined the corresponding series for unital subrings of $\mathbb Z^n$
in the appendix to \cite[Definition~A.1]{ChintaIshamKaplan2024}.
Lee and Lee \cite{LeeLee2025Cotype} defined cotype zeta functions for
subalgebras of more general algebras.
The collection of multivariable series
$\Xi_{n,I}\bigl(s_0,(s_\iota)_{\iota\in I}\bigr)$ encodes the same
information as the cotype zeta function of $\o$-subalgebras of $\o^n$
(or, equivalently, of unital $\o$-subalgebras of $\o^{n+1}$).
As a consequence of our computations, we obtain a formula for the cotype zeta function of $\o$-subalgebras of $\o^3$ (equivalently, of unital $\o$-subalgebras of $\o^4$); see Proposition~\ref{prop: the cotype zeta function of o4}.
For $\o=\mathbb Z_p$, this is the formula conjectured in \cite[Conjecture A.2]{ChintaIshamKaplan2024}.

\medskip

In Sections~\ref{sec: finite-field systems}
and~\ref{sec: applications to irreducible subalgebras},
we establish several results concerning the uniformity problem
for the subalgebra zeta function of $\o^n$. In
Section~\ref{sec: finite-field systems} we study two systems of equations
over $\mathbb F_q$. We classify the solutions of the one-parameter system
and use the geometry of the projective plane to count the solutions of
the two-parameter system.
In Section~\ref{sec: applications to irreducible subalgebras} we apply these
counts to the integrals of \S\ref{subsubsec: representations as integrals}
to obtain the following two formulas for $g_n(q^e)$. Write $[r]_q=1+q+\cdots+q^{r-1}$
for $r\ge1$ and $[0]_q=0$. We take both ordinary and Gaussian binomial
coefficients to be zero when the lower index exceeds the upper index,
and omit summands multiplied by a zero binomial coefficient.
\begin{theorem}\label{thm: irreducible-excess-two-count}
For every $n\ge2$,
\[
 g_n(q^{n+2})=\binom{n-1}{3}_q+q^{n-3}\Biggl(
   \binom n2\bigl([n-2]_q+q\bigr)+(q-2)\binom n3
   +(q^2-2q+3)\binom n4\Biggr).
\]
\end{theorem}

For $\o=\mathbb Z_p$, Theorem~\ref{thm: irreducible-excess-two-count}
gives a corrected expression for the count in \cite{Isham2023}.
For example, the expression there gives $66$ for $g_4(2^6)$, whereas
the theorem above gives $67$, in agreement with \cite[Table~1]{AKKM2021}.

\begin{theorem}\label{thm: irreducible-excess-three-final}
For every $n\ge2$,
\begin{multline*}
 g_n(q^{n+3})=\binom{n-1}{4}_q
  +q^{2n-6}\left(\binom n3+3\binom n4\right)
  +q^{n-4}\Biggl(
   \binom n2\left(\binom{n-2}{2}_q+q[n-2]_q+q^2\right)\\
  {}+\binom n3\bigl((q-2)[n-3]_q+q(q^2+q-3)\bigr)
  +\binom n4\Bigl((q^2-2q+3)[n-3]_q+q\bigl(6(q-1)^2-1\bigr)\Bigr)\\
  {}+(q+1)(q-2)(q^2-2q+2)\binom n5
  +(q+1)(q^4-5q^3+11q^2-10q+5)\binom n6\Biggr).
\end{multline*}
In particular, for each fixed $n\ge2$, $g_n(q^{n+3})$ is a polynomial in $q$.
\end{theorem}

Finally, we extend Liu's decomposition into irreducible factors to
compact discrete valuation rings and use it to obtain an explicit
formula for $f_n(q^9)$, the number of unital subalgebras of $\o^n$ of index $q^9$ (Theorem~\ref{thm: f-index-nine}). Our calculation
uses the elementary-divisor types of subalgebras. For
$\o=\mathbb Z_p$, it gives a correction to \cite[Theorem~1.4]{MishraRay2022}.
The coefficient of $\binom n7$ in their formula is to be increased by
$15p^4+21p^3+6p^2+p+1$. The other coefficients agree.
This is precisely the change obtained by replacing Isham's value of
$g_7(p^9)$ by the value in Theorem~\ref{thm: irreducible-excess-two-count}.


\subsection*{Notation} 
\begin{center}
\begin{tabular}{@{}ll}
    $\on{Mat}_{m\times n}(R)$ & The set of $m\times n$ matrices with coefficients in a ring $R$. \\
    $\on{Mat}_n(R)$& $\on{Mat}_{n\times n}(R)$.\\
    $\on{Id}_n$& The $n\times n$-identity matrix.\\
    $[n]$& The integers from $1$ to $n$.\\
    $\{i_1,\ldots,i_l\}_<$& A set of positive integers with $i_1<i_2<\cdots< i_l$.\\
    $\binom{n}{r}_X$& The Gaussian polynomial $\prod_{i=0}^{r-1} \frac{(1-X^{n-i})}{ (1-X^{r-i})}$.\\
    $\# S$& The cardinality of a set $S$.
\end{tabular}    
\end{center}
Throughout, $\o$ denotes a compact discrete valuation ring, $\m$ its
maximal ideal, and $\pi$ a uniformizing parameter, i.e.\ an element such that
$\m=\pi\o$. We write $q=\#(\o/\m)$. For $x$ in the fraction field of $\o$, we define
$|x|=q^{-v(x)}$, where $v$ denotes the $\m$-adic valuation.

\section{Expressing local cotype zeta functions as $\o$-adic integrals}

Let $n\in \mathbb{N}$. We consider the $\o$-module $\o^n$, endowed with
componentwise multiplication. For a submodule $\Lambda\leq \o^n$, we write
$\Lambda\leq_s \o^n$ to indicate that $\Lambda$ is a finite-index submodule which is
closed under multiplication.

In this section, following the approach of \cite{Voll2010}, we derive an expression
for the generating series
\[
    \Xi_n(s):=\sum_{\Lambda\leq_s \o^n}[\o^n:\Lambda]^{-s}
\]
as a sum of integrals reflecting the elementary-divisor types of the submodules.

In \S\ref{sec: parametrization of submodules}, we use lower-triangular matrices
to parametrize submodules of $\o^n$ according to their elementary-divisor types.
In \S\ref{sec: preliminaries on integration}, we review some basic facts about
integration, and in \S\ref{subsubsec: representations as integrals}, we present the
integral expression for the above series. In \S\ref{sec: lower bound for the abscissa} we obtain a lower bound for the abscissa of convergence of $\Xi_n(s)$.

In this section, we use the following special notation. For a matrix $A\in \on{Mat}_{m\times n}(R)$, where $R$ is a commutative
ring, we denote by $\langle A\rangle$ the $R$-submodule generated by the row
vectors of $A$.

\subsection{Parameterization of finite-index submodules}
\label{sec: parametrization of submodules}

Any submodule $\Lambda\le \o^n$ is of the form $\Lambda=\langle A\rangle$ for some
$A\in \operatorname{Mat}_n(\o)$.
The submodule $\Lambda$ has finite index if and only if $A$ is nonsingular, i.e.\
$\det(A)\neq 0$.
If $A,B\in \operatorname{Mat}_n(\o)$ are nonsingular, then $\langle A\rangle=\langle B\rangle$
if and only if $A=UB$ for some $U\in \GL_n(\o)$.

\medskip

Let $\Lambda\le \o^n$ be a finite-index submodule.
By the elementary divisor theorem, either $\Lambda=\pi^{r_0}\o^n$ for some
$r_0\in \mathbb{N}_0$, or there exist a non-empty subset
$I=\{i_1,\ldots,i_\ell\}_<\subseteq [n-1]$ and
$(r_0,\mathbf{r})\in \mathbb{N}_0\times \mathbb{N}^I$ such that
\[
    \Lambda=\bigl\langle \pi^{r_0} D(I,\mathbf{r}) A \bigr\rangle
    \qquad \text{for some } A\in \GL_n(\o),
\]
where
\[
    D(I,\mathbf{r})
    =\operatorname{diag}\!\bigl(
    \underbrace{\underbrace{\underbrace{\pi^{\sum_{j=1}^\ell r_{i_j}},\ldots,\pi^{\sum_{j=1}^\ell r_{i_j}}}_{i_1},
    \pi^{\sum_{j=2}^\ell r_{i_j}},\ldots,\pi^{\sum_{j=2}^\ell r_{i_j}}}_{i_2},
    \ldots,
    \pi^{r_{i_\ell}},\ldots,\pi^{r_{i_\ell}}}_{i_\ell},
    1,\ldots,1
    \bigr).
\]
The index of $\Lambda$ is given by
\begin{equation*}
    [\o^n:\Lambda]=q^{nr_0+\sum_{\iota\in I} r_\iota\,\iota}.
\end{equation*}
We write $\nu_0(\Lambda):=(I,r_0,\mathbf{r})$.

\medskip

Fix $\emptyset\neq I\subseteq [n-1]$ and $(r_0,\mathbf{r})\in \mathbb{N}_0\times\mathbb{N}^I$.
Our goal is to parametrize the set
\[
    \{\Lambda\le \o^n:\ \nu_0(\Lambda)=(I,r_0,\mathbf{r})\}
\]
by suitable lower-triangular matrices.

Consider the surjective map
\[
   \GL_n(\o)\to \{\Lambda\le \o^n:\ \nu_0(\Lambda)=(I,r_0,\mathbf{r})\},
   \qquad
   A\mapsto \bigl\langle \pi^{r_0}D(I,\mathbf{r})A\bigr\rangle.
\]
Given $A,B\in\GL_n(\o)$, we have
\[
\langle \pi^{r_0}D(I,\mathbf{r})A\rangle
=
\langle \pi^{r_0}D(I,\mathbf{r})B\rangle
\]
if and only if $A=UB$ for some
\[
    U\in \GL_n(\o)\cap D(I,\mathbf{r})^{-1}\GL_n(\o)D(I,\mathbf{r}).
\]
It is straightforward to verify that this intersection is the subgroup
\[
    \Gamma_{I,\mathbf{r}}
    :=
    \left\{
    \begin{pmatrix}
        U_{i_1} & * & \cdots & * & * \\
        \pi^{r_{i_1}}* & U_{i_2-i_1} & \ddots & \vdots & * \\
        \pi^{r_{i_1}+r_{i_2}}* & \pi^{r_{i_2}}* & \ddots & * & * \\
        \vdots & \vdots & \ddots & U_{i_\ell-i_{\ell-1}} & * \\
        \pi^{r_{i_1}+\cdots+r_{i_\ell}}* &
        \pi^{r_{i_2}+\cdots+r_{i_\ell}}* & \cdots &
        \pi^{r_{i_\ell}}* & U_{n-i_\ell}
    \end{pmatrix}
    \right\},
\]
where $U_s$ denotes an element of $\GL_s(\o)$ and $*$ denotes a block matrix of
appropriate size with entries in $\o$. Thus:

\begin{proposition}\label{prop: first parametrization of submodules of type (I,r)}
The map $\Gamma_{I,\mathbf{r}}A\mapsto \langle \pi^{r_0}D(I,\mathbf{r})A\rangle$
induces a bijection
\[
   \Gamma_{I,\mathbf{r}}\backslash\GL_n(\o)
   \cong
   \{\Lambda\le \o^n:\ \nu_0(\Lambda)=(I,r_0,\mathbf{r})\}.
\]
\end{proposition}

\medskip

We now seek a more explicit description of
$\Gamma_{I,\mathbf{r}}\backslash\GL_n(\o)$.
There are two distinguished subgroups of $\Gamma_{I,\mathbf{r}}$.
The first is
\[
    \operatorname{B}_I(\o)=
    \left\{
    \begin{pmatrix}
        U_{i_1} & * & \cdots & * & *\\
        0 & U_{i_2-i_1} & \ddots & \vdots & *\\
        0 & 0 & \ddots & * & *\\
        \vdots & \vdots & \ddots & U_{i_\ell-i_{\ell-1}} & *\\
        0 & 0 & \cdots & 0 & U_{n-i_\ell}
    \end{pmatrix}
    \right\},
\]
and the second is
\[
    \operatorname{L}_{I,\mathbf{r}}
    :=
    \left\{
    \begin{pmatrix}
        \operatorname{Id}_{i_1} & 0 & \cdots & 0 & 0\\
        \pi^{r_{i_1}}* & \operatorname{Id}_{i_2-i_1} & \ddots & \vdots & 0\\
        \pi^{r_{i_1}+r_{i_2}}* & \pi^{r_{i_2}}* & \ddots & 0 & 0\\
        \vdots & \vdots & \ddots & \operatorname{Id}_{i_\ell-i_{\ell-1}} & 0\\
        \pi^{r_{i_1}+\cdots+r_{i_\ell}}* &
        \pi^{r_{i_2}+\cdots+r_{i_\ell}}* &
        \cdots &
        \pi^{r_{i_\ell}}* & \operatorname{Id}_{n-i_\ell}
    \end{pmatrix}
    \right\}.
\]
Every element $U\in\Gamma_{I,\mathbf{r}}$ admits a unique factorization
$U=U'U''$ with $U'\in \operatorname{B}_I(\o)$ and
$U''\in \operatorname{L}_{I,\mathbf{r}}$.
Thus, it is natural to begin by understanding
$\operatorname{B}_I(\o)\backslash\GL_n(\o)$.

\medskip

\medskip

We next describe $\operatorname{B}_I(R)\backslash \GL_n(R)$ for an arbitrary local ring $R$.
Recall that $I=\{i_1,\ldots,i_\ell\}_<$ is an ordered subset of $[n-1]$.
An \emph{$I$-flag in $R^n$} is a chain of $R$-submodules
\[
    R^n=\Lambda_0\supset \Lambda_{i_1}\supset\cdots\supset \Lambda_{i_\ell}\supset \Lambda_n=0
\]
such that $R^n/\Lambda_\iota$ is free of rank $\iota$ for all $\iota\in I$.
In particular, each successive quotient $\Lambda_{i_{j-1}}/\Lambda_{i_j}$ is free.

The natural action of $\GL_n(R)$ on $R^n$ by right multiplication induces a transitive action on the set of $I$-flags.
Let $e_1,\ldots,e_n$ be the standard basis of $R^n$, and consider the canonical $I$-flag
\[
    R^n=\Delta_0\supset \Delta_{i_1}\supset\cdots\supset \Delta_{i_\ell}\supset \Delta_n=0,
\]
where $\Delta_\iota=\langle e_j: j>\iota\rangle$ for $\iota\in I$.
If $A\in \GL_n(R)$ has row vectors $A_1,\ldots,A_n$, then the image of the canonical flag under $A$ is the flag
\[
    \langle A\rangle_I \,:\,
    R^n\supset \langle A_{i_1+1},\ldots,A_n\rangle\supset\cdots\supset
    \langle A_{i_\ell+1},\ldots,A_n\rangle\supset 0.
\]
Moreover, the subgroup $\operatorname{B}_I(R)\subseteq \GL_n(R)$ is precisely the stabilizer of the canonical $I$-flag.
Hence $\operatorname{B}_I(R)\backslash \GL_n(R)$ is naturally identified with the set of $I$-flags in $R^n$.
We use this description to show that a convenient transversal for the left cosets of $\operatorname{B}_I(R)$ in $\GL_n(R)$
is given by the following class of matrices.

\begin{definition}\label{def:I-reduced}
A matrix $A \in \GL_n(R)$ is called \emph{$I$-reduced} if it satisfies the following conditions:
\begin{enumerate}
    \item In each row, the leftmost invertible entry is $1$. We call these entries the \emph{leading $1$'s}.
    \item All entries above a leading $1$ are zero.
    \item Setting $i_0=0$ and $i_{\ell+1}=n$, for each $j=0,\ldots,\ell$,
    the leading $1$'s in rows $i_j+1,\ldots,i_{j+1}$ occur in strictly increasing column positions,
    and the submatrix formed by these rows and their leading columns is the identity.
\end{enumerate}
We denote the set of $I$-reduced matrices by $\mathcal{R}_I(R)$.
\end{definition}

\begin{example}\label{ex: examples of reduced matrices}
The following are examples of $I$-reduced matrices for $n=6$ and $I=\{2,4\}$.
The leading $1$'s are displayed in bold.
\begin{align*}
    \text{(a)}\quad & R=\mathbb{F}_p,\quad
    A=\begin{pmatrix}
    0&0&0&\mathbf{1}&0&0\\
    0&0&0&0&\mathbf 1&0\\
    \mathbf 1&0&0&1&-1&0 \\
    0&0&0&0&0&\mathbf 1\\
    0&\mathbf 1&0&3&1&1\\
    0&0&\mathbf 1&1&1&0
    \end{pmatrix};\qquad
    \text{(b)}\quad  R=\mathbb{Z}_p,\quad
    A=\begin{pmatrix}
    0&0&0&\mathbf{1}&0&0\\
    0&0&0&0&\mathbf 1&0\\
    \mathbf 1&0&0&1&p&0 \\
    0&0&0&-p&p&\mathbf 1\\
    p&\mathbf 1&0&3&1&1\\
    p^2&0&\mathbf 1&1&1&p
    \end{pmatrix}.
\end{align*}
\end{example}

\begin{remark}
Reduction modulo the maximal ideal of $R$ sends an $I$-reduced matrix over $R$
to an $I$-reduced matrix over the residue field.
\end{remark}

\begin{proposition}\label{prop: I-reduced matrices and flags}
The map
\[
A \longmapsto \operatorname{B}_I(R)A
\]
from $\mathcal{R}_I(R)$ to $\operatorname{B}_I(R)\backslash \GL_n(R)$ is a bijection.
\end{proposition}

\begin{proof}
Equivalently, every $I$-flag in $R^n$ is of the form $\langle A\rangle_I$ for a unique
$A\in \mathcal{R}_I(R)$.
We prove this by induction on $n$.

If $n=1$, then $I=\emptyset$ and the claim is immediate.
Assume $n>1$. If $I=\emptyset$, the claim is again immediate, so we assume $I\neq\emptyset$ and write
$I=\{i_1<\cdots<i_\ell\}\subseteq [n-1]$.

Let
\[
R^n=\Lambda_0\supset \Lambda_{i_1}\supset\cdots\supset \Lambda_{i_\ell}\supset \Lambda_n=0
\]
be an $I$-flag. We construct a matrix $A\in \mathcal{R}_I(R)$ such that
$\langle A\rangle_I$ equals this flag.

Choose a basis $\{v_{i_\ell+1},\ldots,v_n\}\subset R^n$ for $\Lambda_{i_\ell}$ and apply Gaussian
elimination to the matrix with row vectors $v_{i_\ell+1},\ldots,v_n$, clearing the other entries in each pivot column.
Over a local ring, the only difference from the field case is that one must pivot on
\emph{invertible} entries rather than merely nonzero ones.
The resulting matrix has the property that the leftmost invertible entry in each row is $1$,
the column positions of these leading $1$'s increase from top to bottom, and the submatrix formed by
these columns is the identity.

Let $A_{i_\ell+1},\ldots,A_n$ be the row vectors of this matrix.
Extend them to a basis
\[
\{v_1,\ldots,v_{i_\ell},A_{i_\ell+1},\ldots,A_n\}
\]
of $R^n$ such that, for each $\iota\in I\setminus\{i_\ell\}$, the subset
$\{v_{\iota+1},\ldots,v_{i_\ell},A_{i_\ell+1},\ldots,A_n\}$ is a basis of $\Lambda_\iota$.
We may further assume that each $v_j$ has zeros in the coordinates corresponding to the leading
$1$'s of $A_{i_\ell+1},\ldots,A_n$.
Deleting these coordinates yields vectors
$v_1',\ldots,v_{i_\ell}'\in R^{i_\ell}$.

Let $I':=I\setminus\{i_\ell\}$.
By the induction hypothesis, there exists an $I'$-reduced matrix
$A'\in \GL_{i_\ell}(R)$ whose associated $I'$-flag coincides with the flag generated by
$v_1',\ldots,v_{i_\ell}'$.
Let $A_1',\ldots,A_{i_\ell}'$ be the rows of $A'$, and lift them to vectors
$A_1,\ldots,A_{i_\ell}\in R^n$ by reinserting zeros in the deleted coordinates.
Then the matrix $A$ with rows $A_1,\ldots,A_n$ is $I$-reduced and satisfies
$\langle A\rangle_I=\{\Lambda_\iota\}$.

We now prove uniqueness.
Suppose $A,B\in \mathcal{R}_I(R)$ satisfy $\langle A\rangle_I=\langle B\rangle_I$.
Let $A''$ (resp.\ $B''$) denote the submatrix consisting of the last $n-i_\ell$ rows of $A$ (resp.\ $B$).
Then $\langle A''\rangle=\langle B''\rangle$, so $A''=UB''$ for some
$U\in \GL_{n-i_\ell}(R)$.

Reducing modulo the maximal ideal of $R$, the matrices $\overline{A''}$ and $\overline{B''}$
are row-equivalent matrices in reduced row echelon form over the residue field, and hence coincide.
In particular, the leading $1$'s in $A''$ and $B''$ occur in the same column positions.
Restricting the equality $A''=UB''$ to these pivot columns yields $\operatorname{Id}=U\operatorname{Id}$,
and therefore $U=\operatorname{Id}$ and $A''=B''$.

Finally, let $A'$ (resp.\ $B'$) be obtained from $A$ (resp.\ $B$) by deleting the last $n-i_\ell$ rows
and the columns containing the leading $1$'s of those rows.
Then $A'$ and $B'$ are $I'$-reduced matrices in $\GL_{i_\ell}(R)$ and satisfy
$\langle A'\rangle_{I'}=\langle B'\rangle_{I'}$.
By induction, $A'=B'$, and hence $A=B$.
\end{proof}

\begin{remark}\label{rem: number of I-reduced matrices}
The proof yields a bijection between $\mathcal{R}_I(R)$ and the set of $I$-flags in $R^n$.
In particular, for $R=\mathbb{F}_q$ we obtain
\[
\#\mathcal{R}_I(\mathbb{F}_q)=\binom{n}{I}_q
:=\prod_{j=1}^{\ell}\binom{i_{j+1}}{i_j}_q,
\qquad i_{\ell+1}=n.
\]
\end{remark}

\medskip

We now return to our original problem.
Let $I\subseteq [n-1]$ and $\mathbf{r}\in \mathbb{N}^I$. Recall that we defined a subgroup
$\Gamma_{I,\mathbf{r}}\subseteq \GL_n(\o)$ admitting a decomposition
\[
\Gamma_{I,\mathbf{r}}=\operatorname{B}_I(\o)\, \operatorname{L}_{I,\mathbf{r}}.
\]
It is straightforward to verify that left multiplication of an $I$-reduced matrix by an element
of $\operatorname{L}_{I,\mathbf{r}}$ preserves the leading $1$'s and produces another $I$-reduced matrix.
Consequently, the quotient
$\operatorname{L}_{I,\mathbf{r}}\backslash \mathcal{R}_I(\o)$
is well defined.

\begin{proposition}\label{prop: parametrization of submodules of type (I,r)}
The map
\[
\operatorname{L}_{I,\mathbf{r}}\backslash \mathcal{R}_I(\o)
\longrightarrow
\{\Lambda\le \o^n:\ \nu_0(\Lambda)=(I,r_0,\mathbf{r})\},
\qquad
\operatorname{L}_{I,\mathbf{r}}A \longmapsto
\langle \pi^{r_0}D(I,\mathbf{r})A\rangle,
\]
is a bijection.
\end{proposition}

\begin{proof}
By Propositions~\ref{prop: I-reduced matrices and flags} and
\ref{prop: first parametrization of submodules of type (I,r)}, we obtain a surjective map
\[
\resizebox{\linewidth}{!}{$\displaystyle
\begin{aligned}
\mathcal{R}_I(\o)\cong \operatorname{B}_I(\o)\backslash \GL_n(\o)
\longrightarrow
\Gamma_{I,\mathbf{r}}\backslash \GL_n(\o)
\cong
\{\Lambda\le \o^n:\ \nu_0(\Lambda)=(I,r_0,\mathbf{r})\},
\qquad
A\longmapsto \langle \pi^{r_0}D(I,\mathbf{r})A\rangle.
\end{aligned}
$}
\]
If $A,B\in \mathcal{R}_I(\o)$ have the same image, then $A=UB$ for some
$U\in \Gamma_{I,\mathbf{r}}$.
Writing $U=U'U''$ with
$U'\in \operatorname{B}_I(\o)$ and $U''\in \operatorname{L}_{I,\mathbf{r}}$,
we see that $U''B\in \mathcal{R}_I(\o)$ and
$A=U'(U''B)\in \mathcal{R}_I(\o)$.
By Proposition~\ref{prop: I-reduced matrices and flags}, this forces
$U'=\operatorname{Id}$, and hence $A=U''B$.
Thus $\operatorname{L}_{I,\mathbf{r}}A=\operatorname{L}_{I,\mathbf{r}}B$, as claimed.
\end{proof}

\medskip

Finally, consider the algebraic group $\operatorname{L}_I$ consisting of matrices of the form
\[
\begin{pmatrix}
\operatorname{Id}_{i_1} & 0 & \cdots & 0 & 0 \\
* & \operatorname{Id}_{i_2-i_1} & \ddots & \vdots & 0 \\
\vdots & \ddots & \ddots & 0 & 0 \\
* & \cdots & * & \operatorname{Id}_{i_\ell-i_{\ell-1}} & 0 \\
* & \cdots & * & * & \operatorname{Id}_{n-i_\ell}
\end{pmatrix}.
\]

Let $A\in \mathcal{R}_I(\o)$ (or $A\in \mathcal{R}_I(\o/\m)$).
Then $A$ can be written uniquely as $A=B\sigma$ with
$B\in \operatorname{L}_I(\o)$ (respectively $B\in \operatorname{L}_I(\o/\m)$) and $\sigma\in \mathbb{S}_n$,
where $B\sigma$ denotes the matrix obtained from $B$ by permuting its columns according to $\sigma$.
Conversely, given $B\in \operatorname{L}_I(\o)$ and $\sigma\in \mathbb{S}_n$, the matrix $B\sigma$ is
$I$-reduced if and only if $\overline{B}\sigma$ is $I$-reduced, where
$\overline{(\cdot)}$ denotes reduction modulo $\m$.

\begin{definition}
For $E\in \operatorname{L}_I(\o/\m)$ we define
\[
\mathcal{T}_{I,E}
:=
\{\sigma\in \mathbb{S}_n:\ E\sigma\in \mathcal{R}_I(\o/\m)\}.
\]
\end{definition}

Combining these observations with
Proposition~\ref{prop: parametrization of submodules of type (I,r)}, we obtain:

\begin{proposition}\label{prop: final parametrization of submodules of type (I,r)}
For each $E\in \operatorname{L}_I(\o/\m)$ and each $\sigma\in \mathcal{T}_{I,E}$, the map
\[
\operatorname{L}_{I,\mathbf{r}}\backslash \{B\in \operatorname{L}_I(\o):\ \overline{B}=E\}
\longrightarrow
\{\Lambda\le \o^n:\ \nu_0(\Lambda)=(I,r_0,\mathbf{r})\},
\qquad
\operatorname{L}_{I,\mathbf{r}}B \longmapsto
\langle \pi^{r_0}D(I,\mathbf{r})B\sigma\rangle,
\]
is injective.
Moreover, as $E$ varies in $\operatorname{L}_I(\o/\m)$ and $\sigma$ varies in $\mathcal{T}_{I,E}$,
the images of these maps are disjoint and together form a partition of
$\{\Lambda\le \o^n:\ \nu_0(\Lambda)=(I,r_0,\mathbf{r})\}$.
\end{proposition}

\subsection{Representations in terms of $\o$-adic integrals}

\subsubsection{\textbf{Preliminaries on $\o$-adic integration}}
\label{sec: preliminaries on integration}

Every compact topological group admits a (left) Haar measure.
In this subsection, we briefly recall the basic facts about $\pi$-adic integration
that will be used throughout the paper.
We will mainly be concerned with two classes of examples.

\medskip

The first example is the additive group $\o^N$ for some $N\in\mathbb{N}$.
In this case, the normalized Haar measure is the unique Borel measure $\mu$ on $\o^N$
satisfying $\mu(\o^N)=1$ and $\mu(x+\mathcal{A})=\mu(\mathcal{A})$ for every Borel set
$\mathcal{A}\subseteq \o^N$ and every $x\in \o^N$.
For $N=1$, this definition implies that
\[
    \mu(a\o)=|a|
    \quad\text{and}\quad
    \mu\bigl(\{x\in \o:\ |x|=q^{-k}\}\bigr)=(1-q^{-1})q^{-k}
    \qquad (k\ge 0).
\]
The normalized Haar measure on $\o^N$ is the product measure associated with the
normalized Haar measure on $\o$.
When integrating a complex-valued function $f(x_1,\ldots,x_N)$ over an open subset
$U\subseteq \o^N$ with respect to this measure, we write
\[
    \int\limits_U f(x_1,\ldots,x_N)\,|dx_1\cdots dx_N|
    \qquad\text{or simply}\qquad
    \int\limits_U f(\x)\,|d\x|.
\]

We will use Haar measure to reinterpret power series in $q^{-s}$ (or in several variables)
as integrals.

\begin{example}\label{ex: from power series to integral}
The geometric series $\sum_{k=0}^\infty q^{-ks}$ can be expressed as an integral over $\o$
as follows.
For $k\ge 0$ we have
\begin{align*}
  q^{-ks}
  &= \frac{q^{-ks}}{(1-q^{-1})q^{-k}}
     \int\limits_{|x|=q^{-k}} |dx| \\
  &= \frac{1}{1-q^{-1}}
     \int\limits_{|x|=q^{-k}} |x|^{s-1}\,|dx|.
\end{align*}
Summing over all $k\ge 0$ yields
\[
   \sum_{k=0}^\infty q^{-ks}
   =\frac{1}{1-q^{-1}}\int\limits_{\o} |x|^{s-1}\,|dx|.
\]
\end{example}

More generally, for power series of the form $\sum_{k=0}^\infty a_k q^{-ks}$,
the coefficients $a_k$ can often be interpreted as measures of suitable subsets of $\o$
or of $\o^N$ for some $N>1$.
Integral representations of this kind are particularly useful, as they allow one to
partition domains of integration and apply the change-of-variables formula.
Specifically, if $f_1,\ldots,f_N$ are analytic functions defining a bijection between open subsets
$\mathcal{A},\mathcal{B}\subseteq \o^N$, then for any continuous function
$h:\mathcal{B}\to \mathbb{C}$ one has
\[
    \int\limits_{\mathcal{B}} h(y_1,\ldots,y_N)\,|dy_1\cdots dy_N|
    =
    \int\limits_{\mathcal{A}} h\bigl(f_1(\x),\ldots,f_N(\x)\bigr)\,
    \bigl|\det J(f_1,\ldots,f_N)(\x)\bigr|\,
    |dx_1\cdots dx_N|.
\]

\medskip

Our second example of a compact topological group is the multiplicative group
$\operatorname{L}_I(\o)$ associated with a subset $I\subseteq [n-1]$.
Its normalized Haar measure is the unique Borel measure $\mu$ satisfying
$\mu\bigl(\operatorname{L}_I(\o)\bigr)=1$ and $\mu(\mathcal{A}B)=\mu(\mathcal{A})$ for every Borel set
$\mathcal{A}\subseteq \operatorname{L}_I(\o)$ and every $B\in \operatorname{L}_I(\o)$.
In particular, if $\Gamma\subseteq \operatorname{L}_I(\o)$ is an open subgroup and
$\mathcal{A}\subseteq \operatorname{L}_I(\o)$ satisfies $\Gamma\mathcal{A}=\mathcal{A}$, then
\begin{equation}\label{eq: cardinal of a quotient}
    \#(\Gamma\backslash \mathcal{A})=\frac{\mu(\mathcal{A})}{\mu(\Gamma)}.
\end{equation}

\medskip

For our applications, we will need to compute the measure of the subgroup
$\Gamma=\operatorname{L}_{I,\mathbf{r}}$, defined in
\S\ref{sec: parametrization of submodules}.
This computation is simplified by the following observation.
The group $\operatorname{L}_I(\o)$ can be identified, as a topological space,
with $\o^N$, where
\begin{equation*}
    N=i_1(i_2-i_1)+i_2(i_3-i_2)+\cdots+i_{\ell-1}(i_\ell-i_{\ell-1})
      +i_\ell(n-i_\ell).
\end{equation*}
Let $\mu'$ denote the measure on $\operatorname{L}_I(\o)$ induced by the Haar measure on $\o^N$
under this identification.
We claim that $\mu'=\mu$.

Indeed, $\mu'\bigl(\operatorname{L}_I(\o)\bigr)=1$ since $\o^N$ has measure $1$.
Fix $B\in \operatorname{L}_I(\o)$.
Right multiplication by $B$ corresponds to an affine transformation
$\varphi:\o^N\to \o^N$ of the form $\varphi(\x)=\lambda+T\x$,
with $\lambda\in \o^N$ and $T\in \operatorname{Mat}_N(\o)$.
Since $\varphi$ is bijective, we have $T\in \GL_N(\o)$, and hence $\det(T)$ is a unit in $\o$.
Therefore $|\det(T)|=1$, and it follows that
\[
   \mu'(\mathcal{A}B)
   =\mu'\bigl(\varphi(\mathcal{A})\bigr)
   =|\det(T)|\,\mu'(\mathcal{A})
   =\mu'(\mathcal{A}).
\]
By uniqueness of Haar measure, this implies $\mu'=\mu$.

Using this identification, one readily obtains
\begin{equation}\label{eq: measure of L(I,r)}
    \mu\bigl(\operatorname{L}_{I,\mathbf{r}}\bigr)
    =q^{-\sum_{\iota\in I} r_\iota\,\iota(n-\iota)}.
\end{equation}

\subsubsection{\textbf{Representations in terms of $\o$-adic integrals}}\label{subsubsec: representations as integrals}

We continue to work with the notation and assumptions introduced in the previous sections.
Recall that we write $\Lambda\leq_s \o^n$ to indicate that $\Lambda$ is a finite-index
submodule of $\o^n$ which is closed under componentwise multiplication.

Our main objective is to compute the generating series
\begin{equation*}
   \Xi_n(s)
   :=
   \sum_{\Lambda\leq_s \o^n}[\o^n:\Lambda]^{-s}.
\end{equation*}
For $n=1$, this series is elementary:
\[
\Xi_1(s)
=
\sum_{r_0=0}^\infty q^{-r_0s}
=
\frac{1}{1-q^{-s}}.
\]
Henceforth we assume that $n\geq 2$.

Using the parametrization of finite-index submodules described in
\S\ref{sec: parametrization of submodules}, we may write
\begin{align*}
\Xi_n(s)
&=
\sum_{r_0=0}^\infty
[\o^n:\pi^{r_0}\o^n]^{-s}
+
\sum_{\emptyset\neq I\subseteq [n-1]}
\sum_{(r_0,\mathbf{r})\in \mathbb{N}_0\times \mathbb{N}^I}
q^{-snr_0-\sum_{\iota\in I}s r_\iota\,\iota}\,
h_n(I,r_0,\mathbf{r}),
\end{align*}
where
\[
h_n(I,r_0,\mathbf{r})
:=
\#\bigl\{
\Lambda\leq_s \o^n:\ \nu_0(\Lambda)=(I,r_0,\mathbf{r})
\bigr\}.
\]

In order to keep track of the elementary-divisor types of the subalgebras,
it is convenient to refine this construction.
For each non-empty subset $I\subseteq [n-1]$, we therefore introduce the
multivariable generating series
\begin{align*}
\Xi_{n,I}\bigl(s_0,(s_\iota)_{\iota\in I}\bigr)
&:=
\sum_{(r_0,\mathbf{r})\in \mathbb{N}_0\times \mathbb{N}^I}
q^{-s_0 n r_0-\sum_{\iota\in I} r_\iota\,\iota\, s_\iota}\,
h_n(I,r_0,\mathbf{r}).
\end{align*}
Note that
\begin{equation}\label{eq:decomposition_Xi_n}
\Xi_n(s)
=
\frac{1}{1-q^{-ns}}
+
\sum_{\emptyset\neq I\subseteq [n-1]}
\Xi_{n,I}\bigl(s,(s)_{\iota\in I}\bigr).
\end{equation}

\medskip

To express $\Xi_{n,I}\bigl(s_0,(s_\iota)_{\iota\in I}\bigr)$ in terms of $\o$-adic integrals,
we first obtain integral expressions for its coefficients $h_n(I,r_0,\mathbf{r})$.
We begin with some definitions.
\begin{definition}
    A pair $(J,K)$ of subsets of $I$ is called {\em $I$-admissible} if $J\neq \emptyset$ and there exist integers $1\leq j<i\leq i'\leq n$ such that 
    \begin{align*}
    J=I_{j}^i:=\{\iota\in I: j\leq \iota<i\}\quad \text{and}\quad K=I_{i'}:=\{\iota\in I: i'\leq \iota\}.
\end{align*}
If $(J,K)$ is an $I$-admissible pair and $B\in \on{L}_I(\o)$, we define
\begin{align*}
    R_{J,K}(B)=\left\{ \sum_{k\le i} B_{ik}B_{i'k}(B^{-1})_{kj}\,:\, 1\leq j\leq i\leq i'\leq n,\ I_{j}^i=J,\ I_{i'}=K\right\}
\end{align*}

\end{definition}

\begin{example}\label{ex: I-admissible pairs for I of one element}
    If $I=\{\iota\}$ is a single-element set, then the only $I$-admissible pair is $(\{\iota\},\emptyset)$. For a matrix
\[
B=
\begin{pmatrix}
\operatorname{Id}_{\iota} & 0 \\
(b_{ij}) & \operatorname{Id}_{n-\iota}
\end{pmatrix}\in \on{L}_I(\o),
\]
one easily obtains
\begin{align*}
R_{\{\iota\},\emptyset}(B)
&=
\{b_{ij}b_{i'j}\mid 1\le j\le \iota <i<i'\le n\}
\cup
\{b_{ij}^2-b_{ij}\mid 1\le j\le \iota<i\le n\}.
\end{align*}
\end{example}
\begin{example}\label{ex: I-admissible pairs for I of two elements}
    If $I=\{i_1,i_2\}$ then the only $I$-admissible pairs are $(\{i_1\},\{i_2\})$, $(\{i_1\},\emptyset)$, $(\{i_1,i_2\},\emptyset)$, and $(\{i_2\},\emptyset)$.
    For a matrix 
\[
B=
\begin{pmatrix}
\operatorname{Id}_{i_1} & 0 & 0\\
(a_{ij}) & \operatorname{Id}_{i_2-i_1} & 0\\
(b_{ij}) & (c_{ij}) & \operatorname{Id}_{n-i_2}
\end{pmatrix}\in \on{L}_I(\o),
\]
one easily obtains
\begin{align*}
R_{\{i_1\},\{i_2\}}(B)
&=
\{a_{ij}a_{i'j}\mid 1\le j\le i_1<i<i'\le i_2\}
\cup
\{a_{ij}^2-a_{ij}\mid 1\le j\le i_1<i\le i_2\},\\    
R_{\{i_1\},\emptyset}(B)
&=
\{a_{ij}(b_{i'j}-c_{i'i})\mid 1\le j\le i_1<i\le i_2<i'\le n\},\\
R_{\{i_1,i_2\},\emptyset}(B)&=
\{b_{ij}b_{i'j}-\sum_{k=i_1+1}^{i_2}a_{kj}c_{ik}c_{i'k}
\mid 1\le j\le i_1,\ i_2<i<i'\le n\}\\
&\quad\cup
\{b_{ij}^2-b_{ij}-\sum_{k=i_1+1}^{i_2}a_{kj}(c_{ik}^2-c_{ik})
\mid 1\le j\le i_1,\ i_2<i\le n\},\\
R_{\{i_2\},\emptyset}(B)
&=
\{c_{ij}c_{i'j}\mid i_1<j\le i_2<i<i'\le n\}
\cup
\{c_{ij}^2-c_{ij}\mid i_1<j\le i_2<i\le n\}.
\end{align*}
\end{example}
\begin{definition}\label{def: definition of tau}
For each $E\in \on{L}_I(\o/\m)$ we define
\begin{equation*}
\tau_{I,E}
:=
\#\mathcal{T}_{I,E}
=
\#\{\sigma\in \mathbb{S}_n:\ E\sigma\in \mathcal{R}_I(\o/\m)\}.
\end{equation*}    Equivalently, $\tau_{I,E}$ is the number of $I$-reduced matrices $F\in \on{GL}_n(\o/\m)$ that become $E$ after a suitable column permutation.
\end{definition}
\begin{remark}
\begin{enumerate}
    \item Note that $\tau_{I,E}=\tau_{I,E'}$ whenever $E$ and $E'$ have the same
pattern of zero entries.
\item The sum of all these coefficients is
\begin{equation}\label{eq: sum of tau}
\sum_{E\in \operatorname{L}_I(\o/\m)} \tau_{I,E}
=
\#\mathcal{R}_I(\o/\m)
=
\binom{n}{I}_q.
\end{equation}
\end{enumerate}
\end{remark}

Recall that if $x,y\in \o$ and $S\subset\o$, the condition $x|yS$ means $x|ys$ for all $s\in S$.
\begin{proposition}\label{prop: expression for hn as an integral}
For all non-empty $I\subseteq [n-1]$ and all $(r_0,\mathbf{r})\in \mathbb{N}_0\times\mathbb{N}^I$,
\[
h_n(I,r_0,\mathbf{r})
=
q^{\sum_{\iota\in I} r_\iota\,\iota(n-\iota)}
\sum_{E\in \operatorname{L}_I(\o/\m)}
\tau_{I,E}
\int\limits_{\substack{\pi^{\sum_{\iota\in J}r_\iota}|\pi^{\sum_{\iota\in K\cup\{0\}}r_\iota}R_{J,K}(B)\\ \text{for all $(J,K)$ $I$-admissible, and}\\ \overline{B}=E}}
|dB|.
\]
\end{proposition}
   
\begin{proof}
By Proposition~\ref{prop: final parametrization of submodules of type (I,r)},
\begin{align*}
h_n(I,r_0,\mathbf{r})
&=
\sum_{E\in \operatorname{L}_I(\o/\m)}
\sum_{\sigma\in \mathcal{T}_{I,E}}
\#
\Bigl(
\operatorname{L}_{I,\mathbf{r}}
\backslash
\bigl\{
B\in \operatorname{L}_I(\o):
\ \overline{B}=E
\ \wedge\
\langle \pi^{r_0}D(I,\mathbf{r})B\sigma\rangle\leq_s \o^n
\bigr\}
\Bigr).
\end{align*}
Note that the condition
\(
\langle \pi^{r_0}D(I,\mathbf{r})B\sigma\rangle\leq_s \o^n
\)
is equivalent to
\(
\langle \pi^{r_0}D(I,\mathbf{r})B\rangle\leq_s \o^n
\),
since $\sigma$ merely permutes the coordinates.
Hence, using \eqref{eq: cardinal of a quotient} and \eqref{eq: measure of L(I,r)}, we may rewrite the above expression as
\begin{align*}
&=
\sum_{E\in \operatorname{L}_I(\o/\m)}
\sum_{\sigma\in \mathcal{T}_{I,E}}
\frac{
\mu\bigl(
\{B\in \operatorname{L}_I(\o):
\ \overline{B}=E
\ \wedge\
\langle \pi^{r_0}D(I,\mathbf{r})B\rangle\leq_s \o^n
\}
\bigr)
}{\mu(\operatorname{L}_{I,\mathbf{r}})}\\
&=
\sum_{E\in \operatorname{L}_I(\o/\m)}
\tau_{I,E}\,
q^{\sum_{\iota\in I} r_\iota\,\iota(n-\iota)}
\,
\mu\bigl(
\{B\in \operatorname{L}_I(\o):
\ \overline{B}=E
\ \wedge\
\langle \pi^{r_0}D(I,\mathbf{r})B\rangle\leq_s \o^n
\}
\bigr).
\end{align*}

We now analyze the condition
\(
\langle \pi^{r_0}D(I,\mathbf{r})B\rangle\leq_s \o^n
\).
Write $D=D(I,\mathbf{r})$ and let $B_1,\ldots,B_n$ denote the row vectors of $B$.
Then the $i$-th row of $\pi^{r_0}DB$ is $\pi^{r_0}D_{ii}B_i$.
Thus the submodule $\langle \pi^{r_0}DB\rangle$ is closed under multiplication if and only if,
for all $1\le i\le i'\le n$, there exists a vector $v\in \o^n$ such that
\[
\pi^{r_0}D_{ii}B_i \, \pi^{r_0}D_{i'i'}B_{i'} = v\,\pi^{r_0}DB.
\]

The $j$-th entry of the vector $(B_i\, B_{i'})B^{-1}$ is
\(
\sum_{k=1}^i B_{ik}B_{i'k}(B^{-1})_{kj}
\).
Consequently, the above condition is equivalent to
\[
\pi^{r_0}D_{ii}D_{i'i'}
\left(
\sum_{k\le i} B_{ik}B_{i'k}(B^{-1})_{kj}
\right)
\equiv 0 \pmod{D_{jj}}
\]
for all $j=1,\ldots,n$.
Equivalently, $\langle \pi^{r_0}DB\rangle$ is closed under multiplication if and only if
\[
\prod_{\substack{\iota\in I\\ \iota\ge j}}\pi^{r_\iota}
\;\Big|\;
\pi^{r_0}
\left(\prod_{\substack{\iota\in I\\ \iota\ge i}}\pi^{r_\iota}\right)
\left(\prod_{\substack{\iota\in I\\ \iota\ge i'}}\pi^{r_\iota}\right)
\left(
\sum_{k=1}^i B_{ik}B_{i'k}(B^{-1})_{kj}
\right)
\]
for all $1\le i\le i'\le n$ and all $j=1,\ldots,n$.
Clearly, it suffices to verify this condition for $j<i$ and $[j,i)\cap I\neq\emptyset$. Thus, the above condition can be written as
\begin{align*}
    \prod_{\iota\in  I_j^i}\pi^{r_\iota}\big|\pi^{r_0}\left(\prod_{\iota\in I_{i'}}\pi^{r_\iota}\right)\left(
\sum_{k=1}^i B_{ik}B_{i'k}(B^{-1})_{kj}
\right),\quad 1\leq j<i\leq i'\leq n.
\end{align*}
Equivalently, 
\begin{align*}
    \prod_{\iota\in J}\pi^{r_\iota}\big| \pi^{r_0}\left(\prod_{\iota\in K}\pi^{r_\iota}\right)R_{J,K}(B)\quad\text{for all } (J,K)\ \text{$I$-admissible.}
\end{align*}
\end{proof}

\begin{proposition}\label{prop: integral expression for Xi_(n,I)} Let $\emptyset\neq I\subseteq [n-1]$. Then $\Xi_{n,I}\bigl(s_0,(s_\iota)_{\iota\in I}\bigr)$ is given by
    \begin{align*}
\frac{q^{-\sum_{\iota\in I}\iota(s_\iota-n+\iota)}}{(1-q^{-1})^{\#I+1}}
\sum_{E\in \operatorname{L}_I(\o/\m)}
\tau_{I,E}
\int\limits_{\substack{\pi^{\#J}\prod_{\iota\in J}x_\iota\big| x_0\pi^{\#K}\prod_{\iota\in K} x_\iota R_{J,K}(B)\\ \text{for all $(J,K)$ $I$-admissible, and} \\ \overline{B}=E}}
|x_0|^{s_0 n-1}
\prod_{\iota\in I}
|x_\iota|^{\iota(s_\iota-(n-\iota))-1}
\,|d\mathbf{x}|\,|dB|,
\end{align*}
where all the variables are in $\o$.
\end{proposition}
\begin{proof}
Proceeding as in Example~\ref{ex: from power series to integral} and using
Proposition~\ref{prop: expression for hn as an integral}, we obtain the following
integral representation for $\Xi_{n,I}\bigl(s_0,(s_\iota)_{\iota\in I}\bigr)$:
\begin{align*}
\frac{1}{(1-q^{-1})^{\#I+1}}
\sum_{E\in \operatorname{L}_I(\o/\m)}
\tau_{I,E}
\int\limits_{\substack{(x_0,(x_\iota)_{\iota\in I},B)\in \o\times\m^I\times \on{L}_I(\o)\\\prod_{\iota\in J}x_\iota\big| x_0\prod_{\iota\in K} x_\iota R_{J,K}(B)\\ \text{for all $(J,K)$ $I$-admissible, and} \\ \overline{B}=E}}
|x_0|^{s_0 n-1}
\prod_{\iota\in I}
|x_\iota|^{\iota(s_\iota-(n-\iota))-1}
\,|d\mathbf{x}|\,|dB|.
\end{align*}
The proof follows from an application of the change of variables $x_\iota=\pi x_\iota'$ for $\iota\in I$.
\end{proof}


\subsubsection{\textbf{Relation with the cotype zeta function of subalgebras}}

Any unital $\mathfrak{o}$-subalgebra of $\o^{n+1}$ of finite index has elementary divisors $\pi^{\alpha_1},\ldots,\pi^{\alpha_n},1$ for some $\alpha_1\geq\cdots\geq \alpha_n\geq 0$.
The cotype zeta function of unital subalgebras of $\o^{n+1}$, as defined in \cite{LeeLee2025Cotype}, is the multivariable series
\begin{align*}
    \sum_{\alpha_1\geq\cdots\geq \alpha_n\geq 0} f_{n+1}(\alpha_1,\ldots,\alpha_n)
    q^{-s_1\alpha_1}\cdots q^{-s_n\alpha_n},
\end{align*}
where $f_{n+1}(\alpha_1,\ldots,\alpha_n)$ denotes the number of unital $\o$-subalgebras of $\o^{n+1}$ with elementary divisors $\pi^{\alpha_1},\ldots,\pi^{\alpha_n},1$.

This is equal to the series
\begin{align*}
    \sum_{\alpha_1\geq\cdots\geq \alpha_n\geq 0}
    h_n(\alpha_1,\ldots,\alpha_n)q^{-s_1\alpha_1}\cdots q^{-s_n\alpha_n},
\end{align*}
where $h_n(\alpha_1,\ldots,\alpha_n)$ denotes the number of $\o$-submodules of $\o^{n}$ that are closed under multiplication and have elementary divisors $\pi^{\alpha_1},\ldots,\pi^{\alpha_n}$.

It is easy to see that the latter series equals
\begin{align}\label{eq: formula for the cotype zeta function}
\frac{1}{1-q^{-(s_1+\cdots+s_n)}}+
\sum_{\emptyset\neq I\subseteq [n-1]}
\Xi_{n,I}\!\biggl(
\frac{s_1+\cdots+s_n}{n},
\left(\frac{\sum_{i\leq \iota}s_i}{\iota}\right)_{\iota\in I}
\biggr).    
\end{align}

\subsection{A lower bound for the abscissa of convergence}
\label{sec: lower bound for the abscissa}

\begin{proposition}\label{prop: lower bound for Xi-n-I}
Let $n\ge2$ and $\emptyset\neq I\subseteq[n-1]$.
The abscissa of convergence of
$\Xi_{n,I}\bigl(s,(s)_{\iota\in I}\bigr)$ is bounded below by
\[
 \max_{\iota\in I}\frac{\iota(n-\iota)}{n+\iota}.
\]
Consequently, the abscissa of convergence of $\Xi_n(s)$ is bounded
below by
\[
 \max_{1\le\iota\le n-1}\frac{\iota(n-\iota)}{n+\iota}.
\]
\end{proposition}

\begin{proof}
Put $\ell=\#I$ and $N_I=\dim\on{L}_I$.
We specialize Proposition~\ref{prop: integral expression for Xi_(n,I)}
at $s_0=s_\iota=s$ for every $\iota\in I$.

Since $R_{J,K}(B)\subseteq\o$, all the divisibility conditions
in the domain of integration follow from
\[
 \pi^\ell\prod_{\iota\in I}x_\iota\mid x_0.
\]
Restricting to this domain therefore gives a Dirichlet subseries
$D_{n,I}(s)$ of
$\Xi_{n,I}\bigl(s,(s)_{\iota\in I}\bigr)$.
Equivalently, this restriction retains precisely the terms with
$r_0\ge\sum_{\iota\in I}r_\iota$.

We compute this subseries using the change of variables
\[
 x_0=\pi^\ell y_0\prod_{\iota\in I}x_\iota.
\]
Its Jacobian has absolute value
$q^{-\ell}\prod_{\iota\in I}|x_\iota|$, so the integrand, including
the measure in the scalar variables, becomes
\[
 q^{-\ell ns}|y_0|^{ns-1}
 \prod_{\iota\in I}
 |x_\iota|^{(n+\iota)s-\iota(n-\iota)-1}
 \,|dy_0|\prod_{\iota\in I}|dx_\iota|.
\]
There are no remaining conditions on $B$ other than
$\overline B=E$. Each such residue class has measure $q^{-N_I}$,
and \eqref{eq: sum of tau} gives
\[
 \sum_{E\in\on{L}_I(\mathbb F_q)}\tau_{I,E}
 =\binom nI_q.
\]
Evaluating the resulting product of elementary integrals yields
\begin{equation*}
 D_{n,I}(s)=
 \frac{\binom nI_q\,q^{-N_I}}{1-q^{-ns}}
 \prod_{\iota\in I}
 \frac{q^{\iota(n-\iota)-(n+\iota)s}}
      {1-q^{\iota(n-\iota)-(n+\iota)s}}.
\end{equation*}
This series has non-negative coefficients and abscissa of
convergence
\[
 \max_{\iota\in I}\frac{\iota(n-\iota)}{n+\iota},
\]
which is also its rightmost real pole. Since it is a subseries
of $\Xi_{n,I}\bigl(s,(s)_{\iota\in I}\bigr)$, the first assertion
follows. The second follows from
\eqref{eq:decomposition_Xi_n} by taking $I=\{\iota\}$.
\end{proof}

\begin{remark}\label{rem: universal lower bound for the abscissa}
The same lower bound holds for the abscissa of convergence of the
subalgebra zeta function of any $\o$-algebra $L$ that is free of
rank $n\ge2$ as an $\o$-module.
Indeed, after choosing an $\o$-basis of $L$, the parametrization
by $I$-reduced matrices yields an integral representation as above.
The symmetry of $\o^n$ under coordinate permutations allows us
to combine the corresponding contributions into the weights
$\tau_{I,E}$. For a general algebra, these contributions are
retained separately, but imposing
\[
 \pi^{\#I}\prod_{\iota\in I}x_\iota\mid x_0
\]
still makes all the divisibility conditions expressing closure
under multiplication automatic.
The same restriction of the integration domains and change of
variables therefore give the lower bound
\[
 \max_{1\le\iota\le n-1}
 \frac{\iota(n-\iota)}{n+\iota}.
\]
\end{remark}

\section{Subalgebras of $\o^n$ with two distinct elementary divisors}\label{sec: the case I of one element}

Let $n\geq 2$ and $I=\{\iota\}$ for some $1\leq \iota<n$. 
In this section, we compute
$\Xi_{n,I}(s_0,s_\iota)$
and present some applications. In particular, we give new proofs of the formulas
for $g_n(p^n)$ and $g_n(p^{n+1})$ obtained in \cite{Liu2007} and
\cite{AKKM2021}, respectively.

If $S$ is a subset of $\o$, by $x|S$ we mean that $x|y$ for all $y\in S$.

\subsection{The generating series of subalgebras with two distinct elementary divisors}
We fix $I=\{\iota\}\subseteq [n-1]$. For 
$B=
\begin{pmatrix}
\operatorname{Id}_{\iota} & 0 \\
(b_{ij}) & \operatorname{Id}_{n-\iota}
\end{pmatrix}\in \on{L}_I(\o)$ we write
\begin{align*}
    R(B):=\{ b_{ij}b_{i'j}\,:\, 1\leq j\leq \iota<i<i'\leq n\}\cup\{b_{ij}(b_{ij}-1)\,:\, 1\leq j\leq\iota<i\leq n\}
\end{align*} 
\begin{definition}
    A matrix (over $\o$ or over $\o/\m$) is called \emph{very elemental} if its entries belong to $\{0,1\}$ and each column contains at most one $1$.
\end{definition}

\begin{lemma}\label{lem: preliminary integral for I=1}
    Let $E\in \on{L}_I(\o/\m)$. Then
\begin{align*}
    \int\limits_{\substack{B\in \on{L}_I(\o)\\ \pi x| R(B)\\ \overline{B}=E}}
    |x|^{s-1}|d x||dB|= \delta_E \frac{(1-q^{-1})q^{-(n-\iota)\iota}}{1-q^{-(s+(n-\iota)\iota)}}.
\end{align*}
where $\delta_E=1$ or $0$ according to whether the lower left block of $E$ is very elemental or not. 
\end{lemma}
\begin{proof}
 The condition $\pi x|b_{ij}(b_{ij}-1)\, \forall i,j$ with \mbox{$\iota +1\leq i\leq n$} and $\ 1\leq j\leq \iota$, means that each coefficient of $(b_{ij})$ is congruent to $0$ or $1$ modulo $\pi x$. The condition $\pi x | b_{ij}b_{i'j}\, \forall i,i',j$ with $\iota+1\leq i<i'\leq n$ and $\ 1\leq j\leq \iota$ means that in each column of $(b_{ij})$ there is at most one unit. Hence, the condition $\pi x|R(B)$ is equivalent to $(b_{ij})=(e_{ij})+\pi x  (c_{ij})$ for some very elemental matrix $(e_{ij})$ and some $(c_{ij})\in \on{Mat}_{(n-\iota)\times\iota}(\o)$. In particular, the integral is 0 if $\delta_E=0$. If $\delta_E=1$, then $(e_{ij})$ is completely determined by $E$, and we can apply the change of variables $(b_{ij})=(e_{ij})+\pi x (c_{ij})$. 
We obtain
\begin{align*}
    \int\limits_{\substack{\overline{B}=E\\ \pi x|R(B)}}
    |x|^{s-1}|d x||dB|=\int |\pi x|^{(n-\iota)\iota}|x|^{s-1}|dx||d(c_{ij})|=\frac{(1-q^{-1})q^{-(n-\iota)\iota}}{1-q^{-(s+(n-\iota)\iota)}}.
\end{align*}
\end{proof}

\begin{proposition}\label{prop: formula for ZI} Let $1\leq\iota\leq n-1$. Then 
    \begin{align*}
    \Xi_{n,\{\iota\}}(s_0,s_\iota)=\frac{q^{-\iota s_\iota}}{1-q^{-s_0n}q^{-\iota s_\iota}q^{\iota(n-\iota)}}\left(\binom{n}{\iota}_q\frac{q^{-s_0n}}{1-q^{-s_0n}}+\frac{\theta_{n-\iota}}{1-q^{-\iota s_\iota}} \right)
\end{align*}
where, for fixed $n$ and $0\leq k\leq n$, $\theta_k$ denotes the coefficient
of $z^{n-k}$ in the expansion of $\prod_{j=1}^{k+1}(1-jz)^{-1}$.
\end{proposition}
\begin{proof} Put $I=\{\iota\}$.
    By Proposition~\ref{prop: integral expression for Xi_(n,I)} and the description of $I$-admissible pairs in Example~\ref{ex: I-admissible pairs for I of one element},
    \begin{align*}
    \Xi_{n,I}(s_0,s_\iota)=\frac{q^{-t_1}}{(1-q^{-1})^2}\sum_{E\in \on{L}_I(\o/\m)}\tau_{I,E}\int\limits_{\substack{(x_0,x_\iota,B)\in \o^2\times \on{L}_I(\o)\\ \pi x_{\iota}| x_0 R(B)\\\bar{B}=E}}|x_0|^{t_0-1}|x_\iota|^{t_1-1}|d\x||dB|,
\end{align*}
where $(t_0,t_1)=\Bigl(s_0n,\iota\bigl(s_\iota-(n-\iota)\bigr)\Bigr)$. 

Let the integral in each term of the sum be denoted by $\mathcal{I}_E=\mc{I}_E(s_0,s_\iota)$.
We have
\begin{align*}
\mc{I}_E&=\int\limits_{\substack{\pi x_\iota|x_0\\\bar{B}=E}}|x_0|^{t_0-1}|x_\iota|^{t_1-1}|d\x||dB|+\int\limits_{\substack{x_0|x_\iota\\ \pi \frac{x_\iota}{x_0}|{R}(B)\\\bar{B}=E}}|x_0|^{t_0-1}|x_\iota|^{t_1-1}|d\x||dB|\\
&=|\pi|^{t_0}    \int\limits_{\substack{\bar{B}=E}}|x_0|^{t_0-1}|x_\iota|^{t_0+t_1-1}|d\x||dB|+\int\limits_{\substack{ \pi {x_\iota}|{R}(B)\\\bar{B}=E}}|x_0|^{t_0+t_1-1}|x_\iota|^{t_1-1}|d\x||dB|\\
&=q^{-t_0}\frac{(1-q^{-1})^2}{(1-q^{-t_0})(1-q^{-(t_0+t_1)})}q^{-(n-\iota)\iota}+\frac{1-q^{-1}}{1-q^{-(t_0+t_1)}}\int\limits_{\substack{\pi x_\iota|R(B)\\ \overline{B}=E}}|x_\iota|^{t_1-1}|dx_\iota||dB|\\
&=q^{-t_0}\frac{(1-q^{-1})^2}{(1-q^{-t_0})(1-q^{-(t_0+t_1)})}q^{-(n-\iota)\iota}+\delta_E\frac{1-q^{-1}}{1-q^{-(t_0+t_1)}}\, \frac{(1-q^{-1})q^{-(n-\iota)\iota}}{1-q^{-(t_1+(n-\iota)\iota)}},
\end{align*}
where the last equality follows from Lemma~\ref{lem: preliminary integral for I=1}. Summarizing,
\begin{align*}
    \mc{I}_E=\frac{(1-q^{-1})^2q^{-(n-\iota)\iota}}{1-q^{-(t_0+t_1)}}\left(\frac{q^{-t_0}}{1-q^{-t_0}}+\frac{\delta_E}{1-q^{-(t_1+(n-\iota)\iota)}}\right),
\end{align*}
where $\delta_E=1$ or $0$ according to whether the lower left block of $E$ is very elemental or not.

Therefore,
\begin{align*}
    \Xi_{n,\{\iota\}}(s_0,s_\iota)= \frac{q^{-(t_1+(n-\iota)\iota)}}{1-q^{-(t_0+t_1)}}\left(\frac{q^{-t_0}}{1-q^{-t_0}}\sum_{E\in \on{L}_I(\o/\m)}\tau_{I,E}+\frac{1}{1-q^{-(t_1+(n-\iota)\iota)}}\sum_{E\in \on{L}_I(\o/\m)}\tau_{I,E}\, \delta_E \right)
\end{align*}
Note that $t_1+(n-\iota)\iota=\iota s_\iota$.
By \eqref{eq: sum of tau},  
\begin{align*}
    \sum_{E\in \on{L}_I(\o/\m)}\tau_{I,E}=\binom{n}{\iota}_q.
\end{align*}
Hence, to complete the proof, we only need to prove that 
\begin{align*}
    \sum_{E\in \on{L}_I(\o/\m)}\tau_{I,E}\, \delta_E=\theta_{n-\iota}.
\end{align*}

The left-hand side equals the number of $I$-reduced matrices over $\o/\m$ whose entries are only $0$ or $1$ and such that there are at most two non-zero entries in each column.
To count these matrices, we first choose the column positions $j_{\iota+1},\ldots,j_n$ for the leading 1's of rows $\iota+1,\ldots,n$. The leading 1's of the first $\iota$ rows are then completely determined.

Next, we need to decide in which columns we are going to add an extra 1. Between columns $j_{\iota+1}$ and $j_{\iota+2}$ we can add extra 1's without restriction in the $(\iota+1)$-th row, so there are $2^{j_{\iota+2}-j_{\iota+1}-1}$ possible choices. Between columns $j_{\iota+2}$ and $j_{\iota+3}$, we can only add extra 1's in rows $\iota+1$ and $\iota+2$, so for each column between $j_{\iota+2}$ and $j_{\iota+3}$ we have 3 possibilities: to add no extra 1, to add an extra 1 in row $\iota+1$, or to add an extra 1 in row $\iota+2$. Hence, there are $3^{j_{\iota+3}-j_{\iota+2}-1}$ possible choices. Continuing the analysis in this way we find that 
\begin{align*}
    \sum_{E\in\on{L}_{I}(\o/\m)}\delta_E\,\tau_{I,E}&=\sum_{1\leq j_{\iota+1}<\cdots < j_n\leq n} 2^{j_{\iota+2}-j_{\iota+1}-1} 3^{j_{\iota+3}-j_{\iota+2}-1}\cdots (n-\iota)^{j_{n}-j_{n-1}-1}(n-\iota+1)^{n-j_{n}}\\
    &=\sum_{\substack{0\leq t_1,\ldots,t_{n-\iota+1}\leq \iota\\ t_1+\cdots+t_{n-\iota+1}=\iota}} 1^{t_1}\cdot2^{t_2}\,\cdots\, (n-\iota+1)^{t_{n-\iota+1}}.
\end{align*}
This is clearly the coefficient of $z^\iota$ in the expansion of $\prod_{i=1}^{n-\iota+1}\frac{1}{1-iz}$, that is, $\theta_{n-\iota}$. This completes the proof.
\end{proof}

\begin{corollary}\label{coro: formula for h_n for |I|=1}
Let $1\leq\iota< n$ and $(r_0,r)\in \mathbb{N}_0\times\mathbb{N}$. Then
    \begin{align*}
        h_n(\{\iota\}, r_0,r)=\begin{cases}
\displaystyle
\binom{n}{\iota}_q\, q^{(r-1)\iota(n-\iota)},
& \text{if } r_0\geq r,\\[2mm]
\displaystyle
\theta_{n-\iota}\, q^{r_0\iota(n-\iota)},
& \text{if } r\geq r_0+1.
\end{cases}
    \end{align*}
\end{corollary}
\begin{proof}
    $h_n(\{\iota\},r_0,r)$ is the coefficient of $q^{-r_0s_0n}q^{-r\iota s_\iota}$ in the series $\Xi_{n,\{\iota\}}(s_0,s_\iota)$. Hence, the result follows by looking at the series expansion of the expression for $\Xi_{n,\{\iota\}}(s_0,s_\iota)$ given in the above proposition.
\end{proof}

\begin{remark}
    The above corollary shows that for any $n$, any $\iota\in [n-1]$ and any $(r_0,r)\in\mathbb{N}_0\times\mb{N}$, there exists a polynomial $P_{n,\iota,r_0,r}(X)$ such that $h_n(\{\iota\},r_0,r)=P_{n,\iota,r_0,r}(q)$ for any compact discrete valuation ring $\o$ with residue field of $q$ elements. We can ask if the same is true for any non-empty $I\subseteq [n-1]$ and any $(r_0,\r)\in\mathbb{N}_0\times\mb{N}^I$. A positive answer to this question would imply that for each $n\geq 1$ and each $e\geq 0$, there exists a polynomial $P_{n,e}(X)$ such that for any compact discrete valuation ring $\o$ with residue field of $q$ elements, the coefficient of $q^{-es}$ in $\Xi_n(s)$ is $P_{n,e}(q)$.
\end{remark}

\subsubsection*{\textbf{The zeta function of unital subalgebras of $\o^3$}} We give a formula for
\begin{align*}
    \Xi_2(s)=\sum_{\Lambda\leq_s\o^2}[\o^2:\Lambda]^{-s}.
\end{align*}

Since the only non-empty $I\subseteq [2-1]$ is $I=\{1\}$,  \eqref{eq:decomposition_Xi_n} gives
\begin{align*}
    \Xi_2(s)=\frac{1}{1-q^{-2s}}+\Xi_{2,\{1\}}(s,s).
\end{align*}
We now apply Proposition~\ref{prop: formula for ZI}. Note that $\binom{2}{1}_q=1+q$ and the coefficient of $z$ in $\frac{1}{(1-z)(1-2z)}$ is $3$. Hence
\begin{align*}
    \Xi_2(s)&=\frac{1}{1-q^{-2s}}+\frac{q^{-s}}{1-q^{-3s+1}}\left((1+q) \frac{q^{-2s}}{1-q^{-2s}}+\frac{3}{1-q^{-s}} \right)\\
    &=\frac{(1+q^{-s})^2}{(1-q^{-s})(1-q^{-3s+1})}.
\end{align*}

\medskip

In particular, if $\o=\mb{Z}_p$ we reobtain
$$\zeta_{\mb{Z}_p^3}^{\scriptscriptstyle{1,<}}(s)=\zeta_{\mb{Z}_p^2}^{\scriptscriptstyle <}(s)= \frac{(1+p^{-s})^2}{(1-p^{-s})(1-p^{-3s+1})}\quad\text{and}\quad   \zeta_{\mb{Z}^3}^{\scriptscriptstyle{1,<}}(s)=\frac{\zeta(s)^3\zeta(3s-1)}{\zeta(2s)^2}.$$

\subsection{Enumeration of irreducible unital subalgebras of index $q^{e}$ for $e\leq n+1$}\label{sec: Application II}
Let $g_n(q^e)$ be the number of irreducible unital subalgebras of $\o^n$ of index $q^e$,
where irreducibility is defined as in Definition~\ref{def: irreducibility},
with congruences taken modulo $\m$. As noted in the introduction, $g_{n+1}(q^e)$ coincides with the number of subalgebras of $\o^n$ of index $q^{e}$ included in $\pi \o^n$. Hence, for $e\geq 1$,
\begin{align}\label{eq: g_n in terms of h}
    g_{n+1}(q^{e})=\delta_{n|e}+\sum\limits_{\substack{\emptyset\neq I\subseteq [n-1]\\ (r_0,\mathbf{r})\in \mathbb{N}\times\mathbb{N}^I\\ r_0n+\sum_{\iota\in I}\iota r_\iota=e }}h_n(I,r_0,\mathbf{r}),
\end{align}
where $\delta_{n|e}$ is 1 or 0 according to whether $n|e$ or not.

It follows that $g_{n+1}(q^e)=0$ if $e<n$, $g_{n+1}(q^n)=1$ and
\begin{align*}
    g_{n+1}(q^{n+1})&=\begin{cases}
         1&\text{if $n=1$}\\
        h_n(\{1\},1,1)&\text{if $n\geq 2$}
    \end{cases},\\
    g_{n+1}(q^{n+2})&=\begin{cases}
        1&\text{if $n=1$}\\
        1+h_2(\{1\},1,2)&\text{if $n=2$}\\
        h_n(\{1\},1,2)+h_n(\{2\},1,1)&\text{if $n\geq 3$}
    \end{cases}
\end{align*}
Using Corollary~\ref{coro: formula for h_n for |I|=1} and the quick calculation $\theta_{n-1}=\binom{n+1}{2}$, we obtain the formulas valid for all $n\geq 1$:
\begin{align*}
    g_{n+1}(q^{n+1})&=\binom{n}{1}_q,\\
    g_{n+1}(q^{n+2})&=\binom{n+1}{2}q^{n-1}+\binom{n}{2}_q.
\end{align*}

\section{Subalgebras of $\o^n$ with three distinct elementary divisors}\label{sec: the case I of two elements}
Let $n\geq 3$ and let $I=\{i_1,i_2\}_<\subseteq [n-1]$.
In \S\ref{sec: a first reduction}, we describe the domain of integration of the integral giving $\Xi_{n,I}(s_0,s_1,s_2)$ in Proposition~\ref{prop: integral expression for Xi_(n,I)} and
express the integral as a finite sum of simpler integrals by partitioning
the domain.
These integrals are then computed explicitly in
\S\ref{sec: the case n=3 and I=1,2} in the special case $n=3$ and $I=\{1,2\}$. We finally use this calculation in \S\ref{sec: the cotype zeta function of o4} to write a formula for the unital subalgebra cotype zeta function of $\o^4$.

\subsection{Initial algorithmic reduction}\label{sec: a first reduction} Fix $I=\{i_1,i_2\}_<\subseteq [n-1]$. By Proposition~\ref{prop: integral expression for Xi_(n,I)} and the description of $I$-admissible pairs given in Example~\ref{ex: I-admissible pairs for I of two elements},
\begin{align*}
    \Xi_{n,I}(s_0,s_1,s_2)=\frac{q^{-t_1-t_2}}{(1-q^{-1})^3}\sum_{E\in\on{L}_I(\o/\m)}\tau_{I,E}\, \mc{I}_E(t_0,t_1,t_2), 
\end{align*}
where 
$$(t_0,\, t_1,\, t_2)=\bigl(s_0n,\,i_1(s_1-n+i_1),\, i_2(s_2-n+i_2)\bigr),$$
\begin{align*}
    \mc{I}_E(t_0,t_1,t_2)=\int\limits_{\substack{(x_0,x_1,x_2,B)\in \o^3\times\on{L}_I(\o)\\ x_1|x_0x_2 R_1(B)\\ \pi x_1|x_0 R_2(B)\\ \pi^2 x_1x_2|x_0 R_3(B)\\ \pi x_2|x_0 R_4(B) \\ \overline{B}=E}}|x_0|^{t_0-1}|x_1|^{t_1-1}|x_2|^{t_2-1}|d\x||dB|,
\end{align*}
Here $R_1(B)$, $R_2(B)$, $R_3(B)$ and $R_4(B)$ denote the four sets
listed in Example~\ref{ex: I-admissible pairs for I of two elements}, in that order.

For complex numbers $u_0,u_1,u_2$ we write
\begin{align*}
    \begin{bmatrix}
        u_0\\ u_1\\ u_2
    \end{bmatrix}=|x_0|^{u_0-1}|x_1|^{u_1-1}|x_2|^{u_2-1}|d\x||dB|.
\end{align*}

By successive partitions of the domain of integration according to the relative divisibility of
$x_0,x_1,x_2$, we will be able to express $\mathcal{I}_E(t_0,t_1,t_2)$ as a sum of integrals of the same type, but where in each divisibility condition, the variables $x_0,x_1,x_2$ only appear on the left.
At each step of the partition, we make a change of
variables and update the variables. For instance, if $\mc{C}$ is a set of conditions defining a subset of $\o^3\times\on{L}_I(\o)$, then
\begin{align*}
    \int\limits_{\mc{C}}\begin{bmatrix}
        u_0\\ u_1\\ u_2
    \end{bmatrix}=\bigl(\int\limits_{\substack{\pi x_1|x_0\\ \mc{C}}}+\int\limits_{\substack{x_0|x_1\\ \mc{C}}}\bigr)\begin{bmatrix}
        u_0\\ u_1\\ u_2
    \end{bmatrix}=|\pi|^{u_0}\int\limits_{\mc{C}'}\begin{bmatrix}
        u_0\\ u_0+u_1\\ u_2
    \end{bmatrix}+\int\limits_{\mc{C}''}\begin{bmatrix}
        u_0+u_1\\ u_1\\ u_2
    \end{bmatrix}
\end{align*}
where $\mc{C}'$ and $\mc{C}''$ are the new sets of conditions obtained after the corresponding changes of variables and updates of variables.

\medskip

Let $\mc{C}_0$ be the set of conditions defining the domain of integration in $\mc{I}_E(t_0,t_1,t_2)$. Then
\begin{align*}
    \mc{I}_E(t_0,t_1,t_2)=\bigl(\int\limits_{\substack{\pi x_1|x_0\\ \mc{C}_0}}+\int\limits_{\substack{x_0|x_1\\ \mc{C}_0}}\bigr)\begin{bmatrix}
        t_0\\ t_1\\ t_2
    \end{bmatrix}=|\pi|^{t_0} \underbrace{\int\limits_{\substack{\pi x_2| x_0 R_3(B)\\ x_2|x_0x_1R_4(B)\\ \overline{B}=E}}\begin{bmatrix}
        t_0\\ t_0+t_1\\ t_2
    \end{bmatrix}}_{=:A_1}+\underbrace{\int\limits_{\substack{x_1|x_2 R_1(B)\\ \pi x_1| R_2(B)\\ \pi^2 x_1x_2| R_3(B)\\ \pi x_2|x_0 R_4(B) \\ \overline{B}=E}}\begin{bmatrix}
        t_0+t_1\\ t_1\\ t_2
    \end{bmatrix}}_{=:A_2}.
\end{align*}
Let $\mc{C}_i$ be the set of conditions defining the domain of integration in $A_i$. We have
\begin{align*}
    A_1&=\bigl(\int\limits_{\substack{\pi x_2|x_0\\ \mc{C}_1}}+\int\limits_{\substack{x_0|x_2\\ \mc{C}_1}}\bigr)\begin{bmatrix}
        t_0\\ t_0+t_1\\ t_2
    \end{bmatrix}=|\pi|^{t_0}\int\limits_{\overline{B}=E}\begin{bmatrix}
       t_0\\ t_0+t_1\\ t_0+t_2
    \end{bmatrix}+\underbrace{\int\limits_{\substack{\pi x_2|  R_3(B)\\ x_2|x_1R_4(B)\\ \overline{B}=E}}\begin{bmatrix}
        t_0+t_2\\ t_0+t_1\\ t_2
    \end{bmatrix}}_{=:A_{12}}.
\end{align*}
Let $\mc{C}_{12}$ be the set of conditions defining the domain of integration of $A_{12}$. Then
\begin{align*}
    A_{12}=\bigl(\int\limits_{\substack{x_2|x_1\\ \mc{C}_{12}}}+\int\limits_{\substack{\pi x_1|x_2\\ \mc{C}_{12}}}\bigr)\begin{bmatrix}
        t_0+t_2\\ t_0+t_1\\ t_2
    \end{bmatrix}=\int\limits_{\substack{\pi x_2|R_3(B)\\\overline{B}=E}}\begin{bmatrix}
        t_0+t_2\\ t_0+t_1\\ t_0+t_1+t_2
    \end{bmatrix}+|\pi|^{t_2}\int\limits_{\substack{\pi^2 x_1x_2|R_3(B)\\\pi x_2|R_4(B)\\\overline{B}=E}}\begin{bmatrix}
        t_0+t_2\\ t_0+t_1+t_2\\ t_2
    \end{bmatrix}
\end{align*}
On the other hand,
\begin{align*}
    A_2=\bigl(\int\limits_{\substack{\pi x_2|x_0\\\mc{C}_2}}+\int\limits_{\substack{x_0| x_2\\\mc{C}_2}}\bigr)\begin{bmatrix}
        t_0+t_1\\ t_1\\ t_2
    \end{bmatrix}=|\pi|^{t_0+t_1}\underbrace{\int\limits_{\substack{x_1|x_2R_1(B)\\\pi x_1|R_2(B)\\\pi^2x_1x_2|R_3(B)\\\overline{B}=E}}\begin{bmatrix}
        t_0+t_1\\ t_1\\ t_0+t_1+t_2
    \end{bmatrix}}_{=:A_{21}}+\underbrace{\int\limits_{\substack{x_1|x_0x_2 R_1(B)\\ \pi x_1|R_2(B)\\ \pi^2x_0x_1x_2|R_3(B)\\\pi x_2|R_4(B)\\\overline{B}=E}}\begin{bmatrix}
        t_0+t_1+t_2\\ t_1\\ t_2
    \end{bmatrix}}_{=:A_{22}}
\end{align*}
Let $\mc{C}_{21}$ be the set of conditions defining the domain of integration in $A_{21}$. Then
\[
\resizebox{\linewidth}{!}{$\displaystyle
\begin{aligned}
A_{21}=\bigl(\int\limits_{\substack{x_1|x_2\\ \mc{C}_{21}}}+\int\limits_{\substack{\pi x_2|x_1\\\mc{C}_{21}}}\bigr)\begin{bmatrix}
        t_0+t_1\\ t_1\\ t_0+t_1+t_2
    \end{bmatrix}=\int\limits_{\substack{\pi x_1|R_2(B)\\ \pi^2 x_1^2x_2|R_3(B)\\\overline{B}=E}}\begin{bmatrix}
        t_0+t_1\\ t_0+2t_1+t_2\\ t_0+t_1+t_2
    \end{bmatrix}+|\pi|^{t_1}\int\limits_{\substack{\pi x_1|R_1(B)\\ \pi^2 x_1x_2|R_2(B)\\ \pi^3 x_1x_2^2|R_3(B)\\\overline{B}=E}}\begin{bmatrix}
        t_0+t_1\\ t_1\\ t_0+2t_1+t_2
    \end{bmatrix}
\end{aligned}
$}
\]
Let $\mc{C}_{22}$ be the set of conditions defining the domain of integration in $A_{22}$. Then
\[
\resizebox{\linewidth}{!}{$\displaystyle
\begin{aligned}
A_{22}=\bigl(\int\limits_{\substack{ x_1|x_2\\ \mc{C}_{22}}}+\int\limits_{\substack{\pi x_2|x_1\\\mc{C}_{22}}}\bigr)\begin{bmatrix}
        t_0+t_1+t_2\\ t_1\\t_2
    \end{bmatrix}=\int\limits_{\substack{\pi x_1|R_2(B)\\\pi^2x_0x_1^2x_2|R_3(B)\\ \pi x_1x_2|R_4(B)\\\overline{B}=E}}\begin{bmatrix}
        t_0+t_1+t_2\\ t_1+t_2\\ t_2
    \end{bmatrix}+|\pi|^{t_1}\underbrace{\int\limits_{\substack{ \pi x_1|x_0R_1(B)\\\pi^2 x_1x_2|R_2(B)\\\pi^3 x_0x_1x_2^2|R_3(B)\\\pi x_2|R_4(B)\\\overline{B}=E}}\begin{bmatrix}
        t_0+t_1+t_2\\ t_1\\ t_1+t_2
    \end{bmatrix}}_{=:A_{222}}
\end{aligned}
$}
\]
Let $\mc{C}_{222}$ be the set of conditions defining the domain of integration in $A_{222}$. Then
\begin{align*}
    A_{222}&=\bigl(\int\limits_{\substack{\pi x_1|x_0\\ \mc{C}_{222}}}+\int\limits_{\substack{ x_0|x_1\\ \mc{C}_{222}}} \bigr)\begin{bmatrix}
        t_0+t_1+t_2\\ t_1\\ t_1+t_2
    \end{bmatrix}\\
    &=|\pi|^{t_0+t_1+t_2}\int\limits_{\substack{ \pi^2 x_1x_2|R_2(B)\\\pi^4 x_0x_1^2x_2^2|R_3(B)\\\pi x_2|R_4(B)\\\overline{B}=E}}\begin{bmatrix}
        t_0+t_1+t_2\\ t_0+2t_1+t_2\\ t_1+t_2
    \end{bmatrix}+\int\limits_{\substack{ \pi x_1|R_1(B)\\\pi^2 x_0x_1x_2|R_2(B)\\\pi^3 x_0^2x_1x_2^2|R_3(B)\\\pi x_2|R_4(B)\\\overline{B}=E}}\begin{bmatrix}
        t_0+2t_1+t_2\\ t_1\\ t_1+t_2
    \end{bmatrix}
\end{align*}

To express the final formula, for $u_1,u_2,u_3,u_4\in\o$ and
$E\in\on{L}_I(\o/\m)$, we define
\begin{align*}
 S_E(u_1,u_2,u_3,u_4)
 &:=
 \int\limits_{\substack{
 B\in\on{L}_I(\o)\\
 u_j\mid R_j(B),\ j=1,2,3,4\\
 \overline{B}=E}}
 |dB|,\\
 S(u_1,u_2,u_3,u_4)
 &:=
 \sum_{E\in\on{L}_I(\o/\m)}
 \tau_{I,E}\,S_E(u_1,u_2,u_3,u_4).
\end{align*}
Writing $\x=(x_0,x_1,x_2)$, we then define, for $k=1,\ldots,8$,
\begin{equation}\label{eq: definition of W_i}
 \mc{W}_k(r_0,r_1,r_2)
 :=
 \frac{1}{(1-q^{-1})^3}
 \int\limits_{\o^3}
 |x_0|^{r_0-1}|x_1|^{r_1-1}|x_2|^{r_2-1}
 S_k(x_0,x_1,x_2)\,|d\x|,
\end{equation}
where the functions $S_k$ are given by
\[
\begin{array}{c|l}
 k & S_k(x_0,x_1,x_2)\\ \hline
 1 & S(1,1,1,1)\\[3pt]
 2 & S(1,1,\pi x_2,1)\\[3pt]
 3 & S(1,1,\pi^2x_1x_2,\pi x_2)\\[3pt]
 4 & S(1,\pi x_1,\pi^2x_1^2x_2,1)\\[3pt]
 5 & S(\pi x_1,\pi^2x_1x_2,\pi^3x_1x_2^2,1)\\[3pt]
 6 & S(1,\pi x_1,\pi^2x_0x_1^2x_2,\pi x_1x_2)\\[3pt]
 7 & S(1,\pi^2x_1x_2,\pi^4x_0x_1^2x_2^2,\pi x_2)\\[3pt]
 8 & S(\pi x_1,\pi^2x_0x_1x_2,\pi^3x_0^2x_1x_2^2,\pi x_2)
\end{array}
\]
For instance, 
\begin{align*}
S_E(1,1,1,1)&=\mu\bigl(\{B\in\on{L}_I(\o)\,:\, \overline{B}=E\}\bigr)=q^{-(i_1(i_2-i_1)+i_2(n-i_2))},\\    
S(1,1,1,1)&=\binom{n}{\{i_1,i_2\}}_qq^{-(i_1(i_2-i_1)+i_2(n-i_2))},
\end{align*}
and
\begin{align}\label{eq: formula for W1}
\mc{W}_1(r_0,r_1,r_2)&=\binom{n}{i_2}_q\binom{i_2}{i_1}_q\frac{q^{-(i_1(i_2-i_1)+i_2(n-i_2))}}{(1-q^{-r_0})(1-q^{-r_1})(1-q^{-r_2})}.
\end{align}

We have shown
\begin{proposition}\label{prop: decomposition in terms of W}
    Let $\{i_1,i_2\}_<\subseteq [n-1]$. For $s_0,s_1,s_2\in \mb{C}$, we put $$(t_0,t_1,t_2)=\bigl(s_0n, i_1(s_1-n+i_1), i_2(s_2-n+i_2)\bigr).$$
    Then
    \begin{align*}
        &q^{t_1+t_2}\Xi_{n,\{i_1,i_2\}}(s_0,s_1,s_2)=|\pi|^{2t_0}\mc{W}_1(t_0,t_0+t_1, t_0+t_2)+|\pi|^{t_0}\mc{W}_2(t_0+t_2,t_0+t_1,t_0+t_1+t_2)\\
        &+|\pi|^{t_0+t_2}\mc{W}_3(t_0+t_2,t_0+t_1+t_2,t_2)+|\pi|^{t_0+t_1}\mc{W}_4(t_0+t_1,t_0+2t_1+t_2,t_0+t_1+t_2)\\
        &+|\pi|^{t_0+2t_1}\mc{W}_5(t_0+t_1,t_1,t_0+2t_1+t_2)+\mc{W}_6(t_0+t_1+t_2,t_1+t_2,t_2)\\
        &+|\pi|^{t_0+2t_1+t_2}\mc{W}_7(t_0+t_1+t_2,t_0+2t_1+t_2,t_1+t_2)+|\pi|^{t_1}\mc{W}_8(t_0+2t_1+t_2,t_1,t_1+t_2).
    \end{align*}
\end{proposition}

Our goal in the next subsection is to compute $\mc{W}_2,\ldots,\mc{W}_8$ when $n=3$ and $I=\{1,2\}$, and obtain $\Xi_{3,\{1,2\}}(s_0,s_1,s_2)$.

\subsection{Explicit and structural evaluation of the associated integrals}\label{sec: the case n=3 and I=1,2}
Assume that $n=3$ and $I=\{1,2\}$.
For $E=\begin{pmatrix}
    1&0&0\\ \alpha&1&0\\ \beta&\gamma&1
\end{pmatrix}\in\on{L}_I(\o/\m)$, we have
\begin{align*}
   S_{\alpha,\beta,\gamma}(u_1,u_2,u_3,u_4):= S_E(u_1,u_2,u_3,u_4)
=
\int\limits_{\substack{
u_1\mid a^2-a,\
u_2\mid a(b-c),\\
u_3\mid b^2-b-a(c^2-c),\ 
u_4\mid c^2-c,\\ \Bar{a}=\alpha,\ \Bar{b}=\beta,\ \Bar{c}=\gamma}}
|d(a,b,c)|.
\end{align*}
We also need to calculate  $\tau_{\alpha,\beta,\gamma}:=\tau_{I,E}$, the number of $\{1,2\}$-reduced
matrices in $\on{GL}_3(\o/\m)$ obtained from $E$ by
permutation of columns. To compute this number, we list all possible
$\{1,2\}$-reduced matrices together with their associated lower triangular
matrices.

\begin{multicols}{2}
\begin{enumerate}
\item
$\begin{pmatrix}
0&0&1\\
0&1&*\\
1&*&*
\end{pmatrix}
\;\Rightarrow\;
\begin{pmatrix}
1&0&0\\
*&1&0\\
*&*&1
\end{pmatrix}$

\item
$\begin{pmatrix}
0&1&0\\
0&0&1\\
1&*&*
\end{pmatrix}
\;\Rightarrow\;
\begin{pmatrix}
1&0&0\\
0&1&0\\
*&*&1
\end{pmatrix}$

\item
$\begin{pmatrix}
0&0&1\\
1&0&*\\
0&1&*
\end{pmatrix}
\;\Rightarrow\;
\begin{pmatrix}
1&0&0\\
*&1&0\\
*&0&1
\end{pmatrix}$

\item
$\begin{pmatrix}
1&0&0\\
0&0&1\\
0&1&*
\end{pmatrix}
\;\Rightarrow\;
\begin{pmatrix}
1&0&0\\
0&1&0\\
0&*&1
\end{pmatrix}$

\item
$\begin{pmatrix}
0&1&0\\
1&*&0\\
0&0&1
\end{pmatrix}
\;\Rightarrow\;
\begin{pmatrix}
1&0&0\\
*&1&0\\
0&0&1
\end{pmatrix}$

\item
$\begin{pmatrix}
1&0&0\\
0&1&0\\
0&0&1
\end{pmatrix}
\;\Rightarrow\;
\begin{pmatrix}
1&0&0\\
0&1&0\\
0&0&1
\end{pmatrix}$
\end{enumerate}
\end{multicols}
From this list, we deduce that

\begin{align}\label{eq: tau for n=3}
\tau_{\alpha,\beta,\gamma}
=
\begin{cases}
1, & \text{if } \alpha\neq 0\ \text{and}\ \gamma \neq 0,\\[4pt]
2, & \text{if $\alpha,\beta$ are non-zero and $\gamma=0$, or $\beta,\gamma$ are non-zero and $\alpha=0$,}\\[4pt]
3, & \text{if exactly one of $\alpha,\beta,\gamma$ is non-zero,}\\[4pt]
6, & \text{if } \alpha = \beta = \gamma = 0.
\end{cases}    
\end{align}

We now compute 
\begin{align*}
    S(u_1,u_2,u_3,u_4)=\sum_{\alpha,\beta,\gamma}\tau_{\alpha,\beta,\gamma}S_{\alpha,\beta,\gamma}(u_1,u_2,u_3,u_4)
\end{align*}
under some assumptions on $u_1,u_2,u_3,u_4$ that apply to the computation of $\Xi_{3,\{1,2\}}(s_0,s_1,s_2)$.

For $\epsilon\in \o/\m$ we define
\[
\chi_0(\epsilon)=
\begin{cases}
1,& \text{if } \epsilon=0,\\
0,& \text{if } \epsilon\neq 0,
\end{cases}
\qquad\text{and}\qquad
\chi^{\times}(\epsilon)=1-\chi_0(\epsilon).
\]
In other words, $\chi_0$ is the characteristic function of the zero element of $\o/\m$ and $\chi^{\times}$ is that of the units.
We extend the definition of $\chi_0$ to more variables by setting
\[
\chi_0(\epsilon_1,\ldots,\epsilon_r)=\chi_0(\epsilon_1)\cdots\chi_0(\epsilon_r).
\]
Note that for any $b\in \o$, 
$$b^2-b=\bigl(b-\chi_0(\overline{b})\bigr)\bigl(b-\chi^\times(\overline{b})\bigr)\quad\text{and}\quad b-\chi_0(\overline{b})\in \o^*.$$
This implies that for any $u\in\o$,
\begin{align*}
    u|b^2-b\, \Leftrightarrow\,  u|b-\chi^\times(\overline{b}).
\end{align*}
In particular, we can write
\begin{align}\label{eq: equivalent presentation of S}
    S_{\alpha,\beta,\gamma}(u_1,u_2,u_3,u_4)=\int\limits_{\substack{u_1|a-\chi^\times(\alpha),\  u_2|a(b-c),\\ u_3|b-\chi^\times(\beta)-a\frac{c-\chi_0(\gamma)}{b-\chi_0(\beta)}(c-\chi^\times(\gamma)),\ u_4|c-\chi^\times(\gamma),\\ \overline{a}=\alpha,\ \overline{b}=\beta,\ \overline{c}=\gamma}}|d(a,b,c)|.
\end{align}
\begin{lemma}\label{lem: change of variables}
    Fix $\alpha,\beta,\gamma\in \o/\m$, and put $\alpha'=\frac{\alpha\bigl(\gamma-\chi_0(\gamma)\bigr)}{\beta-\chi_0(\beta)}$. Then the map $(a',b,c)\mapsto \bigl(a'\frac{b-\chi_0(\beta)}{c-\chi_0(\gamma)},b,c\bigr)$ induces a diffeomorphism 
    $$\{(a',b,c)\in \o^3\,:\, \overline{a'}=\alpha',\overline{b}=\beta,\overline{c}=\gamma\}\cong\{(a,b,c)\in \o^3\,:\,\overline{a}=\alpha,\ \overline{b}=\beta,\ \overline{c}=\gamma\}$$
    whose Jacobian $J(a',b,c)$ has determinant in $\o^*$ for any $(a',b,c)$ in the domain.
\end{lemma}
\begin{proof}
Straightforward verification.
\end{proof}
\begin{lemma}\label{lem: trivial integral}
    Let $v\in \o$. Then 
    $$\int\limits_{\substack{v|c-\chi^\times(\gamma)\\\overline{c}=\gamma}}|dc|=\begin{cases}
        |\pi|&\text{if $v$ is a unit},\\ \chi_0(\gamma^2-\gamma)|v|&\text{if $\pi |v$}.
    \end{cases}$$
\end{lemma}
\begin{proof}
    Straightforward verification.
\end{proof}
\begin{lemma}\label{lem: S(1,1,u,v)}
    Let $u,v\in \o$ with $\pi|u$. Then
    \begin{align*}
        S_{\alpha,\beta,\gamma}(1,1,u,v)=\begin{cases}
            \chi_0\bigl(\beta^2-\beta-\alpha(\gamma^2-\gamma)\bigr)|\pi^2 u|&\text{if $v$ is a unit},\\
            \chi_0(\beta^2-\beta,\gamma^2-\gamma)|\pi uv|&\text{if $\pi|v$}
        \end{cases}
    \end{align*}
\end{lemma}
\begin{proof} The integral is clearly zero unless $\beta^2-\beta-\alpha(\gamma^2-\gamma)=0$ (or $\beta-\chi^\times(\beta)=\alpha\frac{\bigl(\gamma-\chi_0(\gamma)\bigr)}{\bigl(\beta-\chi_0(\beta)\bigr)}\bigl(\gamma-\chi^\times(\gamma)\bigr)$), so we assume that this equality holds. Applying the change of variables from Lemma~\ref{lem: change of variables}, we obtain
\begin{align*}
    S_{\alpha,\beta,\gamma}(1,1,u,v)=\int\limits_{\substack{u|b-\chi^\times(\beta)-a'(c-\chi^\times(\gamma)),\ v|c-\chi^\times(\gamma)\\ \overline{a'}=\alpha',\ \overline{c}=\gamma}}|d(a',b,c)|
\end{align*}  
Condition $\overline{b}=\beta$ was not included, as it follows from the assumption $\beta-\chi^\times(\beta)=\alpha'\bigl(\gamma-\chi^\times(\gamma)\bigr)$ and condition $u|b-\chi^\times(\beta)-a'\bigl(c-\chi^\times(\gamma)\bigr)$. The result of this integral is clearly $|u\pi|\int_{v|c-\chi^\times(\gamma),\ \overline{c}=\gamma}|dc|$. The result now follows from Lemma~\ref{lem: trivial integral}.
\end{proof}

\begin{proposition}\label{prop: S for W_2}
    Let $u_3\in \m$. Then $$S(1,1,u_3,1)=(q^2+8q+1)q^{-2}|u_3|.$$
\end{proposition}
\begin{proof}
Using the expression for $S_{\alpha,\beta,\gamma}(1,1,u_3,1)$ from Lemma~\ref{lem: S(1,1,u,v)}, we obtain
    \begin{align*}
        S(1,1,u_3,1)=|u_3|q^{-2}\sum_{\beta^2-\beta=\alpha(\gamma^2-\gamma)}\tau_{\alpha,\beta,\gamma}.
    \end{align*}
    The solutions of $\beta^2-\beta=\alpha(\gamma^2-\gamma)$ are parameterized as follows: $(0,\epsilon,\gamma)$, $(\alpha, \epsilon,\epsilon')$ for $\alpha\ne0$ and $\epsilon,\epsilon'\in \{0,1\}$, and $\bigl(\frac{\beta^2-\beta}{\gamma^2-\gamma},\beta,\gamma\bigr)$ for $\beta,\gamma\notin\{0,1\}$. Using this and the description of $\tau_{\alpha,\beta,\gamma}$ given in \eqref{eq: tau for n=3} one easily shows that the sum in the above expression is $(q^2+8q+1)$.
\end{proof}
\begin{proposition}\label{prop: S for W_3}
    Let $u_3,u_4\in \m$. Then $$S(1,1,u_3,u_4)=7(q+1)q^{-1}|u_3u_4|.$$
\end{proposition}
\begin{proof}
Using the expression for $S_{\alpha,\beta,\gamma}(1,1,u_3,u_4)$ from  Lemma~\ref{lem: S(1,1,u,v)}, we obtain
  \begin{align*}
      S(1,1,u_3,u_4)=|u_3u_4|q^{-1}\sum_{\beta^2-\beta=\gamma^2-\gamma=0}\tau_{\alpha,\beta,\gamma}.
  \end{align*}
   Using \eqref{eq: tau for n=3}, one easily obtains that the sum in the above expression is $7(q+1)$.
\end{proof}

For the next propositions, we shall use the following lemma.
\begin{lemma}\label{lem: a lemma on integrals}
Let $u_1,u_2,u_3,u_4\in \o\setminus\{0\}$ with $\pi|u_2|u_3$. 
\begin{enumerate}
    \item If $u_1$ is a unit and $u_2|u_4$, then $S_{\alpha,\beta,\gamma}(u_1,u_2,u_3,u_4)$ is given by
    \begin{align*}
        \bigl(\chi_0(\alpha,\beta,\gamma-1)+\chi_0(\alpha,\beta-1,\gamma)\bigr)|u_2u_3u_4|+\chi_0(\beta-\gamma,\gamma^2-\gamma)|u_3u_4\pi|.
    \end{align*}
    \item If $u_4$ is a unit and $u_1|u_2$, then $S_{\alpha,\beta,\gamma}(u_1,u_2,u_3,u_4)$ is given by
    \begin{align*}
        \bigl(\chi_0(\alpha,\beta^2-\beta)+\chi_0(\alpha-1,\beta-\gamma)\bigr)|\pi u_2u_3|+\chi_0(\beta-\gamma,\gamma^2-\gamma)|u_2u_3|\int\limits_{\substack{\pi u_1|\pi a|u_2,\ \overline{a}=\alpha-\chi^\times(\alpha)}}\frac{1}{|a|}|da|.
    \end{align*}
    \item If $\pi|u_4$, $u_1u_4|u_2$ and $\pi u_4|u_2$, then $S_{\alpha,\beta,\gamma}(u_1,u_2,u_3,u_4)$ is given by
    \begin{align*}
        &\bigl(\chi_0(\alpha,\beta,\gamma-1)+\chi_0(\alpha,\beta-1,\gamma)\bigr)|u_2u_3u_4|+\chi_0(\alpha^2-\alpha,\beta-\gamma,\gamma^2-\gamma)|u_2u_3|\\
        &+\chi_0(\beta-\gamma,\gamma^2-\gamma)|u_2u_3|\int\limits_{\substack{\pi u_1|\pi a|\frac{u_2}{u_4},\ \overline{a}=\alpha-\chi^\times(\alpha)}}\frac{1}{|a|}|da|.
    \end{align*}
\end{enumerate}
\end{lemma}
\begin{proof}
We use the presentation of $S_{\alpha,\beta,\gamma}(u_1,u_2,u_3,u_4)$ given in \eqref{eq: equivalent presentation of S}. The integral is clearly zero unless
\begin{align}\label{eq: condition on alpha,beta,gamma}
    \alpha(\beta-\gamma)=0\quad\text{and}\quad \beta^2-\beta-\alpha(\gamma^2-\gamma)=0\ \Bigl(\text{or}\ \beta-\chi^\times(\beta)=\alpha\frac{\gamma-\chi_0(\gamma)}{\beta-\chi_0(\beta)}\bigl(\gamma-\chi^\times(\gamma)\bigr)\Bigr). 
\end{align}

Assume first that $\beta\neq \gamma$. Then \eqref{eq: condition on alpha,beta,gamma} translates as $\alpha=0$ and $\beta^2=\beta$ (or $\beta=\chi^\times(\beta)$).
Condition $u_2|a(b-c)$ can be replaced by $u_2|a$. Conditions $u_1|a$, $\overline{a}=0$ and $\overline{b}=\beta$ follow from the other conditions. Applying the change of variables from Lemma~\ref{lem: change of variables} and Lemma~\ref{lem: trivial integral}, we obtain:
\begin{align*}
   S_{\alpha,\beta,\gamma}(u_1,u_2,u_3,u_4)&=\chi_0(\alpha,\beta^2-\beta)
   \int\limits_{\substack{  u_2|a'\\ u_3|b-\chi^\times(\beta)-a'(c-\chi^\times(\gamma))\\ u_4|c-\chi^\times(\gamma),\   \overline{c}=\gamma}}|d(a',b,c)|\\
    =\chi_0(\alpha,\beta^2-\beta)&|u_2u_3|\int\limits_{u_4|c-\chi^\times(\gamma),\ \overline{c}=\gamma}|dc|=\begin{cases}
       \chi_0(\alpha,\beta^2-\beta) |\pi u_2u_3|&\text{if $u_4\in \o^*$}\\
        \chi_0(\alpha,\beta^2-\beta,\gamma^2-\gamma)|u_2u_3u_4|&\text{if $\pi |u_4$}
    \end{cases}.
\end{align*}
Since $\beta\neq \gamma$, the result when $\pi|u_4$ can also be written as $\bigl(\chi_0(\alpha,\beta,\gamma-1)+\chi_0(\alpha,\beta-1,\gamma)\bigr)|u_2u_3u_4|$.

Assume now that $\beta=\gamma$. 
Applying the change of variables from Lemma~\ref{lem: change of variables} (note that $\alpha'=\alpha$), we obtain
\begin{align*}
    S_{\alpha,\beta,\gamma}(u_1,u_2,u_3,u_4)=\int\limits_{\substack{u_1|a'\frac{b-\chi_0(\beta)}{c-\chi_0(\gamma)}-\chi^\times(\alpha) \\  u_2|a'(b-c)\\ u_3|b-\chi^\times(\beta)-a'(c-\chi^\times(\gamma))\\ u_4|c-\chi^\times(\gamma)\\ \overline{a'}=\alpha,\ \overline{c}=\gamma}}|d(a',b,c)|=\int\limits_{\substack{u_1|a'-\chi^\times(\alpha) \\  u_2|a'(b-c)\\ u_3|b-\chi^\times(\gamma)-a'(c-\chi^\times(\gamma))\\ u_4|c-\chi^\times(\gamma)\\ \overline{a'}=\alpha,\ \overline{c}=\gamma}}|d(a',b,c)|.
\end{align*}
Condition $\overline{b}=\beta$ was omitted in the first expression because it follows from the hypothesis $\beta-\chi^\times(\beta)-\alpha\bigl(\gamma-\chi^\times(\gamma)\bigr)=0$ and the other conditions. The factor $\frac{b-\chi_0(\beta)}{c-\chi_0(\gamma)}$ from the first condition was omitted in the second expression. When $\alpha=0$, this factor can be omitted because it is a unit. When $\alpha\neq 0$, since $\beta=\gamma$, it is enough to show that the other conditions imply that $b\equiv c\mod u_1$. But this follows from $u_2|a'(b-c)$ because $a'$ is a unit.
To proceed with the computation, we apply the change of variable $b-\chi^\times(\gamma)-a'\bigl(c-\chi^\times(\gamma)\bigr)=b'u_3$. The integral becomes
\begin{align*}
        |u_3|\int\limits_{\substack{ u_1|a'-\chi^\times(\alpha)\\ u_2|a'(a'-1)(c-\chi^\times(\gamma))\\ u_4|c-\chi^\times(\gamma)\\\overline{a'}=\alpha,\  \overline{c}=\gamma}}|d(a',b',c)|=  \int\limits_{\substack{ u_1|a'-\chi^\times(\alpha)\\ u_2|(a'-\chi^\times(\alpha))(c-\chi^\times(\gamma))\\ u_4|c-\chi^\times(\gamma)\\\overline{a'}=\alpha,\  \overline{c}=\gamma}}|u_3||d(a',c)|.
    \end{align*}
We denote this last integral by $S$ and analyze it in two cases.

Case I: $u_1$ is a unit and $u_2|u_4$. Note that $S$ can be non-zero only if $\gamma-\chi^\times(\gamma)=0$, that is, $\gamma^2-\gamma=0$. 
In that case, condition $\overline{c}=\gamma$ follows from $u_4|c-\chi^\times(\gamma)$ and 
\begin{align*}
   S&= \chi_0(\gamma^2-\gamma)\int\limits_{\substack{u_4|c-\chi^\times(\gamma),\  \overline{a'}=\alpha}}|u_3||d(a',c)|=\chi_0(\beta^2-\beta)|\pi u_3u_4|\\
   &=\bigl(\chi_0(\alpha,\beta^2-\beta)+\chi^\times(\alpha)\chi_0(\beta^2-\beta)\bigr)|\pi u_3u_4|.
\end{align*}

Case II: $\pi u_4|u_2$ and $u_1u_4|u_2$. We express 
$$S=S_1+S_2,$$
where $S_1$ and $S_2$ are the integrals over the respective parts of the domain where $\frac{u_2}{u_4}|a'-\chi^\times(\alpha)$ and $\pi \bigl(a'-\chi^\times(\alpha)\bigr)|\frac{u_2}{u_4}$. Since $\pi|\frac{u_2}{u_4}$, $S_1$ can be non-zero only if $\alpha-\chi^\times(\alpha)=0$, that is, $\alpha^2-\alpha=0$. Moreover, if $\alpha=0$ then \eqref{eq: condition on alpha,beta,gamma} reduces to $\gamma^2-\gamma=0$.
Thus,
\begin{align*}
    S_1&=\bigl(\chi_0(\alpha,\gamma^2-\gamma)+\chi_0(\alpha-1)\bigr)\int\limits_{\substack{\frac{u_2}{u_4}|a'-\chi^\times(\alpha),\ u_4|c-\chi^\times(\gamma),\ \overline{c}=\gamma }}|u_3||d(a',c)|\\
    &=\bigl(\chi_0(\alpha,\gamma^2-\gamma)+\chi_0(\alpha-1)\bigr)\int\limits_{\substack{u_4|c-\chi^\times(\gamma), \overline{c}=\gamma}}\frac{|u_2u_3|}{|u_4|}|dc|\\
    &=\begin{cases}
      \bigl(\chi_0(\alpha,\gamma^2-\gamma)+\chi_0(\alpha-1)\bigr)  |\pi u_2u_3|&\text{if $u_4$ is a unit}\\
        \chi_0(\alpha^2-\alpha,\gamma^2-\gamma)|u_2u_3|&\text{if $\pi|u_4$}
    \end{cases}.
\end{align*}
In $S_2$, conditions $\pi u_4|\frac{u_2}{a'-\chi^\times(\alpha)}$ and $\frac{u_2}{a'-\chi^\times(\alpha)}|c-\chi^\times(\gamma)$ imply that $S_2$ can be non-zero only if $\gamma-\chi^\times(\gamma)=0$, that is, $\gamma^2-\gamma=0$. In that case, conditions $u_4|c-\chi^\times(\gamma)$ and $\overline{c}=\gamma$ are redundant. Hence,
\begin{align*}
    S_2=\chi_0(\gamma^2-\gamma)\int\limits_{\substack{\pi (a'-\chi^\times(\alpha))|\frac{u_2}{u_4},\ u_1|a'-\chi^\times(\alpha)\\ \frac{u_2}{a'-\chi^\times(\alpha)}|c-\chi^\times(\gamma),\ \overline{a'}=\alpha}}|u_3||d(a',c)| =\chi_0(\gamma^2-\gamma)|u_2u_3|\int\limits_{\substack{\pi u_1|\pi a|\frac{u_2}{u_4}\\ \overline{a}=\alpha-\chi^\times(\alpha)}}\frac{1}{|a|}|da|.
\end{align*}

The final formula for $S_{\alpha,\beta,\gamma}(u_1,u_2,u_3,u_4)$ follows easily from the above calculations.
\end{proof}

\begin{proposition}\label{prop: S for W_4}
Let $u_2,u_3\in \o\setminus\{0\}$ with $\pi|u_2|u_3$. Then 
$$S(1,u_2,u_3,1)=(10q-2)q^{-1}|u_2u_3|+12|u_2u_3|\int\limits_{\pi^2a|u_2}\frac{1}{|a|}|da|.$$    
\end{proposition}
\begin{proof}
   According to Lemma~\ref{lem: a lemma on integrals}(b), $S_{\alpha,\beta,\gamma}(1,u_2,u_3,1)$ is given by
    \begin{align*}
        \bigl(\chi_0(\alpha,\beta^2-\beta)+\chi_0(\alpha-1,\beta-\gamma)\bigr)|\pi u_2u_3|+\chi_0(\beta-\gamma,\gamma^2-\gamma)|u_2u_3|\int\limits_{\substack{\pi a|u_2\\ \overline{a}=\alpha-\chi^\times(\alpha)}}\frac{1}{|a|}|da|.
    \end{align*}
    The integral in the second term is $
        \int_{\substack{\pi |a\\ \pi a|u_2}}\frac{1}{|a|}|da|=\int_{\pi^2 a|u_2}\frac{1}{|a|}|da|$ 
    if $\alpha=0,1$, and it is $\int_{\overline{a}=\alpha-1}|da|=|\pi|$ if $\alpha\neq 0,1$.

    It follows that $S(1,u_2,u_3,1)$ is given by
    \begin{align*}
    \bigl(\sum_{\substack{\beta=0,1\\ \gamma}}\tau_{0,\beta,\gamma}+\sum_\gamma\tau_{1,\gamma,\gamma}\bigr)|u_2u_3|q^{-1}+\bigl(\sum_{\substack{\alpha=0,1\\\gamma=0,1}}\tau_{\alpha,\gamma,\gamma}\bigr)|u_2u_3|\int\limits_{\pi^2 a|u_2}\frac{1}{|a|}|da|+\bigl(\sum_{\substack{\alpha\neq 0,1\\ \gamma=0,1}}\tau_{\alpha,\gamma,\gamma}\bigr)|u_2u_3|q^{-1}.
    \end{align*}
The coefficient $\sum_{\substack{\beta=0,1\\ \gamma}}\tau_{0,\beta,\gamma}+\sum_\gamma\tau_{1,\gamma,\gamma}+\sum_{\substack{\alpha\neq 0,1\\ \gamma=0,1}}\tau_{\alpha,\gamma,\gamma}$ is easily computed using \eqref{eq: tau for n=3} and the result is $10q-2$. Similarly, the coefficient  $\sum_{\substack{\alpha=0,1\\\gamma=0,1}}\tau_{\alpha,\gamma,\gamma}$ equals 12. This proves the proposition.
\end{proof}

\begin{proposition}\label{prop: S for W_5}
    Let $u_1,u_2,u_3\in \o\setminus\{0\}$ with $\pi|u_1|u_2|u_3$. Then
    \begin{align*}
        S(u_1,u_2,u_3,1)=6(q+1)q^{-1}|u_2u_3|+12|u_2u_3|\int\limits_{\pi a|\frac{u_2}{u_1}}\frac{1}{|a|}|da|.
    \end{align*}
\end{proposition}
\begin{proof}
According to Lemma~\ref{lem: a lemma on integrals}(b), $S_{\alpha,\beta,\gamma}(u_1,u_2,u_3,1)$ is given by
     \begin{align*}
     \bigl(\chi_0(\alpha,\beta^2-\beta)+\chi_0(\alpha-1,\beta-\gamma)\bigr)|\pi u_2u_3|+\chi_0(\alpha^2-\alpha,\beta-\gamma,\gamma^2-\gamma)|u_2u_3|\int\limits_{\substack{\pi u_1|\pi a|u_2}}\frac{1}{|a|}|da|.
    \end{align*}
    Applying the change of variables $a=u_1a'$, the integral in the second term becomes $\int_{\pi a|\frac{u_2}{u_1}}\frac{1}{|a|}|da|$. 

    It follows that
    \begin{align*}
        S(u_1,u_2,u_3,1)=\bigl(\sum_{\substack{\beta=0,1\\ \gamma}}\tau_{0,\beta,\gamma}+\sum_{\gamma}\tau_{1,\gamma,\gamma}\bigr)|u_2u_3|q^{-1}+\bigl(\sum_{\substack{\alpha=0,1\\ \gamma=0,1}}\tau_{\alpha,\gamma,\gamma}\bigr)|u_2u_3|\int\limits_{\pi a|\frac{u_2}{u_1}}\frac{1}{|a|}|da|.
    \end{align*}
    The coefficient $\sum_{\substack{\beta=0,1\\ \gamma}}\tau_{0,\beta,\gamma}+\sum_{\gamma}\tau_{1,\gamma,\gamma}$ is easily computed using \eqref{eq: tau for n=3}, and the result is $6(q+1)$. Similarly, the coefficient  $\sum_{\substack{\alpha=0,1\\\gamma=0,1}}\tau_{\alpha,\gamma,\gamma}$ equals 12. This proves the proposition.
\end{proof}

\begin{proposition}\label{prop: S for W_6}
    Let $u_2,u_3,u_4\in \o$ with $\pi|u_2|u_3$ and $u_2|u_4$. Then 
    \begin{align*}
        S(1,u_2,u_3,u_4)=6|u_2u_3u_4| + 4(q+1)q^{-1}|u_3u_4|.
    \end{align*}
\end{proposition}
\begin{proof}
    Using the formula for $S_{\alpha,\beta,\gamma}(1,u_2,u_3,u_4)$ given in Lemma~\ref{lem: a lemma on integrals}(a), we obtain
    \begin{align*}
        S(1,u_2,u_3,u_4)=\bigl(\tau_{0,0,1}+\tau_{0,1,0}\bigr)|u_2u_3u_4|+\bigl(\sum_\alpha (\tau_{\alpha,0,0}+\tau_{\alpha,1,1})\bigr)|u_3u_4|q^{-1}.
    \end{align*}
    Using \eqref{eq: tau for n=3}, one easily obtains $\tau_{0,0,1}+\tau_{0,1,0}=6$ and $\sum_\alpha (\tau_{\alpha,0,0}+\tau_{\alpha,1,1})=4(q+1)$.
\end{proof}

\begin{proposition}\label{prop: S for W_7}
    Let $u_2,u_3,u_4\in \o\setminus\{0\}$ with $\pi|u_4$ and $\pi u_4|u_2|u_3$. Then
    \begin{align*}
        S(1,u_2,u_3,u_4)=6|u_2u_3u_4|+(16q-8)|u_2u_3|q^{-1}+12|u_2u_3|\int\limits_{\pi^2a|\frac{u_2}{u_4}}\frac{1}{|a|}|da|.
    \end{align*}
\end{proposition}
\begin{proof}
    By Lemma~\ref{lem: a lemma on integrals}(c), $S_{\alpha,\beta,\gamma}(1,u_2,u_3,u_4)$ is given by
        \begin{align*}
        &\bigl(\chi_0(\alpha,\beta,\gamma-1)+\chi_0(\alpha,\beta-1,\gamma)\bigr)|u_2u_3u_4|+\chi_0(\alpha^2-\alpha,\beta-\gamma,\gamma^2-\gamma)|u_2u_3|\\
        &+\chi_0(\beta-\gamma,\gamma^2-\gamma)|u_2u_3|\int\limits_{\substack{\pi a|\frac{u_2}{u_4},\ \overline{a}=\alpha-\chi^\times(\alpha)}}\frac{1}{|a|}|da|.
    \end{align*}
    The integral in the last term equals $|\pi|$ if $\alpha^2-\alpha\neq 0$, and equals $\int_{\substack{\pi|a\\ \pi a|\frac{u_2}{u_4}}}\frac{1}{|a|}|da|=\int_{\pi^2 a|\frac{u_2}{u_4}}\frac{1}{|a|}|da|$ if $\alpha^2-\alpha=0$. Therefore, 
    \begin{align*}
        S(1,u_2,u_3,u_4)&=\bigl(\tau_{0,0,1}+\tau_{0,1,0}\bigr)|u_2u_3u_4|+\bigl(\sum_{\substack{\alpha=0,1\\ \gamma=0,1}}\tau_{\alpha,\gamma,\gamma}\bigr)|u_2u_3|+\bigl(\sum_{\substack{\alpha\neq 0,1\\ \gamma=0,1}}\tau_{\alpha,\gamma,\gamma}\bigr)|u_2u_3|q^{-1}\\
        &+\bigl(\sum_{\substack{\alpha=0,1\\ \gamma=0,1}}\tau_{\alpha,\gamma,\gamma}\bigr)|u_2u_3|\int\limits_{\pi^2 a|\frac{u_2}{u_4}}\frac{1}{|a|}|da|.
    \end{align*}
    Using \eqref{eq: tau for n=3}, one easily obtains that the coefficients in the last expression are, respectively, $6$, $12$, $4(q-2)$ and $12$. The proposition follows.
\end{proof}
\begin{proposition}\label{prop: S for W_8}
    Let $u_1,u_2,u_3,u_4\in \o\setminus\{0\}$ with $\pi|u_1$, $\pi|u_4$ and $u_1u_4|u_2|u_3$. Then
    \begin{align*}
        S(u_1,u_2,u_3,u_4)={}&6|u_2u_3u_4|+12|u_2u_3|
        +12|u_2u_3|\int\limits_{\pi a|\frac{u_2}{u_1u_4}}\frac{1}{|a|}|da|.
    \end{align*}
\end{proposition}
\begin{proof}
    By Lemma~\ref{lem: a lemma on integrals}(c), $S_{\alpha,\beta,\gamma}(u_1,u_2,u_3,u_4)$ is given by
    \begin{align*}
        &\bigl(\chi_0(\alpha,\beta,\gamma-1)+\chi_0(\alpha,\beta-1,\gamma)\bigr)|u_2u_3u_4|+\chi_0(\alpha^2-\alpha,\beta-\gamma,\gamma^2-\gamma)|u_2u_3|\\
        &+\chi_0(\alpha^2-\alpha,\beta-\gamma,\gamma^2-\gamma)|u_2u_3|\int\limits_{\substack{\pi u_1|\pi a|\frac{u_2}{u_4}}}\frac{1}{|a|}|da|.
    \end{align*}
    Applying the change of variables $a=a'u_1$, the integral in the last term becomes $\int_{\pi a|\frac{u_2}{u_1u_4}}\frac{1}{|a|}|da|$.
    Therefore,
    \begin{align*}
        S(u_1,u_2,u_3,u_4)&=\bigl(\tau_{0,0,1}+\tau_{0,1,0}\bigr)|u_2u_3u_4|+\bigl(\sum_{\substack{\alpha=0,1\\ \gamma=0,1}}\tau_{\alpha,\gamma,\gamma}\bigr)|u_2u_3|+\bigl(\sum_{\substack{\alpha=0,1\\ \gamma=0,1}}\tau_{\alpha,\gamma,\gamma}\bigr)|u_2u_3|\int\limits_{\pi a|\frac{u_2}{u_1u_4}}\frac{1}{|a|}|da|.
    \end{align*}
    As in the previous proposition, the coefficients in the above terms are respectively $6$, $12$, $12$. The proposition follows.
\end{proof}

\medskip

We are now in a position to write a formula for the functions $\mc{W}_i(r_0,r_1,r_2)$ introduced in \eqref{eq: definition of W_i} when $n=3$ and $I=\{1,2\}$. The restrictions to non-zero divisibility parameters in the preceding evaluations do not affect these integrals, since the sets where any $x_j=0$ have measure zero. To write the final result, we use the notation
\begin{align*}
    \zeta_q(s):=\frac{1}{1-q^{-1}}\int|u|^{s-1}|du|=\frac{1}{1-q^{-s}}.
\end{align*}

\begin{itemize}
    \item The expression for $\mc{W}_1$ is a particular case of \eqref{eq: formula for W1}:
   \begin{align*}
\mc{W}_1(r_0,r_1,r_2)=(1+q+q^2)(1+q)\,q^{-3}\,\zeta_q(r_0)\zeta_q(r_1)\zeta_q(r_2).
\end{align*}
    \item By Proposition~\ref{prop: S for W_2},
\begin{align*}
\mc W_{2}(r_0,r_1,r_2)&= 
\frac{1}{(1-q^{-1})^3}
\int |x_0|^{r_0-1}|x_1|^{r_1-1}|x_2|^{r_2-1}\,(q^2+8q+1)q^{-2}|\pi x_2|\,|d(x_0,x_1,x_2)|\\
&=(q^2+8q+1)\,q^{-3}\,
\zeta_q(r_0)\zeta_q(r_1)\zeta_q(r_2+1).
\end{align*}
\item By Proposition~\ref{prop: S for W_3},
\begin{align*}
\mc W_3(r_0,r_1,r_2)    &=
\frac{1}{(1-q^{-1})^3}
\int |x_0|^{r_0-1}|x_1|^{r_1-1}|x_2|^{r_2-1}\, 7(q+1)q^{-1}
|\pi^3x_1x_2^2|\,|d(x_0,x_1,x_2)|\\
&=7(q+1)\,q^{-4}\,
\zeta_q(r_0)\zeta_q(r_1+1)\zeta_q(r_2+2).
\end{align*}
\end{itemize}
 For the calculation of $\mc{W}_4,\, \mc{W}_5,\, \mc{W}_7\,$ and $\mc{W}_8\,$, we shall need the following integral, where $\epsilon\in\{0,1\}$.
\begin{align*}
\frac{1}{1-q^{-1}}
\int\limits_{\pi^\epsilon a\mid x}\frac{|x|^{t-1}}{|a|}\,|d(x,a)|
&=
\frac{|\pi|^{\epsilon t}}{1-q^{-1}}
\int |a|^{t-1}|y|^{t-1}\,|d(y,a)|
=
|\pi|^{\epsilon t}\,(1-q^{-1})\,\zeta_q(t)^2.
\end{align*}

\begin{itemize}
\item By Proposition~\ref{prop: S for W_4},
\begin{align*}
\mc W_{4}(r_0,r_1,r_2)
&=\frac{1}{(1-q^{-1})^3}
\int |x_0|^{r_0-1}|x_1|^{r_1-1}|x_2|^{r_2-1}(10q-2)q^{-1}\,|\pi^3x_1^3x_2||d(x_0,x_1,x_2)|\\
&\quad+\frac{1}{(1-q^{-1})^3}
\int\limits_{\pi a\mid x_1} |x_0|^{r_0-1}|x_1|^{r_1-1}|x_2|^{r_2-1}\,
12\,\frac{|\pi^3x_1^3x_2|}{|a|}
\,|d(x_0,x_1,x_2,a)|\\
=(10q-2)&q^{-4}\,
\zeta_q(r_0)\zeta_q(r_1+3)\zeta_q(r_2+1)
+12\,q^{-6-r_1}\,(1-q^{-1})\zeta_q(r_0)\zeta_q(r_1+3)^2\zeta_q(r_2+1).    
\end{align*}
\item By Proposition~\ref{prop: S for W_5},
\begin{align*}
\mc W_{5}(r_0,r_1,r_2)
&=\frac{1}{(1-q^{-1})^3}
\int |x_0|^{r_0-1}|x_1|^{r_1-1}|x_2|^{r_2-1}\, 6(q+1)q^{-1}\,|\pi^5x_1^2x_2^3||d(x_0,x_1,x_2)|\\
&\quad +\frac{1}{(1-q^{-1})^3}
\int\limits_{a\mid x_2} |x_0|^{r_0-1}|x_1|^{r_1-1}|x_2|^{r_2-1}\,
12\,\frac{|\pi^5x_1^2x_2^3|}{|a|}\,|d(a,x_0,x_1,x_2)|\\
=6(q+1)&q^{-6}\,
\zeta_q(r_0)\zeta_q(r_1+2)\zeta_q(r_2+3)
+12\,q^{-5}(1-q^{-1})\zeta_q(r_0)\zeta_q(r_1+2)\zeta_q(r_2+3)^2.
\end{align*}
\item By Proposition~\ref{prop: S for W_6}, 
\begin{align*}
\mc W_{6}(r_0,r_1,r_2)
&=\frac{1}{(1-q^{-1})^3}
\int |x_0|^{r_0-1}|x_1|^{r_1-1}|x_2|^{r_2-1}
\\
&\quad \quad \quad \,\left(
6\,|\pi^4x_0x_1^4x_2^2|
+4(q+1)q^{-1}\,|\pi^3x_0x_1^3x_2^2|
\right)\,|d(x_0,x_1,x_2)|\\
=6\,q^{-4}&
\zeta_q(r_0+1)\zeta_q(r_1+4)\zeta_q(r_2+2)
+4(q+1)\,q^{-4}\,
\zeta_q(r_0+1)\zeta_q(r_1+3)\zeta_q(r_2+2).
\end{align*}
\item By Proposition~\ref{prop: S for W_7},
\begin{align*}
\mc W_{7}(r_0,r_1,r_2)
&=\frac{1}{(1-q^{-1})^3}
\int |x_0|^{r_0-1}|x_1|^{r_1-1}|x_2|^{r_2-1}\\
&\quad\quad\quad\, 
\left(
6\,|\pi^7x_0x_1^3x_2^4|
+(16q-8)q^{-1}|\pi^6 x_0x_1^3x_2^3|\right)|d(x_0,x_1,x_2)|\\
&
\quad +\frac{12}{(1-q^{-1})^3}
\int\limits_{\pi a\mid x_1} |x_0|^{r_0-1}|x_1|^{r_1-1}|x_2|^{r_2-1}\frac{|\pi^6x_0x_1^3x_2^3|}{|a|}\,|d(a, x_0,x_1,x_2)|\\
=6\,q^{-7}&
\zeta_q(r_0+1)\zeta_q(r_1+3)\zeta_q(r_2+4)
+(16q-8)\,q^{-7}\,
\zeta_q(r_0+1)\zeta_q(r_1+3)\zeta_q(r_2+3)\\
&\quad
+12\,q^{-9-r_1}(1-q^{-1})\zeta_q(r_0+1)\zeta_q(r_1+3)^2\zeta_q(r_2+3).
\end{align*}
\item By Proposition~\ref{prop: S for W_8},
\begin{align*}
\mc W_{8}(r_0,r_1,r_2)
&=\frac{1}{(1-q^{-1})^3}
\int |x_0|^{r_0-1}|x_1|^{r_1-1}|x_2|^{r_2-1}\left(
6\,|\pi^6x_0^3x_1^2x_2^4|
+12\,|\pi^5x_0^3x_1^2x_2^3|\right)|d(x_0,x_1,x_2)|\\
&\quad 
+\frac{1}{(1-q^{-1})^3}
\int\limits_{\pi a\mid x_0}
|x_0|^{r_0-1}|x_1|^{r_1-1}|x_2|^{r_2-1}12\frac{|\pi^5x_0^3x_1^2x_2^3|}{|a|}\,|d(a, x_0,x_1,x_2)|\\
&=6\,q^{-6}\,
\zeta_q(r_0+3)\zeta_q(r_1+2)\zeta_q(r_2+4)
+12\,q^{-5}\,
\zeta_q(r_0+3)\zeta_q(r_1+2)\zeta_q (r_2+3)\\
&\quad
+12\,q^{-8-r_0}(1-q^{-1})\zeta_q(r_0+3)^2\zeta_q(r_1+2)\zeta_q(r_2+3).
\end{align*}
\end{itemize}

Finally, combining Proposition~\ref{prop: decomposition in terms of W} with
the above formulas for $\mathcal{W}_i$ gives the following expression
for $\Xi_{3,\{1,2\}}(s_0,s_1,s_2)$.
\begin{proposition}\label{prop: formula for Xi for I=12}
Setting     
$$X=q^{-3s_0},\ Y=q^{-s_1+2}\quad \text{and}\quad Z=q^{-2s_2+2},$$
$\Xi_{3,\{1,2\}}(s_0,s_1,s_2)$ is given by
\begin{align*}
\frac{YZ\;P(X,Y,Z;q)}{q^{11}\,(1-X)(1-XY)(1-XZ)(1-q^{-2}Y)(1-q^{-2}Z)(1-q^{-3}YZ)(1-q^{-1}XYZ)}    
\end{align*}
where
\[
P(X,Y,Z;q)=a_0+a_1X+a_2X^2+a_3X^3,
\]
with
\[
\begin{aligned}
a_0&=(4q^8+10q^7)+(8q^6-4q^5)Y-18q^4YZ,\\
a_1&=
(q^{10}+8q^9-3q^8-10q^7)
+(2q^8-11q^7-q^6)Z+(5q^8-20q^7-9q^6+4q^5)Y\\
&\quad+(-q^7-32q^6+6q^5+19q^4)YZ+
(-6q^5+19q^4+q^3)YZ^2\\&\quad
+(-5q^5+12q^4+5q^3)Y^2Z
+(12q^3+5q^2-q)Y^2Z^2,\\
a_2&=
(q^{11}+q^{10}-6q^9)
+(-q^9-7q^8+18q^7)Y
+(-q^9-4q^8+9q^7)Z\\
&\quad
+(-14q^8+5q^7+32q^6-5q^5)YZ
+(11q^6+6q^5-17q^4)YZ^2
+(8q^6+14q^5-10q^4-4q^3)Y^2Z\\
&\quad
+(13q^5+2q^4-18q^3-7q^2)Y^2Z^2
+(-7q^3-7q^2)Y^3Z^2
+(-6q^3-6q^2)Y^2Z^3,\\
a_3&=
(-q^{10}-2q^9+11q^8-6q^7)YZ
+(q^8+2q^7-8q^6+3q^5)YZ^2
+(q^8+2q^7-5q^6-6q^5)Y^2Z\\
&\quad+(q^7+q^6-13q^5-4q^4+3q^3)Y^2Z^2
+(-q^5-2q^4+4q^3+5q^2)Y^2Z^3\\
&\quad
+(-q^5-2q^4+5q^3+6q^2)Y^3Z^2
+(q^3+2q^2+2q+1)Y^3Z^3.
\end{aligned}
\]
\end{proposition}

\subsection{The cotype zeta function of unital subalgebras of $\o^4$}\label{sec: the cotype zeta function of o4}

Expression~\eqref{eq: formula for the cotype zeta function} for the cotype zeta function of unital subalgebras of $\o^4$ takes the form
\begin{align*}
    \frac{1}{1-q^{-(s_1+s_2+s_3)}}&+
    \Xi_{3,\{1\}}\!\left(\frac{s_1+s_2+s_3}{3},s_1\right)
    +\Xi_{3,\{2\}}\!\left(\frac{s_1+s_2+s_3}{3},\frac{s_1+s_2}{2}\right)\\
    &+\Xi_{3,\{1,2\}}\!\left(\frac{s_1+s_2+s_3}{3},s_1,\frac{s_1+s_2}{2}\right).
\end{align*}
By Proposition~\ref{prop: formula for ZI}, $\Xi_{3,\{1\}}\!\left(\frac{s_1+s_2+s_3}{3},s_1\right)$ and $\Xi_{3,\{2\}}\!\left(\frac{s_1+s_2+s_3}{3},\frac{s_1+s_2}{2}\right)$ are given, respectively, by
\begin{align*}
\frac{q^{-s_1}}{1-q^{-(s_1+s_2+s_3)-s_1+2}}
\left(
\frac{(1+q+q^2)q^{-(s_1+s_2+s_3)}}{1-q^{-(s_1+s_2+s_3)}}
+\frac{6}{1-q^{-s_1}}
\right)
\end{align*}
and
\begin{align*}
\frac{q^{-(s_1+s_2)}}{1-q^{-(s_1+s_2+s_3)-(s_1+s_2)+2}}
\left(
\frac{(1+q+q^2)q^{-(s_1+s_2+s_3)}}{1-q^{-(s_1+s_2+s_3)}}
+\frac{7}{1-q^{-(s_1+s_2)}}
\right).
\end{align*}
A formula for
$\Xi_{3,\{1,2\}}\!\left(\frac{s_1+s_2+s_3}{3},s_1,\frac{s_1+s_2}{2}\right)$
follows from Proposition~\ref{prop: formula for Xi for I=12}.
After computing the sum of these rational functions, we obtain the following result.
\begin{proposition}\label{prop: the cotype zeta function of o4}
    Setting $x=q^{-s_1},\, y=q^{-s_2},\, z=q^{-s_3}$, the cotype zeta function of unital subalgebras of $\o^4$ is 
    \[F(x,y,z)=
\frac{N(x,y,z)}
{(1-x)(1-xy)(1-xyz)(1-qx^2y)(1-q^2x^2yz)(1-q^2x^2y^2z)(1-q^3x^3y^2z)},
\]
where
\[
N(x,y,z)=A_0(x,y)+A_1(x,y)z+A_2(x,y)z^2+A_3(x,y)z^3,
\]
and
\[
\begin{aligned}
A_0(x,y)=\;&1+5x+6xy+(3q-2)x^2y+(3q-4)x^3y-6qx^3y^2-6qx^4y^2.\\
A_1(x,y)=\;&
(q-5)x^2y+(q-6)x^2y^2-(q+1)x^3y+(1-5q-5q^2)x^3y^2-(q+1)x^3y^3\\
&+(5-3q-14q^2)x^4y^2+(1+6q-6q^2-4q^3)x^4y^3\\
&+(q+q^2)x^5y^2+(8q+8q^2-7q^3)x^5y^3+(q+q^2)x^5y^4\\
&+(-q+4q^2+5q^3)x^6y^3+(-q+13q^3)x^6y^4+(-q^2+5q^3)x^7y^4.\\
A_2(x,y)=\;&
(5q^2-q^3)x^4y^3+(13q^2-q^4)x^5y^3+(5q^2+4q^3-q^4)x^5y^4\\
&+(q^3+q^4)x^6y^3+(-7q^2+8q^3+8q^4)x^6y^4+(q^3+q^4)x^6y^5\\
&+(-4q^2-6q^3+q^5+6q^4)x^7y^4+(-14q^3-3q^4+5q^5)x^7y^5\\
&-(q^5+q^4)x^8y^4+(q^5-5q^4-5q^3)x^8y^5-(q^5+q^4)x^8y^6\\
&+(-6q^5+q^4)x^9y^5+(-5q^5+q^4)x^9y^6.\\
A_3(x,y)=\;&
-6q^4x^7y^5+(3q^4-4q^5)x^8y^6-6q^4x^8y^5+(3q^4-2q^5)x^9y^6\\
&+6q^5x^{10}y^6+5q^5x^{10}y^7+q^5x^{11}y^7.
\end{aligned}
\]
\end{proposition}
\begin{remark}
Our formula coincides with the one predicted in \cite[Conjecture A.2]{ChintaIshamKaplan2024} in the case $\o=\mathbb{Z}_p$. 
\end{remark}

\section{Finite-field systems for counting subalgebras}
\label{sec: finite-field systems}

In this section, we study two systems of equations over
$\mathbb F_q=\o/\m$ that arise in the computation of
$g_{n}(q^{n+2})$ and $g_{n}(q^{n+3})$.
We first classify the solutions of the one-parameter system and then use
this classification to describe the degenerate and elementary solutions
of the two-parameter system. For the remaining solutions, we use the
geometry of the projective plane. In the last subsection, we count these
solutions according to their support and zero entries.

\subsection{Classification and geometry of the solution sets}

\subsubsection{\textbf{A system with one parameter}}\label{sec: system B=aC}

Let $m\geq1$. We first classify the triples
$(\alpha,\x,\z)\in\mathbb F_q\times(\mathbb F_q^m)^2$ satisfying
\begin{equation}\label{eq: system with one alpha}
\begin{cases}
 z_i^2-z_i=\alpha(x_i^2-x_i),&i\in[m],\\
 z_i z_j=\alpha x_i x_j,&1\leq i<j\leq m.
\end{cases}
\end{equation}
\begin{definition}
Write $\Supp(v)=\{i:v_i\neq0\}$ for $v\in\mathbb F_q^m$, and call $v$
\emph{elementary} if $v\in\{0,\e_1,\ldots,\e_m\}$, where the $\e_i$ are
the standard basis vectors. A solution is \emph{degenerate} if $\alpha=0$
and \emph{non-degenerate} otherwise. Its support is
$\Supp(\x)\cup\Supp(\z)$. A non-degenerate solution is \emph{elementary}
if at least one of $\x,\z$ is elementary, and \emph{non-elementary} otherwise.
\end{definition}

\begin{proposition}\label{prop: classification with one alpha}
The solutions of \eqref{eq: system with one alpha} form the following disjoint families.
\begin{enumerate}[label=(\alph*)]
\item Degenerate solutions:
$(0,\x,\z)$, where $\z$ is elementary and $\x$ is arbitrary.
\item Non-degenerate elementary solutions:
$(\alpha,\x,\z)$, where $\alpha\in\mathbb F_q^\times$ and both $\x,\z$ are elementary.
\item Non-degenerate non-elementary solutions with support $\{i\}$:
\[
\left(\frac{b^2-b}{c^2-c},\ c\e_i,\ b\e_i\right),
\qquad b,c\in\mathbb F_q\setminus\{0,1\}.
\]
\item Non-degenerate non-elementary solutions with support $\{i,j\}$, $i<j$:
\[
\left(\frac{b^2-b}{c^2-c},\ c\e_i+(1-c)\e_j,\ b\e_i+(1-b)\e_j\right),
\qquad b,c\in\mathbb F_q\setminus\{0,1\},
\]
or
\[
\left(1,\ b\e_i+c\e_j,\ b\e_i+c\e_j\right),
\qquad b,c\in\mathbb F_q^\times,\quad b+c\neq1.
\]
\item Non-degenerate non-elementary solutions with support of size at least three:
$(1,\x,\x)$, where $\#\Supp(\x)\geq3$.
\end{enumerate}
\end{proposition}

\begin{proof}
A vector $v$ is elementary precisely when $v_i^2-v_i=0$ for all $i$ and
$v_i v_j=0$ for all $i<j$. This gives (a) and (b), while (c) follows from
the single remaining diagonal equation.

We prove (e) next. After deleting zero coordinates, assume that the support of the solution
is $[m]$, with $m\geq3$. For $\alpha\neq0$, each $[a:b]\in\mathbb P^1(\mathbb F_q)$
has a unique $[a':b']$ satisfying $aa'=\alpha bb'$. The cross equations
therefore give
$[z_2:x_2]=\cdots=[z_m:x_m]$; exchanging the indices shows that all
$[z_i:x_i]$ coincide. The same equations now yield
$z_i^2=\alpha x_i^2$ for each $i$. Combining these with the diagonal
equations gives $z_i=\alpha x_i$. Since $(z_i,x_i)\neq(0,0)$, it follows
that $\alpha=1$ and $\z=\x$.

For (d), we may assume that the support of the solution is $\{1,2\}$. Adjoin the coordinates
\[
z_0=1-z_1-z_2,\qquad x_0=1-x_1-x_2.
\]
It is easy to check that the resulting triple again satisfies \eqref{eq: system with one alpha},
with coordinates indexed by $\{0,1,2\}$. If $(z_0,x_0)\neq(0,0)$, part (e)
gives $\alpha=1$ and $\z=\x$, with $x_1+x_2\neq1$, yielding the second family
in (d). Otherwise, $z_1+z_2=x_1+x_2=1$. Non-elementarity gives
$z_1,x_1\notin\{0,1\}$, and the remaining condition is
\[
\alpha=\frac{z_1^2-z_1}{x_1^2-x_1},
\]
which yields the first family. Each displayed triple satisfies the system,
and the families are disjoint by their degeneracy, elementarity, support,
and, in (d), their coordinate sums.
\end{proof}

For later counts, write $\operatorname h(v)=\max\Supp(v)$, with
$\operatorname h(0)=0$.
\begin{lemma}\label{lem: bounded-one-parameter-count}
For $0\le a,b\le m$, the number of non-degenerate solutions of
\eqref{eq: system with one alpha} with
$\operatorname h(\x)\le a$ and $\operatorname h(\z)\le b$ is
\[
 F_q(a,b)=(q-1)(a+1)(b+1)+G_q\bigl(\min(a,b)\bigr),
\]
where
\[
 G_q(t)=q^t-t-1+(q-2)(q-3)\binom{t+1}{2}.
\]
Here $G_q(t)$ counts the non-degenerate, non-elementary solutions
supported in $[t]$.
\end{lemma}
\begin{proof}
The elementary solutions contribute $(q-1)(a+1)(b+1)$. In every
non-elementary solution the two vectors have the same support, contained
in $[t]$ for $t=\min(a,b)$. Using the classification above, group these
solutions into those of support one, those with coincident vectors of
support at least two, and those with distinct vectors of support two.
Their respective counts are
\[
 t(q-2)^2,\qquad q^t-1-t(q-1),\qquad
 (q-2)(q-3)\binom t2.
\]
Their sum is $G_q(t)$.
\end{proof}

\subsubsection{\textbf{A system with two parameters}}\label{sec: system B=a1C1+a2C2}

We now consider the system
\begin{equation}\label{eq: system with a1 and a2}
\begin{cases}
 z_i^2-z_i=\alpha_1(x_i^2-x_i)+\alpha_2(y_i^2-y_i),&i\in[m],\\
 z_iz_j=\alpha_1x_ix_j+\alpha_2y_iy_j,&1\leq i<j\leq m,
\end{cases}
\end{equation}
in $(\alpha_1,\alpha_2,\x,\y,\z)\in\mathbb F_q^2\times(\mathbb F_q^m)^3$.
\begin{definition}
A solution of \eqref{eq: system with a1 and a2} is {\em degenerate} if $\alpha_1\alpha_2=0$. A non-degenerate solution is called  \emph{elementary} if at least one of
$\x,\y,\z$ is elementary. The {\em support} of a non-degenerate solution is
\[
 S=\{i\in[m]:(x_i,y_i,z_i)\ne(0,0,0)\}.
\]
\end{definition}

\begin{lemma}
The degenerate solutions reduce to the one-parameter system
\eqref{eq: system with one alpha}: either $\alpha_2=0$,
$(\alpha_1,\x,\z)$ is a solution of that system and $\y$ is arbitrary, or
$\alpha_1=0$, $\alpha_2\ne0$, $(\alpha_2,\y,\z)$ is a solution and $\x$ is
arbitrary. The non-degenerate elementary solutions split into the following
three disjoint classes:
\begin{enumerate}
 \item $\z$ is elementary and
 $(-\alpha_1/\alpha_2,\x,\y)$ solves
 \eqref{eq: system with one alpha};
 \item $\x$ is elementary, $\alpha_1\ne0$ is arbitrary, and
 $(\alpha_2,\y,\z)$ is a non-degenerate, non-elementary solution of
 \eqref{eq: system with one alpha};
 \item $\y$ is elementary, $\alpha_2\ne0$ is arbitrary, and
 $(\alpha_1,\x,\z)$ is a non-degenerate, non-elementary solution of
 \eqref{eq: system with one alpha}.
\end{enumerate}
\end{lemma}
\begin{proof}
For a degenerate solution, set the vanishing coefficient equal to zero.
For the elementary solutions, the quadratic expressions belonging to
an elementary vector vanish. In the second and third classes, the
one-parameter solution is non-elementary because $\z$ is non-elementary
and both coefficients are non-zero.
\end{proof}

For the remaining solutions, put $\alpha=(\alpha_1,\alpha_2)\in
(\mathbb F_q^\times)^2$ and consider the non-degenerate symmetric bilinear form
\[
 B_\alpha(u,v)=\alpha_1u_1v_1+\alpha_2u_2v_2-u_3v_3
 \qquad (u,v\in\mathbb F_q^3).
\]
Write $\mathbf1=(1,1,1)$ and $e=[1:1:1]\in \P^2=\mathbb P^2(\mathbb F_q)$.
For $P=[v]\in\mathbb P^2$, let
\[
 P^{\perp_\alpha}=\{[u]\in\mathbb P^2:B_\alpha(u,v)=0\}.
\]
If $v_i=(x_i,y_i,z_i)$, then \eqref{eq: system with a1 and a2} becomes
\begin{equation}\label{eq: two-parameter-polar-form}
 B_\alpha(v_i,v_i)=B_\alpha(v_i,\mathbf1),\qquad
 B_\alpha(v_i,v_j)=0\quad(i\ne j).
\end{equation}

\begin{definition}
We call $(\alpha,P)$ \emph{type~I} if neither $P$ nor $e$ lies on
$P^{\perp_\alpha}$, and \emph{type~II} if both do.
\end{definition}
These are precisely the directions that contain a non-zero solution of the diagonal equation.
Indeed, for a fixed non-zero representative $v$ of $P$, its non-zero
multiple $tv$ satisfies that equation if and only if
\[
 tB_\alpha(v,v)=B_\alpha(v,\mathbf1).
\]
There is consequently a unique such vector on a type-I direction, whereas
every non-zero vector on a type-II direction satisfies the equation.

\begin{lemma}\label{lem: type-II-pairs-and-polars}
The type-II pairs are given by the following table:
\[
\begin{array}{c|c}
 Q&(\alpha_1,\alpha_2)\\ \hline
 {[a:b:1]}&\displaystyle
 \left(\frac{1-b}{a(a-b)},\frac{a-1}{b(a-b)}\right)\\[2mm]
 {[0:1:1]}&(t,1)\\
 {[1:0:1]}&(1,t)\\
 {[1:1:0]}&(t,-t)\\
 e&(t,1-t)
\end{array}
\]
In the first row, $a,b\notin\{0,1\}$ and $a\ne b$; in the next three rows,
$t\ne0$; in the last row, $t\notin\{0,1\}$. The last row gives exactly
the type-II pairs with $\alpha_1+\alpha_2=1$.

If $(\alpha,Q)$ is type~II and $P\in Q^{\perp_\alpha}\setminus\{Q\}$,
then $(\alpha,P)$ is type~I if and only if $Q\ne e$.
\end{lemma}
\begin{proof}
For $Q=[a:b:c]$, the type-II conditions are
\[
 \begin{pmatrix}a&b\\a^2&b^2\end{pmatrix}
 \begin{pmatrix}\alpha_1\\\alpha_2\end{pmatrix}
 =\begin{pmatrix}c\\c^2\end{pmatrix}.
\]
If $ab(a-b)\ne0$, solving gives
\[
 \alpha_1=\frac{c(c-b)}{a(a-b)},\qquad
 \alpha_2=\frac{c(a-c)}{b(a-b)}.
\]
Both coefficients are non-zero precisely when $c\ne0$ and $c\notin\{a,b\}$,
so we may normalize $c=1$. The cases $a=0$, $b=0$, and $a=b$ give the
remaining rows. In the first row,
$\alpha_1+\alpha_2-1=-(a-1)(b-1)/(ab)\ne0$; the other rows yield the
assertion about the sum of the coefficients.

Now $Q^{\perp_\alpha}$ contains $Q$ and $e$. If $Q\ne e$, the distinct
lines $P^{\perp_\alpha}$ and $Q^{\perp_\alpha}$ meet at $Q$. Neither $P$
nor $e$ therefore lies on $P^{\perp_\alpha}$, proving that $P$ is type~I.
If $Q=e$, then $e\in P^{\perp_\alpha}$, so $P$ is not type~I.
\end{proof}

We call a type-II pair \emph{general} if it belongs to the first row of the
table, equivalently, if $Q$ has no zero coordinates and $Q\ne e$.

\begin{lemma}\label{lem: orthogonal points in the projective plane}
Let $P_1,\ldots,P_r\in\mathbb P^2$ satisfy
$P_i\in P_j^{\perp_\alpha}$ for $i\ne j$. If $r\ge4$, all but at most one
of these points coincide.
\end{lemma}
\begin{proof}
If the points do not all coincide, reindex them so that $P_1$ differs from
both $P_2$ and $P_3$. The lines $P_j^{\perp_\alpha}$ for $j\ge3$ contain
the distinct points $P_1,P_2$, hence $P_j=P_3$. The lines
$P_2^{\perp_\alpha}$ and $P_4^{\perp_\alpha}$ likewise contain the distinct
points $P_1,P_3$, so $P_2=P_4=P_3$.
\end{proof}

\begin{definition}
For a non-degenerate solution $(\alpha,\x,\y,\z)$ with support $S$, its \emph{core} is the set
\[
 \mathcal C=\{[v_i]:i\in S\}\subset \P^2.
\]
For $P\in \mathcal{C}$ we define its {\em multiplicity} as $
 \operatorname{mult}(P)=\#\{i\in S:[v_i]=P\}.$
\end{definition}
For every  $P\in \mathcal C$, $(\alpha,P)$ is type~I or type~II. If its multiplicity exceeds
one, it is type~II, since \eqref{eq: two-parameter-polar-form} then gives
$B_\alpha(v_i,v_i)=B_\alpha(v_i,\mathbf1)=0$.

We shall also use the following complement construction.
\begin{lemma}\label{lem: complementary-coordinate-two-parameter}
Let $v_1,\ldots,v_m\in\mathbb F_q^3$, and set
\[
 v_{m+1}=\mathbf1-\sum_{i=1}^m v_i.
\]
The following conditions are equivalent:
\begin{enumerate}[label=(\roman*)]
 \item The vectors $v_1,\ldots,v_m$ satisfy
 \eqref{eq: system with a1 and a2}.
 \item $B_\alpha(v_i,v_j)=0$ for all $1\le i<j\le m+1$.
 \item The vectors $v_1,\ldots,v_{m+1}$ satisfy
 \eqref{eq: system with a1 and a2} with $m$ replaced by $m+1$.
\end{enumerate}
\end{lemma}
\begin{proof}
For (i)$\Rightarrow$(ii), the original vectors are pairwise orthogonal, and
for every $i\in[m]$ we have
\[
 B_\alpha(v_i,v_{m+1})
 =B_\alpha(v_i,\mathbf1)-B_\alpha(v_i,v_i)
  -\sum_{\substack{j\le m\\j\ne i}}B_\alpha(v_i,v_j)=0.
\]
For (ii)$\Rightarrow$(iii), pairwise orthogonality and
$\sum_{i=1}^{m+1}v_i=\mathbf1$ give
$B_\alpha(v_i,\mathbf1)=B_\alpha(v_i,v_i)$ for every $i\in[m+1]$.
Finally, (iii)$\Rightarrow$(i) follows by restricting to the first $m$ vectors.
\end{proof}

\begin{proposition}\label{prop: geometry-of-two-parameter-cores}
Let a non-degenerate solution have support $S$ and core $\mathcal C$.
\begin{enumerate}
 \item $\#\mathcal C\le3$, and $\#\mathcal C\le2$ if $\#S\ge4$.
 \item If $\mathcal C=\{Q\}$ and $\#S\ge2$, then $(\alpha,Q)$ is type~II.
 \item Two distinct core points cannot both be type~II. If
 $\mathcal C=\{P,Q\}$ and $\#S\ge3$, then, after interchanging $P$ and $Q$
 if necessary, $P$ is type~I of multiplicity one and $Q$ is type~II,
 with $Q\ne e$ and $P\in Q^{\perp_\alpha}$.
 \item If $\#\mathcal C=3$, then $\#S=3$. The three core points are
 non-collinear and type~I, the matrix with rows $v_i$ for $i\in S$ has
 rank three, and $\sum_i x_i=\sum_i y_i=\sum_i z_i=1$.
\end{enumerate}
\end{proposition}
\begin{proof}
The first assertion follows from
Lemma~\ref{lem: orthogonal points in the projective plane}; the second
follows from multiplicity. If two distinct core points $P,Q$ were type~II,
both polar lines would be the line through $P,Q$, contradicting the
non-degeneracy of the polarity. Thus, in a two-point core with at least
three supported rows, exactly one point is repeated, and the preceding
lemmas give the third assertion.

Three distinct core points cannot be collinear: otherwise their polar
lines, each containing the other two points, would coincide. None can be
type~II, since its polar line would then contain all three points. They
therefore have rank three. Appending the complementary row gives another
solution; if that row were non-zero, its support would have size four and
its core would still contain the original three points, contradicting
the first assertion. The complementary row is consequently zero.
\end{proof}

For non-elementary solutions these statements give the following explicit
families. A one-point core with at least two supported rows is obtained by
choosing a type-II pair whose point has no zero coordinates and taking
arbitrary non-zero multiples of a representative in the supported rows.
A two-point core with at least three supported rows is obtained by choosing
a general type-II pair $(\alpha,Q)$ and
$P\in Q^{\perp_\alpha}\setminus\{Q\}$. One row is the unique non-zero
vector on $P$ satisfying the diagonal equation; the other supported rows
are arbitrary non-zero multiples of a representative of $Q$, with at least
two such rows.

Indeed, a type-II point with a zero coordinate makes a one-point solution
elementary. If the type-II point of a two-point core has a zero coordinate,
the three coordinate-zero cases in
Lemma~\ref{lem: type-II-pairs-and-polars} force the corresponding column
to be elementary. The type-II point must therefore be general, and the
constructions above satisfy both the diagonal and the orthogonality
conditions. They are automatically non-elementary when the type-II row
occurs at least twice.

The same two-point construction applies with only two supported rows,
provided the resulting solution is non-elementary. If $P$ has a zero
coordinate, this means that the corresponding entry of the vector chosen
on $Q$ must differ from one; otherwise that column is elementary. When
both points are type~I, we use the complement construction below to count
the solutions.

\begin{lemma}\label{lem: non-elementary solutions}
\label{lem: at-most-one-zero-on-support}
In a non-degenerate, non-elementary solution, at least two of $\x,\y,\z$
have support $S$, and at most one entry among $x_i,y_i,z_i$ with $i\in S$
is zero.
\end{lemma}
\begin{proof}
If two columns fail to have support $S$, permute the columns and make the
corresponding change of non-zero coefficients so that $x_i=y_j=0$ for some
$i,j\in S$. If $i\ne j$, the
off-diagonal equation gives $z_iz_j=0$; thus in either case some supported
row has two zero entries. Its remaining entry must be one, and the
off-diagonal equations force that column to be elementary, a contradiction.

It remains to consider zeros in the third column. For $\#S\le2$, two such
zeros make it elementary. For larger support, a one-point core has no
zero coordinates. In a two-point core, the general type-II point has no
zero coordinates, and its polar line avoids the coordinate vertices; only
the singleton row can have a zero, and at most one. Finally, a three-point
core has support three and column sums one, so two zeros in a column would
make it elementary as well.
\end{proof}

\subsection{Counts by support and zero pattern}\label{sec: counts by support}

\begin{definition}\label{def: A and B}
Fix a support $S\subseteq[m]$ of size $s\ge1$. Among the non-degenerate,
non-elementary solutions of \eqref{eq: system with a1 and a2} with support
$S$, let $A_s(q)$ count those with no zero entry in the supported rows
and let $B_s(q)$ count those with exactly one zero entry in a prescribed
position in these rows. The coefficients $\alpha_1,\alpha_2\in\mathbb F_q^\times$
are included among the variables being counted.
\end{definition}
Row permutations and permutations of the three vectors, with the
corresponding change of coefficients, show that the counts depend only
on $s$ and $q$.

Keeping the support and the zero position fixed allows the weights
$\tau_{I,E}$ to be applied separately. We first dispose of cores of size
one or two; this also covers every support of size greater than three.

\begin{proposition}\label{prop: count-small-cores}
For a fixed support of size $s\ge3$, the contributions from one-point and
two-point cores are
\[
\begin{array}{c|c|c}
 \text{size of the core}&A_s(q)&B_s(q)\\ \hline
 1&(q-2)^2(q-1)^s&0\\
 2&s(q-2)(q-3)^2(q-1)^{s-1}&(q-2)(q-3)(q-1)^{s-1}.
\end{array}
\]
In particular, for $s\ge4$,
\begin{align*}
 A_s(q)&=(q-2)(q-1)^{s-1}
       \bigl((q-2)(q-1)+s(q-3)^2\bigr),\\
 B_s(q)&=(q-2)(q-3)(q-1)^{s-1}.
\end{align*}
\end{proposition}
\begin{proof}
There are $(q-2)(q-3)$ general type-II pairs and $q-2$ pairs with point
$e$, hence $(q-2)^2$ type-II pairs with no zero coordinates. A one-point
core allows an independent non-zero scalar in each supported row, giving
the first row of the table.

For a two-point core, choose its general type-II pair $(\alpha,Q)$ in
$(q-2)(q-3)$ ways. Its polar line has three points with a zero coordinate.
After removing them and $Q$, there are $q-3$ choices for the type-I point.
Choose its singleton row in $s$ ways; its vector is uniquely determined,
and the other $s-1$ rows have arbitrary non-zero scalars. This gives the
entry in the $A_s$ column.

For $B_s$, the zero fixes the singleton row, and its column determines the
unique point on $Q^{\perp_\alpha}$ with that coordinate zero. There remain
$(q-1)^{s-1}$ choices of the type-II row scalars. Since $s\ge3$, each vector
still has at least two non-zero entries and is non-elementary. Finally,
cores have at most two points when $s\ge4$, so summing the table gives
the last two formulas.
\end{proof}

We next count the solutions whose core has three points. Relabelling their
support, we may take $m=3$. Put
\[
 E=\{[1:0:0],[0:1:0],[0:0:1]\},\qquad
 D=\{[u:v:w]:uvw=0\},\qquad U=\mathbb P^2\setminus D,
\]
and consider
\[
 \sigma:\mathbb P^2\setminus E\longrightarrow\mathbb P^2,
 \qquad [u:v:w]\longmapsto[uv:uw:vw].
\]
Its restriction to $U$ is an involution. If
\[
 L=\{u+v+w=0\},\qquad C=\{uv+uw+vw=0\},
\]
then $\sigma(L\cap U)=C\cap U$, $C\cap D=E$, and
$\#(C\cap U)=q-2$. In particular, $C$ is a non-singular conic with
$q+1$ rational points.

\begin{lemma}\label{lem: counting-three-point-cores}
For $m=3$, let $N_3(q)$ be the number of non-degenerate, non-elementary
solutions of \eqref{eq: system with a1 and a2} whose core has three points
and whose nine entries are non-zero. If one of the nine entries is prescribed
to be zero and the other eight are non-zero, denote the corresponding number
by $M_3(q)$. Then
\begin{align*}
 N_3(q)&=(q-2)(q-3)(q-4)(q^2-5q+10),\\
 M_3(q)&=(q-2)^2(q^2-5q+8).
\end{align*}
\end{lemma}
\begin{proof}
By Proposition~\ref{prop: geometry-of-two-parameter-cores}, the three column
sums are one and the matrix of the solution has rank three. For column
directions outside $L$, these sums determine unique representatives.
Writing $\widetilde\sigma(v)=(v_1v_2,v_1v_3,v_2v_3)$ for vectors, the
off-diagonal equations are
\[
 \widetilde\sigma(\z)
 =\alpha_1\widetilde\sigma(\x)+\alpha_2\widetilde\sigma(\y).
\]
The image directions are therefore collinear. They are distinct: at least
two columns lie in $U$, where $\sigma$ is injective, and a column in
$D\setminus E$ maps to a vertex in $E$. Conversely, three distinct
collinear image directions determine unique non-zero coefficients once the
column representatives have been normalized. The diagonal equations then
follow from the off-diagonal equations and the column sums, as in
Lemma~\ref{lem: complementary-coordinate-two-parameter}. We must additionally
retain the rank-three condition on the original columns.

Suppose first that all entries are non-zero. The image directions belong
to $U\setminus C$. Their line must avoid $E$: if, for example,
$\ell=\{a u+b v+c w=0\}$ has $a=0$, its preimages in $U$ lie on the line
$b u+c v=0$. Conversely, if $abc\ne0$, its preimages lie on the non-singular
conic $a uv+b uw+c vw=0$, so three distinct preimages are non-collinear.
Thus rank three is equivalent to the image line avoiding $E$.

For such a line put $k_\ell=\#\bigl(\ell\setminus(C\cup D)\bigr)$. The possible
lines and their admissible points are
\[
\begin{array}{c|c|c}
 \text{type of line}&\text{number}&k_\ell\\ \hline
 \text{secant through two points of }C\cap U
   &\binom{q-2}{2}&q-4\\
 \text{tangent at a point of }C\cap U&q-2&q-3\\
 \text{disjoint from }C&\frac{q(q-1)}2&q-2
\end{array}
\]
Only non-empty classes contribute. The number of exterior lines follows
by subtracting the $\binom{q+1}{2}$ secants and $q+1$ tangents from
$q^2+q+1$. Each line contributes $k_\ell(k_\ell-1)(k_\ell-2)$ ordered
triples, and summing gives $N_3(q)$.

For the second count we may prescribe $z_1=0$. Then
$\z=(0,t,1-t)$ with $t\notin\{0,1\}$, and its image is $[0:0:1]$.
The two other image directions are distinct points of $U\setminus C$ on
a line through this vertex. The two coordinate lines contribute nothing.
Of the remaining lines, one is tangent to $C$, with $q-1$ admissible points,
and $q-2$ are secants, each with $q-2$ admissible points.

To count the choices of $\z$, write such a line as $a u+b v=0$, with
$ab\ne0$. The preimages of the two chosen points lie on $a v+b w=0$. The only point
on $u=0$ that would make the three original columns collinear is
$[0:b:-a]$. For the tangent, $a=b$, this point lies on $L$ and is already
excluded, leaving $q-2$ choices of $\z$. For a secant, $a\ne b$, it is
admissible and must be removed, leaving $q-3$ choices. Consequently,
\begin{align*}
 M_3(q)
 &=(q-2)(q-1)(q-2)
   +(q-2)(q-3)(q-2)(q-3)\\
 &=(q-2)^2(q^2-5q+8).
\end{align*}
\end{proof}

Recall the definition of $A_s(q)$ and $B_s(q)$ from Definition~\ref{def: A and B}.
\begin{proposition}\label{prop: three-support-no-zero-count}
For a fixed support of size three,
\[
 A_3(q)=(q-2)(5q^4-41q^3+132q^2-209q+149).
\]
\end{proposition}
\begin{proof}
Add $N_3(q)$ to the one-point and two-point core contributions in
Proposition~\ref{prop: count-small-cores}. This gives
\[
 A_3(q)=(q-2)^2(q-1)^3+3(q-2)(q-3)^2(q-1)^2+N_3(q),
\]
which simplifies to the stated expression.
\end{proof}

\begin{proposition}\label{prop: three-support-prescribed-zero-count}
For a fixed support of size three and a prescribed zero entry,
\[
 B_3(q)=(q-2)(2q^3-12q^2+25q-19).
\]
\end{proposition}
\begin{proof}
For the prescribed zero, Proposition~\ref{prop: count-small-cores} and
Lemma~\ref{lem: counting-three-point-cores} give
\[
 B_3(q)=(q-2)(q-3)(q-1)^2+M_3(q),
\]
and the result follows.
\end{proof}

\begin{lemma}\label{lem: zero-complement-two-row-solutions}
For $m=2$, the number of non-degenerate solutions with all six entries
non-zero and
\[
 x_1+x_2=y_1+y_2=z_1+z_2=1
\]
is $(q-2)^4$. Among them, the number whose core consists of two points of
type~I is
\[
 Z_0(q):=(q-2)(q-3)(q^2-5q+10).
\]
\end{lemma}
\begin{proof}
Write the first row as $(a,b,c)$; the second is $(1-a,1-b,1-c)$, with
$a,b,c\notin\{0,1\}$. Setting $f(t)=t^2-t$, the identities
$f(1-t)=f(t)$ and $t(1-t)=-f(t)$ reduce the system to
\[
 f(c)=\alpha_1 f(a)+\alpha_2 f(b).
\]
For each of the $(q-2)^3$ choices of $(a,b,c)$, there are $q-2$ pairs of
non-zero coefficients. Thus the total is $(q-2)^4$.

The rows define the same projective point precisely when the first is
$t(1,1,1)$, with $t\notin\{0,1\}$. In that case
$\alpha_1+\alpha_2=1$, giving $(q-2)^2$ solutions.

If the two core points are distinct and one is of type II, the latter is
general. Choose its pair in $(q-2)(q-3)$ ways and its row in two ways.
For its representative $(a,b,1)$, the row scalar must avoid
$0,a^{-1},b^{-1},1$, leaving $q-4$ choices; these four values are distinct.
The remaining solutions have two type-I points, so their number is
\[
 (q-2)^4-(q-2)^2-2(q-2)(q-3)(q-4)
 =(q-2)(q-3)(q^2-5q+10).
\]
\end{proof}

\begin{proposition}\label{prop: two-support-no-zero-count}
For a fixed support of size two,
\[
 A_2(q)=(q-2)(q^4-5q^3+10q^2-16q+22).
\]
\end{proposition}
\begin{proof}
A one-point core contributes $(q-2)^2(q-1)^2$. For a core consisting of one
type-I point and one type-II point, choose the type-II row in two ways and
its general pair in $(q-2)(q-3)$ ways. There are $q-3$ choices for the
type-I point with no zero coordinates and $q-1$ non-zero scalars for the
type-II row. This contributes $2(q-2)(q-3)^2(q-1)$.

It remains to count cores consisting of two type-I points. Append the row
\[
 \left(1-x_1-x_2,\;1-y_1-y_2,\;1-z_1-z_2\right).
\]
If it is zero, Lemma~\ref{lem: zero-complement-two-row-solutions} gives
$Z_0(q)$ solutions. Otherwise the new row is orthogonal to both original
rows (Lemma~\ref{lem: complementary-coordinate-two-parameter}). Since both are of type~I, it cannot be proportional to either, so
the augmented solution has a core of three points (Proposition~\ref{prop: geometry-of-two-parameter-cores}). Its new row either
has no zero entries, contributing $N_3(q)$,
or has one zero entry, contributing $3M_3(q)$. Conversely, deleting the
third row in either of these cases gives a solution with all six entries
non-zero and a core consisting of two type-I points.

Hence
\begin{align*}
 A_2(q)={}&(q-2)^2(q-1)^2+2(q-2)(q-3)^2(q-1)+Z_0(q)+N_3(q)+3M_3(q),
\end{align*}
which yields the claimed formula.
\end{proof}

\begin{proposition}\label{prop: two-support-prescribed-zero-count}
For a fixed support of size two and a prescribed zero entry,
\[
 B_2(q)=(q-2)^2(q^2-4q+5).
\]
\end{proposition}
\begin{proof}
The core has two points. Suppose first that one is of type II. The
prescribed zero fixes the type-I row and, for each general type-II pair,
its point is uniquely determined. The type-II row has $q-1$ non-zero
scalars, but the scalar giving an entry equal to one in the column of the
prescribed zero must be excluded: it would make that column elementary.
The contribution is therefore $(q-2)^2(q-3)$.

If both points are of type I, the complementary row cannot be zero,
since the column containing the prescribed zero would then be elementary.
As in the proof of Proposition~\ref{prop: two-support-no-zero-count},
appending this row gives a three-point core. The prescribed zero is the
only one by Lemma~\ref{lem: at-most-one-zero-on-support}. Conversely, deleting the third row preserves
non-elementarity: the two non-zero entries in the column containing the
zero sum to one, so the surviving entry is not one. Thus this case contributes
$M_3(q)$, and
\[
 B_2(q)=(q-2)^2(q-3)+M_3(q)
       =(q-2)^2(q^2-4q+5).
\]
\end{proof}

\begin{proposition}\label{prop: one-support-count}
For a fixed support of size one,
\[
 A_1(q)=(q-2)^4.
\]
There are no non-degenerate, non-elementary solutions with a zero entry
in the supported row.
\end{proposition}
\begin{proof}
The three supported entries must lie outside $\{0,1\}$. For each of their
$(q-2)^3$ possible values, the single diagonal equation has $q-2$ pairs of
non-zero coefficients, as in Lemma~\ref{lem: zero-complement-two-row-solutions}.
A zero entry would make its vector elementary.
\end{proof}

We shall also need the non-elementary counts with unequal height bounds.
\begin{lemma}\label{lem: bounded-non-elementary-two-parameter-count}
Let $0\le a\le b\le c\le m$. The number of non-degenerate,
non-elementary solutions of \eqref{eq: system with a1 and a2} with
\[
 \operatorname h(\x)\le a,\qquad
 \operatorname h(\y)\le b,\qquad
 \operatorname h(\z)\le c
\]
is independent of $c$ and equals
\begin{equation}\label{eq: bounded-two-parameter-count}
\begin{split}
 N_q(a,b)={}&\sum_{s=1}^{a}\binom as\bigl(A_s(q)+3sB_s(q)\bigr)+(b-a)\sum_{s=2}^{a+1}\binom a{s-1}B_s(q).
\end{split}
\end{equation}
Empty sums are zero, and $B_1(q)=0$.
\end{lemma}
\begin{proof}
If the support is contained in $[a]$, choose its $s$ indices and use
$A_s(q)$ for no zero entries or $3sB_s(q)$ for one zero entry in a
prescribed position. If the support is not contained in $[a]$,
Lemma~\ref{lem: at-most-one-zero-on-support} allows exactly one supported
index $v>a$. It must satisfy $v\le b$, since otherwise that row would
have two zeros. Its $x$-entry is the unique zero in the supported rows.
Choose $v$ in $b-a$ ways and the other $s-1$ support indices in $[a]$;
the count is then $B_s(q)$. These alternatives give the formula.
\end{proof}

\section{Applications to uniformity}
\label{sec: applications to irreducible subalgebras}

In this final section, we first use the results of Section~\ref{sec: finite-field systems} to compute
$g_n(q^{n+2})$ and $g_n(q^{n+3})$. We shall see in particular that for fixed $n$, these quantities are polynomials in $q$ with universal integer coefficients. We then use these counts and the decomposition into
irreducible factors to compute $f_n(q^9)$.

For the first two computations, we use the correspondence between
subalgebras of $\o^n$ contained in $\pi\o^n$ and irreducible unital
subalgebras of $\o^{n+1}$. Thus $n$ denotes the rank of the former ambient
algebra, and the two counts are $g_{n+1}(q^{n+3})$ and
$g_{n+1}(q^{n+4})$. Replacing $n$ by $n-1$ recovers the notation of the
theorems stated in the introduction. Recall from \eqref{eq: g_n in terms of h}
that the non-homothetic terms
contributing to $g_{n+1}(q^e)$ are indexed by
$\emptyset\ne I\subseteq[n-1]$ and satisfy
\[
 nr_0+\sum_{\iota\in I}\iota r_\iota=e,\qquad
 r_0\ge1,\quad r_\iota\ge1.
\]
A homothetic term is added when $n\mid e$.

\subsection{Enumeration of irreducible unital subalgebras of $\o^{n+1}$ of indices $q^{n+3}$ and $q^{n+4}$}

\subsubsection{\textbf{The case of index $q^{n+3}$}}\label{subsubsec: gn(q,n+2)}
\label{subsubsec: irreducible-index-n-plus-three}

We first compute $g_{n+1}(q^{n+3})$. For $n\geq4$, the decomposition by
elementary-divisor type gives
\begin{equation}\label{eq: irreducible-excess-three-decomposition}
\begin{split}
g_{n+1}(q^{n+3})={}&h_n(\{1\},1,3)+h_n(\{3\},1,1)+h_n\bigl(\{1,2\},1,(1,1)\bigr).
\end{split}
\end{equation}
The terms with $\#I=1$ are given by
Corollary~\ref{coro: formula for h_n for |I|=1}, so only the last term
remains to be determined. The smaller ranks will be included at the end.

\begin{proposition}\label{prop: h-12-small-index}
For every $n\geq3$,
\begin{equation}\label{eq: h-12-small-index}
\begin{split}
h_n\bigl(\{1,2\},1,(1,1)\bigr)=q^{n-2}\Biggl(
&\binom{n+1}{2}[n-1]_q+(q-2)\binom{n+1}{3}+(q^2-2q+3)\binom{n+1}{4}\Biggr).
\end{split}
\end{equation}
\end{proposition}

\begin{proof}
We keep the notation from \S\ref{sec: system B=aC}.
Put $m=n-2$ and, for $\alpha\in\mathbb F_q$ and
$\boldsymbol\beta,\boldsymbol\gamma\in\mathbb F_q^m$, write
\[
E(\alpha,\boldsymbol\beta,\boldsymbol\gamma)
=\begin{pmatrix}
1&0&0\\
\alpha&1&0\\
\boldsymbol\beta&\boldsymbol\gamma&\operatorname{Id}_m
\end{pmatrix}.
\]
We apply Proposition~\ref{prop: expression for hn as an integral}, with
$I=\{1,2\}$ and $(r_0,r_1,r_2)=(1,1,1)$. For the corresponding matrix
$B=E(a,\mathbf b,\mathbf c)$ over $\o$, the only non-trivial conditions are
\[
 \pi\mid b_r^2-b_r-a(c_r^2-c_r),\qquad
 \pi\mid b_rb_s-ac_rc_s\quad(r<s).
\]
They depend only on $E=\overline B$. Call $E$ admissible if
$(\alpha,\boldsymbol\gamma,\boldsymbol\beta)$ solves
\eqref{eq: system with one alpha}.
The integral prefactor is
$q^{(n-1)+2(n-2)}=q^{3n-5}$, and every admissible residue class in
$\operatorname L_I(\o)$ has measure $q^{-(2n-3)}$. Thus the formula gives
directly
\begin{equation}\label{eq: h-12-residual-count}
h_n\bigl(\{1,2\},1,(1,1)\bigr)=q^{n-2}A,
\qquad
A=\sum_{E\text{ admissible}}\tau_{I,E}.
\end{equation}

We evaluate this weighted sum by counting the corresponding $I$-reduced
matrices. For distinct $i,j\in[n]$, let
$\ell_1<\cdots<\ell_m$ be the remaining column indices. Define
$R(i,j,\alpha,\boldsymbol\beta,\boldsymbol\gamma)$ by prescribing its
columns as
\[
R_i=(1,\alpha,\beta_1,\ldots,\beta_m)^{\mathsf T},\qquad
R_j=(0,1,\gamma_1,\ldots,\gamma_m)^{\mathsf T},\qquad
R_{\ell_r}=\e_{r+2}.
\]
These are the column permutations of $E$ in which its last $m$ diagonal
ones occur in increasing column positions. The definition of an
$I$-reduced matrix gives precisely the following possibilities:
\[
\begin{array}{c|c}
i<j&\alpha=0,\quad
\operatorname h(\boldsymbol\beta)<i,\quad
\operatorname h(\boldsymbol\gamma)<j-1\\[2pt]
j<i&\operatorname h(\boldsymbol\beta)<i-1,\quad
\operatorname h(\boldsymbol\gamma)<j.
\end{array}
\]
In the second case, $\alpha$ is arbitrary. Thus $A$ counts these reduced
matrices whose parameters solve the one-parameter system.

The $i<j$ case contributes $i q^{j-2}$. For $j<i$, the degenerate
parameters contribute $(i-1)q^{j-1}$ and the non-degenerate ones contribute
$F_q(j-1,i-2)$ by Lemma~\ref{lem: bounded-one-parameter-count}. Hence
\[
 A=\sum_{i<j}i q^{j-2}
   +\sum_{j<i}\bigl((i-1)q^{j-1}+F_q(j-1,i-2)\bigr).
\]
Substituting $F_q$ and $G_q$, the geometric terms combine to
$\binom{n+1}{2}[n-1]_q$, leaving
\[
 A=\binom{n+1}{2}[n-1]_q
   +(q-1)\sum_{j<i}j(i-1)-\sum_{j<i}j
   +(q-2)(q-3)\sum_{j<i}\binom j2.
\]
The identities
\[
 \sum_{j<i}j(i-1)=3\binom{n+1}{4}+\binom{n+1}{3},\qquad
 \sum_{j<i}j=\binom{n+1}{3},\qquad
 \sum_{j<i}\binom j2=\binom{n+1}{4}
\]
therefore give
\[
 A=\binom{n+1}{2}[n-1]_q+(q-2)\binom{n+1}{3}
    +(q^2-2q+3)\binom{n+1}{4}.
\]
Together with \eqref{eq: h-12-residual-count}, this proves the formula.
\end{proof}

\begin{proof}[Proof of Theorem~\ref{thm: irreducible-excess-two-count}]
We keep the rank parameter $n$ of this subsection, so the count in
question is $g_{n+1}(q^{n+3})$. For $n\ge4$,
\begin{equation*}
 g_{n+1}(q^{n+3})
 =\binom n3_q+\binom{n+1}{2}q^{n-1}+h_n\bigl(\{1,2\},1,(1,1)\bigr).
\end{equation*}
This follows from \eqref{eq: irreducible-excess-three-decomposition}
and Corollary~\ref{coro: formula for h_n for |I|=1}. Substituting
\eqref{eq: h-12-small-index} gives the stated formula.

For $n=3$, the type decomposition instead gives
\[
g_4(q^6)=1+h_3(\{1\},1,3)+h_3\bigl(\{1,2\},1,(1,1)\bigr)
       =q^3+14q^2+q+1.
\]
Here the initial $1$ counts the homothetic subalgebra
$\pi^2\mathfrak o^3$ and agrees with $\binom33_q$ in the stated formula.
For $n=2$,
\[
g_3(q^5)=h_2(\{1\},1,3)+h_2(\{1\},2,1)
        =3q+(q+1)=4q+1.
\]
Finally, $g_2(q^4)=1$, since the corresponding subalgebra of
$\mathfrak o$ is its unique ideal of index $q^4$.
All three values agree with the formula of the theorem.
\end{proof}

\subsubsection{\textbf{The case of index $q^{n+4}$}}
\label{sec: gn(q,n+3)}
\label{sec: irreducible-index-n-plus-four}

At the next index level, the same decomposition gives, for $n\ge5$,
\begin{equation}\label{eq: irreducible-excess-four-decomposition}
\begin{split}
 g_{n+1}(q^{n+4})={}&h_n(\{1\},1,4)+h_n(\{2\},1,2)
                       +h_n(\{4\},1,1)\\
                   &+h_n\bigl(\{1,2\},1,(2,1)\bigr)
                       +h_n\bigl(\{1,3\},1,(1,1)\bigr).
\end{split}
\end{equation}
For the smaller ranks, it takes the form
\begin{align*}
 g_2(q^5)={}&1,\\
 g_3(q^6)={}&1+h_2(\{1\},1,4)+h_2(\{1\},2,2),\\
 g_4(q^7)={}&h_3(\{1\},1,4)+h_3(\{1\},2,1)
               +h_3(\{2\},1,2)+h_3\bigl(\{1,2\},1,(2,1)\bigr),\\
 g_5(q^8)={}&1+h_4(\{1\},1,4)+h_4(\{2\},1,2)+h_4\bigl(\{1,2\},1,(2,1)\bigr)
               +h_4\bigl(\{1,3\},1,(1,1)\bigr).
\end{align*}
The terms with $\#I=1$ are given by
Corollary~\ref{coro: formula for h_n for |I|=1}. The two terms that remain
to be evaluated are
\[
 h_n\bigl(\{1,2\},1,(2,1)\bigr)\quad(n\ge3),\qquad
 h_n\bigl(\{1,3\},1,(1,1)\bigr)\quad(n\ge4).
\]

\subsubsection*{The term $h_n\bigl(\{1,2\},1,(2,1)\bigr)$}

Put $m=n-2$ and $I=\{1,2\}$. For
\[
 B=E(a,\mathbf b,\mathbf c)
 =\begin{pmatrix}
 1&0&0\\ a&1&0\\ \mathbf b&\mathbf c&\operatorname{Id}_m
 \end{pmatrix}\in\operatorname L_I(\o),
\]
write
\begin{align*}
 R_2(B)&=\{a(b_i-c_i):i\in[m]\},\\
 R_3(B)&=\{b_i^2-b_i-a(c_i^2-c_i):i\in[m]\}
       \cup\{b_ib_j-ac_ic_j:1\le i<j\le m\}.
\end{align*}
Let $\mathcal A$ be the set of residue triples
$(\alpha,\boldsymbol\beta,\boldsymbol\gamma)\in
\mathbb F_q\times(\mathbb F_q^m)^2$ such that
$(\alpha,\boldsymbol\gamma,\boldsymbol\beta)$ solves
\eqref{eq: system with one alpha} and
$\alpha(\boldsymbol\beta-\boldsymbol\gamma)=0$.
Proposition~\ref{prop: classification with one alpha} gives the disjoint
description
\begin{equation}\label{eq: admissible-residues-h-12-21}
\begin{array}{c|c}
 \alpha&\text{conditions on }\boldsymbol\beta,\boldsymbol\gamma\\ \hline
 0&\boldsymbol\beta\text{ elementary},\quad
       \boldsymbol\gamma\text{ arbitrary}\\
 1&\boldsymbol\beta=\boldsymbol\gamma\text{ arbitrary}\\
 \alpha\notin\{0,1\}&
       \boldsymbol\beta=\boldsymbol\gamma\text{ elementary}.
\end{array}
\end{equation}
For a triple in $\mathcal A$, define its lifting volume by
\[
 \mathcal V(\alpha,\boldsymbol\beta,\boldsymbol\gamma)
 :=\int\limits_{\substack{\overline B=E(\alpha,\boldsymbol\beta,\boldsymbol\gamma)\\
                         \pi^2\mid R_3(B)}}|dB|.
\]

\begin{lemma}\label{lem: h-12-21-residue-reduction}
For every $n\ge3$,
\[
 h_n\bigl(\{1,2\},1,(2,1)\bigr)
 =q^{4n-6}\sum_{(\alpha,\boldsymbol\beta,\boldsymbol\gamma)\in\mathcal A}
   \tau_{I,E(\alpha,\boldsymbol\beta,\boldsymbol\gamma)}
   \mathcal V(\alpha,\boldsymbol\beta,\boldsymbol\gamma).
\]
\end{lemma}
\begin{proof}
Apply Proposition~\ref{prop: expression for hn as an integral} with
$(r_0,r_1,r_2)=(1,2,1)$. Its prefactor is
$q^{2(n-1)+2(n-2)}=q^{4n-6}$, and its only non-trivial conditions are
\[
 \pi\mid R_2(B),\qquad \pi^2\mid R_3(B).
\]
The first condition depends only on $\overline B$ and gives
$\alpha(\boldsymbol\beta-\boldsymbol\gamma)=0$. The second implies the
one-parameter system modulo $\pi$. Thus precisely the residue classes in
$\mathcal A$ contribute, and the remaining condition on each class is
the one defining $\mathcal V$.
\end{proof}

\begin{lemma}\label{lem: h-12-21-lifting-volumes}
For $(\alpha,\boldsymbol\beta,\boldsymbol\gamma)\in\mathcal A$,
\[
 \mathcal V(\alpha,\boldsymbol\beta,\boldsymbol\gamma)
 =q^{-(3n-4)}\lambda(\alpha,\boldsymbol\beta,\boldsymbol\gamma),
\]
where $\lambda=q$ in the following cases:
\begin{enumerate}
 \item $\alpha=0$, $\boldsymbol\beta=0$, and
       $\#\Supp(\boldsymbol\gamma)\le1$;
 \item $\alpha=0$, $\boldsymbol\beta=\e_u$, and either
       $\Supp(\boldsymbol\gamma)\subseteq\{u\}$, or
       $\Supp(\boldsymbol\gamma)\setminus\{u\}=\{v\}$ with
       $\gamma_u+\gamma_v=1$;
 \item $\alpha=1$ and either $\#\Supp(\boldsymbol\gamma)\le1$, or
       $\Supp(\boldsymbol\gamma)=\{u,v\}$ with $\gamma_u+\gamma_v=1$;
 \item $\alpha\notin\{0,1\}$.
\end{enumerate}
For all other triples in $\mathcal A$, $\lambda=1$.
\end{lemma}
\begin{proof}
We use $m=n-2$, so that $3n-4=3m+2$, and treat the three cases in
\eqref{eq: admissible-residues-h-12-21}.

Suppose first that $\alpha=0$. If $\boldsymbol\beta=0$, substitute
$a=\pi a'$ and $\mathbf b=\pi\mathbf b'$. Writing
$t=\overline{a'}$ and $d_i=\overline{b_i'}$, the conditions modulo
$\pi^2$ become
\[
 d_i=-t\gamma_i(\gamma_i-1),\qquad
 t\gamma_i\gamma_j=0\quad(i<j).
\]
For each admissible $t$, the $d_i$ are uniquely determined. The
Jacobian is $q^{-(m+1)}$, while fixing the residues of
$\mathbf c$ and $(a',\mathbf b')$ contributes $q^{-m}q^{-(m+1)}$.
Hence $\mathcal V$ is $q^{-(3m+2)}$ times the number of possible $t$.
That number is $q$ if $\#\Supp(\boldsymbol\gamma)\le1$ and is one
otherwise.

If $\boldsymbol\beta=\e_u$, use
$a=\pi a'$, $b_u=1+\pi b_u'$ and $b_i=\pi b_i'$ for $i\ne u$.
The diagonal equations again determine every $d_i$ from $t$. Eliminating
them from the cross equations leaves
\[
 \begin{aligned}
 t\gamma_j(\gamma_u+\gamma_j-1)&=0&& (j\ne u),\\
 t\gamma_i\gamma_j&=0&& (i<j,\ i,j\ne u).
 \end{aligned}
\]
Thus $t$ is arbitrary precisely in case (b); in all other cases it is
zero. The same measure calculation proves (a) and (b).

Now suppose that $\alpha=1$, so that
$\boldsymbol\beta=\boldsymbol\gamma$. Substitute
$a=1+\pi a'$ and $\mathbf b=\mathbf c+\pi\mathbf b'$.
The Jacobian is $q^{-(m+1)}$, and the residue equations are the homogeneous
linear system
\begin{equation*}
\begin{cases}
 \gamma_jd_i+\gamma_id_j=t\gamma_i\gamma_j,&i<j,\\
 (2\gamma_i-1)d_i=t\gamma_i(\gamma_i-1),&i\in[m].
\end{cases}
\end{equation*}
Let $L(\boldsymbol\gamma)$ be its number of solutions
$(t,\mathbf d)\in\mathbb F_q\times\mathbb F_q^m$. The residue condition
on $\mathbf c$ and the preceding Jacobian give
\[
 \mathcal V(1,\boldsymbol\gamma,\boldsymbol\gamma)
 =q^{-(3m+2)}L(\boldsymbol\gamma).
\]
The diagonal equations force $d_i=0$ when $\gamma_i=0$. If
$\#\Supp(\boldsymbol\gamma)\le1$, there are $q$ solutions: $t$ is free
when $\boldsymbol\gamma$ is elementary, and otherwise the supported
$d_i$ is free and determines $t$.

For support $\{u,v\}$, put $g=\gamma_u$, $h=\gamma_v$ and
$r_u=d_u/g$, $r_v=d_v/h$. The cross equation gives $t=r_u+r_v$, and
the remaining equations have matrix
\[
 \begin{pmatrix}g&1-g\\1-h&h\end{pmatrix},
 \qquad \det=g+h-1.
\]
Its rank is at least one. Thus $L(\boldsymbol\gamma)=q$ if $g+h=1$
and $L(\boldsymbol\gamma)=1$ otherwise. If the support has at least three
elements, the cross equations instead give $d_i/\gamma_i=r$ throughout
the support and $t=2r$. The diagonal equations then force $r=0$.
Only the zero solution remains, so $L(\boldsymbol\gamma)=1$. This proves (c).

Finally, let $\alpha\notin\{0,1\}$. Then
$\boldsymbol\beta=\boldsymbol\gamma$ is elementary. Using its canonical
lift in $\o^m$, substitute
$\mathbf b=\boldsymbol\gamma+\pi\mathbf b'$ and
$\mathbf c=\boldsymbol\gamma+\pi\mathbf c'$. The conditions reduce to
$\overline{b_i'}=\alpha\overline{c_i'}$ for $i\in[m]$; the cross
conditions add nothing. The Jacobian contributes $q^{-2m}$, the condition
$\overline a=\alpha$ contributes $q^{-1}$, and the $m$ residue equations
contribute $q^{-m}$. Their product is $q^{-(3m+1)}$, proving (d).
\end{proof}

Combining the two lemmas gives
\begin{equation}\label{eq: h-12-21-weighted-lift-count}
 h_n\bigl(\{1,2\},1,(2,1)\bigr)
 =q^{n-2}\sum_{(\alpha,\boldsymbol\beta,\boldsymbol\gamma)\in\mathcal A}
   \tau_{I,E(\alpha,\boldsymbol\beta,\boldsymbol\gamma)}
   \lambda(\alpha,\boldsymbol\beta,\boldsymbol\gamma).
\end{equation}
We now evaluate this sum by separating the two values of $\lambda$.

\begin{proposition}\label{prop: h-12-21-final}
For every $n\ge3$,
\begin{equation*}
\begin{split}
 h_n\bigl(\{1,2\},1,(2,1)\bigr)=q^{n-2}\Biggl(
 &\binom{n+1}{2}[n-1]_q+(q^2+q-3)\binom{n+1}{3}+6(q-1)^2\binom{n+1}{4}\Biggr).
\end{split}
\end{equation*}
\end{proposition}
\begin{proof}
For $\varepsilon\in\{1,q\}$, let $\mathcal R_\varepsilon$ be the set of
$I$-reduced matrices $R(i,j,\alpha,\boldsymbol\beta,\boldsymbol\gamma)$
introduced in the proof of Proposition~\ref{prop: h-12-small-index}, whose
parameters belong to $\mathcal A$ and satisfy
$\lambda(\alpha,\boldsymbol\beta,\boldsymbol\gamma)=\varepsilon$.
The definition of $\tau_{I,E}$ gives
\[
 C_\varepsilon:=\#\mathcal R_\varepsilon
 =\sum_{\substack{(\alpha,\boldsymbol\beta,\boldsymbol\gamma)\in\mathcal A\\
        \lambda(\alpha,\boldsymbol\beta,\boldsymbol\gamma)=\varepsilon}}
   \tau_{I,E(\alpha,\boldsymbol\beta,\boldsymbol\gamma)}.
\]
Thus \eqref{eq: h-12-21-weighted-lift-count} becomes
\begin{equation}\label{eq: h-12-21-two-matrix-counts}
 h_n\bigl(\{1,2\},1,(2,1)\bigr)=q^{n-2}(C_1+qC_q).
\end{equation}

We use the same height bounds as in the proof of
Proposition~\ref{prop: h-12-small-index}. For fixed pivot columns $i,j$,
the contributions to $C_q$ are
\[
\begin{array}{c|c|c}
 &i<j&j<i\\ \hline
 \alpha=0
 &i\bigl(1+(j-2)(q-1)\bigr)
 &j(i-j)+j(j-1)(q-1)\\[2pt]
 \alpha=1
 &0&1+(j-1)(q-1)+(q-2)\binom{j-1}{2}\\[2pt]
 \alpha\notin\{0,1\}
 &0&(q-2)j
\end{array}
\]
To justify the first row, let $b$ be the number of allowed coordinates
of $\boldsymbol\gamma$. If $\boldsymbol\beta=0$, or if
$\boldsymbol\beta=\e_u$ with $u\le b$, the condition $\lambda=q$
allows $1+b(q-1)$ choices of $\boldsymbol\gamma$. Indeed, in the latter
case there are $q$ choices supported on $\{u\}$ and
$(b-1)(q-1)$ choices with one further supported coordinate and coordinate
sum one. If $u>b$, it instead allows the $b+1$ elementary vectors.
For $i<j$ there are $i$ possible elementary vectors
$\boldsymbol\beta$, all of the first kind. For $j<i$, split the
$i-2$ possible indices $u$ at $j-1$; this gives the displayed contribution.

In the second row, $\boldsymbol\beta=\boldsymbol\gamma$ is supported
in $[j-1]$. Support of size at most one gives $1+(j-1)(q-1)$ vectors,
and support of size two with coordinate sum one gives
$(q-2)\binom{j-1}{2}$. The third row has $q-2$ choices of $\alpha$ and
$j$ choices of the common elementary vector.

Summing the table over $i,j\in[n]$ and using the binomial identities
from the preceding computation, we obtain
\begin{align*}
 C_q={}&q\binom{n+1}{3}+(5q-4)\binom{n+1}{4}
       +\binom n2+(q-1)\binom n3+(q-2)\binom n4\\
     ={}&(q+1)\binom{n+1}{3}+6(q-1)\binom{n+1}{4}.
\end{align*}

For $C_1$, count the complementary admissible reduced matrices. For fixed
$i<j$, all $i q^{j-2}$ admissible parameters have $\alpha=0$.
For $j<i$, the three classes in
\eqref{eq: admissible-residues-h-12-21} contribute respectively
$(i-1)q^{j-1}$, $q^{j-1}$, and $(q-2)j$. Consequently,
\begin{align*}
 C_1+C_q
 &=\sum_{i<j}i q^{j-2}
   +\sum_{j<i}\bigl(i q^{j-1}+(q-2)j\bigr)\\
 &=\binom{n+1}{2}[n-1]_q+(q-2)\binom{n+1}{3},
\end{align*}
and hence
\[
 C_1=\binom{n+1}{2}[n-1]_q
       -3\binom{n+1}{3}-6(q-1)\binom{n+1}{4}.
\]
Substituting these two counts into
\eqref{eq: h-12-21-two-matrix-counts} proves the formula.
\end{proof}

\subsubsection*{The term $h_n\bigl(\{1,3\},1,(1,1)\bigr)$}

Put $m=n-3$, $I=\{1,3\}$, and write
\[
 E_{13}(\boldsymbol\alpha,\boldsymbol\beta,
                 \boldsymbol\gamma_1,\boldsymbol\gamma_2)
 =\begin{pmatrix}
 1&0&0&0\\ \alpha_1&1&0&0\\ \alpha_2&0&1&0\\
 \boldsymbol\beta&\boldsymbol\gamma_1&\boldsymbol\gamma_2&\operatorname{Id}_m
 \end{pmatrix},\qquad \boldsymbol\alpha=(\alpha_1,\alpha_2).
\]
The formula in Proposition~\ref{prop: expression for hn as an integral} now has prefactor $q^{(n-1)+3(n-3)}=q^{4n-10}$.
Its only non-trivial condition is the two-parameter system
\eqref{eq: system with a1 and a2} modulo $\pi$, with
$(\x,\y,\z)=(\boldsymbol\gamma_1,\boldsymbol\gamma_2,\boldsymbol\beta)$.
Every admissible residue class has measure $q^{-(3n-7)}$, so
\begin{equation}\label{eq: h-13-residual-count}
 h_n\bigl(\{1,3\},1,(1,1)\bigr)=q^{n-3}C_{13}(n,q),\qquad
 C_{13}(n,q)=\sum_{\text{solutions of }\eqref{eq: system with a1 and a2}}
                 \tau_{I,E_{13}}.
\end{equation}

As before, we interpret this sum as a count of reduced matrices.
Choose distinct $i,j,k\in[n]$ with $j<k$, and let
$\ell_1<\cdots<\ell_m$ be the remaining column indices. Define
$R(i,j,k,\boldsymbol\alpha,\boldsymbol\beta,
                 \boldsymbol\gamma_1,\boldsymbol\gamma_2)$ by
\[
 \begin{aligned}
 R_i&=(1,\alpha_1,\alpha_2,\boldsymbol\beta^{\mathsf T})^{\mathsf T},&
 R_j&=(0,1,0,\boldsymbol\gamma_1^{\mathsf T})^{\mathsf T},\\
 R_k&=(0,0,1,\boldsymbol\gamma_2^{\mathsf T})^{\mathsf T},&
 R_{\ell_r}&=\e_{r+3}.
 \end{aligned}
\]
Let $a,b,c$ denote the respective upper bounds on
$\operatorname h(\boldsymbol\gamma_1)$,
$\operatorname h(\boldsymbol\gamma_2)$ and
$\operatorname h(\boldsymbol\beta)$. The matrix is $I$-reduced precisely
under the following conditions:
\begin{equation}\label{eq: h-13-reduced-height-bounds}
\begin{array}{c|c|c}
 \text{order}&\text{coefficient restrictions}&(a,b,c)\\ \hline
 i<j<k&\alpha_1=\alpha_2=0&(j-2,k-3,i-1)\\
 j<i<k&\alpha_2=0&(j-1,k-3,i-2)\\
 j<k<i&\text{none}&(j-1,k-2,i-3).
\end{array}
\end{equation}
Indeed, each non-zero entry below a pivot must occur in a row whose own
pivot precedes that column. Deleting the three pivot columns gives the
height bounds. The middle block requires $j<k$, and a non-zero
$\alpha_1$ or $\alpha_2$ requires respectively $j<i$ or $k<i$.

\begin{proposition}\label{prop: h-13-final}
For $n\ge4$,
\[
 h_n\bigl(\{1,3\},1,(1,1)\bigr)=q^{n-3}\Phi_n(q),
\]
where
\begin{align*}
 \Phi_n(q)={}&\binom{n+1}{2}\binom{n-1}{2}_q+[n-2]_q\left((q-2)\binom{n+1}{3}
                    +(q^2-2q+3)\binom{n+1}{4}\right)
       -q\binom{n+1}{4}\\
 &+(q+1)(q-2)(q^2-2q+2)\binom{n+1}{5}+(q+1)(q^4-5q^3+11q^2-10q+5)\binom{n+1}{6}.
\end{align*}
\end{proposition}
\begin{proof}
We count the matrices in \eqref{eq: h-13-residual-count} using
\eqref{eq: h-13-reduced-height-bounds}.
First suppose that $\alpha_1\alpha_2=0$. If both coefficients are zero,
$\boldsymbol\beta$ is elementary and the two other vectors are arbitrary.
If only one coefficient is zero, the remaining system is the one-parameter
system, while the vector with zero coefficient is arbitrary. By
Lemma~\ref{lem: bounded-one-parameter-count}, the degenerate contribution
for each pivot triple is
\[
 d(i,j,k)=
 \begin{cases}
 (c+1)q^{a+b},&i<j<k,\\
 (c+1)q^{a+b}+q^bF_q(a,c),&j<i<k,\\
 (c+1)q^{a+b}+q^bF_q(a,c)+q^aF_q(b,c),&j<k<i.
 \end{cases}
\]
These cases are disjoint: the term with both coefficients zero is counted
once, and each $F_q$ counts a non-zero coefficient.

Non-degenerate solutions require $j<k<i$, so $a\le b\le c$.
The three elementary families from Section~\ref{sec: system B=a1C1+a2C2}
have the following contributions:
\[
\begin{array}{c|c}
 \text{family}&\text{number}\\ \hline
 \boldsymbol\beta\text{ elementary}
   &(c+1)(q-1)F_q(a,b)\\
 \boldsymbol\gamma_1\text{ elementary},\quad
     \boldsymbol\beta\text{ non-elementary}
   &(a+1)(q-1)G_q(b)\\
 \boldsymbol\gamma_2\text{ elementary},\quad
     \boldsymbol\beta,\boldsymbol\gamma_1\text{ non-elementary}
   &(b+1)(q-1)G_q(a).
\end{array}
\]
In the first row, choose the elementary vector and $\alpha_2\ne0$.
The remaining vectors and the parameter $-\alpha_1/\alpha_2$ are
counted by $F_q(a,b)$. The other rows follow from the
non-elementary counts $G_q(b)$ and $G_q(a)$ in the same way.
Denote the sum of these three entries by $e_q(a,b,c)$.

The remaining, non-elementary contribution is $N_q(a,b)$ from
Lemma~\ref{lem: bounded-non-elementary-two-parameter-count}. Therefore
\begin{equation*}
 \begin{split}
 C_{13}(n,q)={}&
  \sum_{\substack{i,j,k\in[n]\ j<k,\ i\notin\{j,k\}}}d(i,j,k)+\sum_{j<k<i}\bigl(e_q(j-1,k-2,i-3)+N_q(j-1,k-2)\bigr).
 \end{split}
\end{equation*}
All terms are given explicitly by the results of Section~\ref{sec: finite-field systems}.

We first expand the $F_q$-terms in the degenerate contribution and
collect the powers coming from the $q^{\min(u,v)}$ term of $F_q(u,v)$.
Together with $(c+1)q^{a+b}$, these have coefficients $c+1=i$,
$c+2=i$, and $c+3=i$ in the three pivot orders, respectively.
Deleting the pivot column $i$ and reindexing the other two pivots gives
\[
 \sum_{i=1}^n i\sum_{1\le j'<k'\le n-1}q^{j'+k'-3}
 =\binom{n+1}{2}\binom{n-1}{2}_q.
\]
For the elementary contribution, expanding $F_q(a,b)$ in the table above gives
\[
\begin{aligned}
 e_q(a,b,c)={}&(q-1)^2(c+1)(a+1)(b+1)\\
 &+(q-1)\bigl((b+c+2)G_q(a)+(a+1)G_q(b)\bigr).
\end{aligned}
\]
For the non-elementary contribution, we sum over the support size before
the pivot positions. Put $m=n-3$. The substitution
$(a,b,c)=(j-1,k-2,i-3)$ identifies $j<k<i$ with
$0\le a\le b\le c\le m$. For every integer $r\ge0$, Vandermonde's identity gives
\[
\begin{aligned}
 \sum_{0\le a\le b\le c\le m}\binom ar
 &=\sum_{a=0}^m\binom{m-a+2}{2}\binom ar
 =\binom{m+3}{r+3},\\
 \sum_{0\le a\le b\le c\le m}(b-a)\binom ar
 &=\sum_{a=0}^m\binom{m-a+2}{3}\binom ar
 =\binom{m+3}{r+4}.
\end{aligned}
\]
Applying these identities to the two sums in
\eqref{eq: bounded-two-parameter-count} yields the single support sum
\[
 \sum_{j<k<i}N_q(j-1,k-2)
 =\sum_{s=1}^{n-3}\binom n{s+3}
       \bigl(A_s(q)+(3s+1)B_s(q)\bigr).
\]
Indeed, after summing over the pivot positions, the two sums in that
formula contribute $3sB_s(q)$ and $B_s(q)$, respectively, in addition
to $A_s(q)$.

It remains to insert the explicit expressions for $F_q$, $G_q$, $A_s$
and $B_s$ and evaluate finite sums. The geometric and support sums are
evaluated using
\[
\begin{aligned}
 \sum_{a=0}^m\binom{m-a+r}{r}q^a
 &=\frac{q^{m+r+1}-\sum_{t=0}^r\binom{m+r+1}{t}(q-1)^t}
         {(q-1)^{r+1}},\\
 \sum_{s=0}^{n-r}\binom n{s+r}T^s
 &=\frac{(1+T)^n-\sum_{t=0}^{r-1}\binom nt T^t}{T^r},
 \qquad 1\le r\le n.
\end{aligned}
\]
The first identity holds for $r\ge0$, and the second is the binomial
theorem with its first $r$ terms removed. Polynomial factors in $a$ or
$s$ are obtained by applying $q\frac{d}{dq}$ or $T\frac{d}{dT}$ to the
respective identities; in the support sum we then set $T=q-1$.
Combining the degenerate, elementary and non-elementary contributions
and collecting terms gives
\begin{align*}
 C_{13}(n,q)-\binom{n+1}{2}\binom{n-1}{2}_q
 ={}&[n-2]_q\left((q-2)\binom{n+1}{3}
                   +(q^2-2q+3)\binom{n+1}{4}\right)\\
 &-q\binom{n+1}{4}
  +(q+1)(q-2)(q^2-2q+2)\binom{n+1}{5}\\
 &+(q+1)(q^4-5q^3+11q^2-10q+5)\binom{n+1}{6}.
\end{align*}
Together with \eqref{eq: h-13-residual-count}, this proves the formula.
\end{proof}

\begin{proof}[Proof of Theorem~\ref{thm: irreducible-excess-three-final}]
As above, we use $n$ for the rank of the ambient algebra $\o^n$
and compute $g_{n+1}(q^{n+4})$.
For $n\ge5$, substitute Propositions~\ref{prop: h-12-21-final} and
\ref{prop: h-13-final} into
\eqref{eq: irreducible-excess-four-decomposition}, together with
\begin{align*}
 h_n(\{1\},1,4)&=\binom{n+1}{2}q^{n-1},\\
 h_n(\{2\},1,2)&=
   \left(\binom{n+1}{3}+3\binom{n+1}{4}\right)q^{2n-4},\\
 h_n(\{4\},1,1)&=\binom n4_q.
\end{align*}
For the second identity, Corollary~\ref{coro: formula for h_n for |I|=1} uses
$\theta_{n-2}=\binom{n+1}{3}+3\binom{n+1}{4}$. Expanding $\Phi_n(q)$ and
collecting the coefficients of $\binom{n+1}{j}$ gives the stated formula.
At $n=4$, the homothetic term replaces $h_n(\{4\},1,1)$ and equals
$\binom44_q=1$. The smaller-rank decompositions give
\[
 g_3(q^6)=q^2+4q+1,\qquad
 g_4(q^7)=10q^3+12q^2+q+1.
\]
Finally, the rank-one ideal count gives $g_2(q^5)=1$.
These values agree with the formula of the theorem.
\end{proof}

\subsection{Enumeration of unital subalgebras of $\o^n$ of index $q^9$}
\label{sec: unital-subalgebras-index-nine}

We return to finite-index unital $\o$-subalgebras of $\o^n$. We begin
with the decomposition underlying Liu's recurrence
and use it to compute $f_n(q^9)$.

\begin{theorem}\label{prop: irreducible-factorization-over-o}
A finite-index unital $\o$-subalgebra $\Lambda$ of $\o^n$ is
irreducible if and only if it is local. Moreover, there is a unique
partition $\mathcal P$ of $[n]$ such that
\[
 \Lambda=\prod_{B\in\mathcal P}\Lambda_B
 \quad\text{inside}\quad
 \o^n=\prod_{B\in\mathcal P}\o^B,
\]
where $\Lambda_B$ is a finite-index irreducible unital subalgebra
of $\o^B$.
\end{theorem}
\begin{proof}
Since $\o$ is complete, it is henselian. As $\Lambda$ is finite over
$\o$, it is a finite product of local rings
\cite[Propositions 18.5.9(ii) and 18.5.14]{Grothendieck1967EGAIV4}.

If $\Lambda$ is local, the coordinate projections followed by
reduction modulo $\m$ are surjective maps $\Lambda\to\mathbb F_q$.
Their kernels are all the unique maximal ideal, and the induced
maps from the residue field coincide because they agree on the
diagonal scalars. Hence $\Lambda$ is irreducible.
Conversely, the idempotents of $\o^n$ are exactly the sums of
distinct standard basis vectors. If $\Lambda$ is irreducible,
its only idempotents are therefore $0$ and $1$, so its product
decomposition has a single local factor.

The identities of the local factors are precisely the primitive
idempotents of $\Lambda$. They are characteristic vectors
of disjoint nonempty subsets $B$ of $[n]$ whose union is $[n]$.
The corresponding local factors are the projections $\Lambda_B$
of $\Lambda$ onto $\o^B$. These have finite index since $\Lambda$
contains $\pi^N\o^n$ for some $N$, and they are irreducible by the
first part. The primitive idempotents determine the partition and
its factors uniquely.
\end{proof}

\begin{corollary}\label{coro: subalgebra-counts-by-blocks}
Set $f_0(1)=1$ and $f_0(q^e)=0$ for $e>0$. For $n\ge1$ and $e\ge0$,
\begin{equation}\label{eq: recurrence for f over o}
 f_n(q^e)=\sum_{i=0}^e\sum_{j=1}^n
 \binom{n-1}{j-1}f_{n-j}(q^{e-i})g_j(q^i).
\end{equation}
Equivalently, if $m_d=\#\{i:r_i=d\}$ for a tuple
$(r_1,\ldots,r_t)$, then
\begin{equation}\label{eq: direct formula for f over o}
 f_n(q^e)=
 \sum_{\substack{1\le r_1\le\cdots\le r_t,\ \sum_i r_i=n\\
                 t\ge1}}
 \frac{\binom{n}{r_1,\ldots,r_t}}{\prod_{d\ge1}m_d!}
 \sum_{\substack{e_1,\ldots,e_t\ge0\\ \sum_i e_i=e}}
 \prod_{i=1}^t g_{r_i}(q^{e_i}).
\end{equation}
\end{corollary}
\begin{proof}
Arguing as in \cite[Proposition 4.4]{Liu2007}, choose the block
containing the first coordinate. If it has size $j$ and its
irreducible factor has index $q^i$, there are $\binom{n-1}{j-1}$
choices for the block, $g_j(q^i)$ choices for its factor and
$f_{n-j}(q^{e-i})$ choices for the subalgebra on the remaining
coordinates. This gives \eqref{eq: recurrence for f over o}.

For \eqref{eq: direct formula for f over o}, count all blocks
at once. The multinomial coefficient counts ordered blocks of
the prescribed sizes. After summing over their indices, each
subalgebra is counted $\prod_d m_d!$ times, since blocks of the
same size can be interchanged. The product of the indices is
the index of the product subalgebra, which gives the inner
summation condition.
\end{proof}

For $r\ge2$ we have $g_r(q^i)=0$ if $i<r-1$ and
$g_r(q^{r-1})=1$. Also $g_1(1)=1$ and $g_1(q^i)=0$ for $i>0$.
\begin{definition}
Let $c_k(q^e)$ count the unital subalgebras of $\o^k$ of index $q^e$
whose associated partitions have all blocks of size at least two.
Set $c_0(q^e)=f_0(q^e)$ and $c_1(q^e)=0$.
\end{definition}

\begin{corollary}
For $n\ge1$ and $e\ge0$, we have
\begin{equation}\label{eq: count by nonsingleton coordinates}
 f_n(q^e)=\sum_{S\subseteq[n]}c_{\#S}(q^e)
         =\sum_{k=0}^{\min(n,2e)}\binom nk c_k(q^e).
\end{equation}
Moreover,
\begin{equation}\label{eq: coefficient formula for c}
 c_k(q^e)=
 \sum_{\substack{m_2,\ldots,m_{e+1}\ge0\\
                  \sum_{r=2}^{e+1}r m_r=k}}
 \frac{k!}{\prod_{r=2}^{e+1}\bigl((r!)^{m_r}m_r!\bigr)}
 [T^e]\prod_{r=2}^{e+1}
 \left(\sum_{i=r-1}^{e}g_r(q^i)T^i\right)^{m_r}.
\end{equation}
Here $[T^e]$ extracts the coefficient of $T^e$. For $e=0$ the empty
tuple contributes $1$ when $k=0$.
\end{corollary}
\begin{proof}
Let $S$ be the union of the blocks of size at least two in the
partition associated to a subalgebra. Each coordinate outside $S$
contributes the unique singleton factor, of index $1$. Thus for a
fixed $S$ there are $c_{\#S}(q^e)$ choices. A block of size $r\ge2$
contributes at least $r-1$ to $e$, and $r\le2(r-1)$, so $\#S\le2e$.
This proves \eqref{eq: count by nonsingleton coordinates}.

For the second formula, let $m_r$ be the number of blocks of size $r$
in $S$. Each such block contributes an exponent of at least $r-1$,
so $2\le r\le e+1$. The multinomial factor in
\eqref{eq: direct formula for f over o}, divided by $\prod_r m_r!$,
counts the partitions with these multiplicities. Recording the exponent
of each factor by a power of $T$ gives \eqref{eq: coefficient formula for c}.
\end{proof}

Formula~\eqref{eq: count by nonsingleton coordinates} is analogous to
the expansion following Proposition~4.4 of Liu \cite{Liu2007}.

\subsubsection*{The remaining irreducible count}

For $e=9$, the coefficient in each summand of
\eqref{eq: coefficient formula for c} is
\[
 [T^9]\prod_{r=2}^{10}
 \left(\sum_{i=r-1}^{9}g_r(q^i)T^i\right)^{m_r}.
\]
Since $k\le n$ in \eqref{eq: count by nonsingleton coordinates},
only the counts $g_r(q^i)$ with $2\le r\le\min(n,10)$ and
$r-1\le i\le9$ can enter the calculation of $f_n(q^9)$.

We have $g_2(q^i)=1$ for $i\ge1$. To extract the counts for $r=3,4$,
write $G_r(T)=\sum_{i\ge0}g_r(q^i)T^i$. For $r=3$, the formula for
$\Xi_2$ in Section~\ref{sec: the case I of one element} and
Corollary~\ref{coro: formula for h_n for |I|=1} give
\[
 G_3(T)=\frac{(1+T)^2}{(1-T)(1-qT^3)}-1-\frac{3T}{1-T}.
\]
Here the last two terms remove the contributions with $r_0=0$.
For $r=4$, in the notation of
Proposition~\ref{prop: the cotype zeta function of o4},
\[
 G_4(T)=F(T,T,T)-F(T,T,0),
\]
since the exponent of $z$ is the smallest elementary-divisor exponent
$r_0$ of the corresponding subalgebra of $\o^3$.
For $r\ge5$, Section~\ref{sec: Application II} and
Theorems~\ref{thm: irreducible-excess-two-count}
and~\ref{thm: irreducible-excess-three-final} give the counts up to
exponent $r+3$.

For $e<9$, the irreducible counts entering
\eqref{eq: recurrence for f over o} agree with those in
\cite[Proposition 3.1 and Tables 1--2]{AKKM2021}, with $p$ replaced by $q$.
Thus the resulting polynomials for $f_n(q^e)$ agree as well, and we give
the explicit formula only for $e=9$. The only irreducible count still
needed is $g_5(q^9)$.

By \eqref{eq: g_n in terms of h}, we need the types with
$4r_0+r_1+2r_2+3r_3=9$, where $r_0\ge1$ and the non-zero $r_j$ with
$j\ge1$ specify $I$. For $r_0=2$, only the type with $r_1=1$
occurs. For $r_0=1$, the remaining exponent $5$ is written as
$5\cdot 1$, $3\cdot1+2$, $1+2\cdot2$, $2\cdot1+3$, or $2+3$.
The two terms with $\#I=1$ contribute
$h_4(\{1\},2,1)+h_4(\{1\},1,5)=\binom41_q+10q^3$ by
Corollary~\ref{coro: formula for h_n for |I|=1}. Consequently,
\begin{equation}\label{eq: index-nine-g5-decomposition}
\begin{split}
 g_5(q^9)={}&11q^3+q^2+q+1\\
&+h_4\bigl(\{2,3\},1,(1,1)\bigr)+h_4\bigl(\{1,3\},1,(2,1)\bigr)+h_4\bigl(\{1,2\},1,(3,1)\bigr)+h_4\bigl(\{1,2\},1,(1,2)\bigr).
\end{split}
\end{equation}

We evaluate these four terms using
Proposition~\ref{prop: expression for hn as an integral} and the
$I$-admissible pairs in
Example~\ref{ex: I-admissible pairs for I of two elements}.
Write $f(t)=t^2-t$. For a tuple $v$, let $i(v)$ be the position of the
first non-zero entry of the tuple obtained by appending $1$ to $v$.
In each integral, $E$ runs over $\operatorname L_I(\mathbb F_q)$
and $V_E$ is the measure of the indicated domain with $\overline B=E$.

\begin{proposition}
We have
\[
 h_4\bigl(\{2,3\},1,(1,1)\bigr)=q^2(q^3+16q^2+16q+1).
\]
\end{proposition}
\begin{proof}
For $I=\{2,3\}$, write
\[
 B=\begin{pmatrix}
 1&0&0&0\\
 0&1&0&0\\
 a_1&a_2&1&0\\
 b_1&b_2&c&1
 \end{pmatrix},
 \qquad
 E=\overline B=\begin{pmatrix}
 1&0&0&0\\
 0&1&0&0\\
 \alpha_1&\alpha_2&1&0\\
 \beta_1&\beta_2&\gamma&1
 \end{pmatrix}.
\]
Proposition~\ref{prop: expression for hn as an integral} gives
\begin{align*}
 h_4\bigl(\{2,3\},1,(1,1)\bigr)&=q^7\sum_E\tau_{I,E}V_E,\\
 V_E&=\int\limits_{\substack{\overline B=E\\
        \pi\mid f(b_i)-a_i f(c)\ (i=1,2)}}|dB|.
\end{align*}
These are conditions on $E$ alone. A class contributes $q^{-5}$ if
$f(\beta_i)=\alpha_i f(\gamma)$ for $i=1,2$, and zero otherwise.

To compute $\tau_{I,E}$, keep the first two pivot columns in order.
The third pivot column may be inserted in any of the first
$i(\alpha_1,\alpha_2)$ positions. Write $w^{(k)}$ for the vector obtained
by inserting $\gamma$ in position $k$ of $(\beta_1,\beta_2)$. In that
order the last pivot column has $i(w^{(k)})$ possible positions, since
it must precede every other non-zero entry in the last row. Thus
\[
 \tau_{I,E}=\sum_{k=1}^{i(\alpha_1,\alpha_2)}i(w^{(k)}).
\]
For $\gamma\in\{0,1\}$ both $\beta_i$ belong to $\{0,1\}$ and the
$\alpha_i$ are free. Otherwise $\alpha_i=f(\beta_i)/f(\gamma)$.
Summing the weights gives
\[
 \begin{array}{c|ccc}
 i(\alpha_1,\alpha_2)&\gamma=0&\gamma=1&
                  \gamma\notin\{0,1\}\\ \hline
 1&11q(q-1)&4q(q-1)&q(q-2)^2\\
 2&20(q-1)&10(q-1)&5(q-2)^2\\
 3&28&17&17(q-2)
 \end{array}
\]
For instance, when $i(\alpha_1,\alpha_2)=1$ and $\gamma=0$, the four
choices of $(\beta_1,\beta_2)$ give weights $4,3,2,2$. When
$i(\alpha_1,\alpha_2)=3$, the sums over these four choices for the three
insertions are $11,9,8$ if $\gamma=0$, and $4,6,7$ for each $\gamma\ne0$.
The total weight is $q^3+16q^2+16q+1$. Multiplying by the prefactor
and the volume of a class gives the stated value.
\end{proof}

\begin{proposition}
We have
\[
 h_4\bigl(\{1,3\},1,(2,1)\bigr)=q^3(25q^2-15q+5).
\]
\end{proposition}
\begin{proof}
For $I=\{1,3\}$, write
\[
 B=\begin{pmatrix}
 1&0&0&0\\
 a_1&1&0&0\\
 a_2&0&1&0\\
 b&c_1&c_2&1
 \end{pmatrix},
 \qquad
 E=\overline B=\begin{pmatrix}
 1&0&0&0\\
 \alpha_1&1&0&0\\
 \alpha_2&0&1&0\\
 \beta&\gamma_1&\gamma_2&1
 \end{pmatrix}.
\]
Proposition~\ref{prop: expression for hn as an integral} gives
\begin{align*}
 h_4\bigl(\{1,3\},1,(2,1)\bigr)&=q^9\sum_E\tau_{I,E}V_E,\\
 V_E&=\int\limits_{\substack{\overline B=E\\
        \pi\mid a_i(b-c_i)\ (i=1,2)\\
        \pi^2\mid f(b)-a_1 f(c_1)-a_2 f(c_2)}}|dB|.
\end{align*}
Set $\delta=1$ if $\beta=0$ and $\delta=0$ otherwise, and take a unit
$d\in\o$ with residue $\beta-\delta$. On the class $E$, the substitution
\[
 a_i=a_i'\frac{b-\delta}{d}
\]
preserves the residues and the measure. Since
$f(b)=(b-\delta)(b-1+\delta)$, the last condition becomes
\[
 b\equiv 1-\delta+d^{-1}\bigl(a_1'f(c_1)+a_2'f(c_2)\bigr)\pmod{\pi^2}.
\]
The right-hand side reduces to $\beta$ precisely when
$f(\beta)=\alpha_1 f(\gamma_1)+\alpha_2 f(\gamma_2)$. The first two
divisibilities depend only on the residue class. Hence the admissible
classes are those satisfying
\[
 \alpha_i(\beta-\gamma_i)=0\quad (i=1,2),\qquad
 f(\beta)=\alpha_1 f(\gamma_1)+\alpha_2 f(\gamma_2).
\]
For each such class, $a_1',a_2',c_1,c_2$ range independently over
their residue classes, and the congruence specifies one class modulo
$\pi^2$ for $b$. The volume is therefore $q^{-6}$.

Keep the second and third pivot columns in order. The first pivot column
may be inserted in any position after the pivot column associated to each
non-zero $\alpha_i$. For each permitted insertion, the contribution to
$\tau_{I,E}$ is $i(w)$, where $w$ is obtained by inserting $\beta$ into
$(\gamma_1,\gamma_2)$ in the same position. The admissible classes give
the sums of weights in the following table.
\[
 \begin{array}{c|ccc}
 \text{coefficients}&\beta=0&\beta=1&\beta\notin\{0,1\}\\ \hline
 \alpha_1=\alpha_2=0&4(q^2+q+1)&3q^2+2q+1&0\\
 \alpha_1=0,\ \alpha_2\ne0&(q-1)(q+3)&(q-1)(q+1)&(q-2)(q+1)\\
 \alpha_1\ne0,\ \alpha_2=0&(q-1)(5q+3)&2q(q-1)&2q(q-2)\\
 \alpha_1\alpha_2\ne0&4(q-1)^2&(q-1)^2&(q-2)^2
 \end{array}
\]
In the first row, all three insertions are allowed and both $\gamma_i$
are free. For $\beta=1$, the weights are $3,5,6$ according as
$\gamma_1\ne0$, $\gamma_1=0\ne\gamma_2$, or $\gamma_1=\gamma_2=0$.
When precisely one $\alpha_i$ is non-zero, $\beta=\gamma_i$ and
$(1-\alpha_i)f(\beta)=0$. Thus for $\beta=0,1$ there are $q-1$
choices of $\alpha_i$, and for each other $\beta$ there is only
$\alpha_i=1$. The sum over the free $\gamma$ coordinate is $q+3$
at $\beta=0$ and $q+1$ otherwise if $i=2$, and $5q+3$ at $\beta=0$
and $2q$ otherwise if $i=1$. When both coefficients are non-zero,
$\beta=\gamma_1=\gamma_2$ and the weight is $4$ at $\beta=0$ and $1$
otherwise. For $\beta\notin\{0,1\}$ we also require
$\alpha_1+\alpha_2=1$, giving $q-2$ choices for the coefficient pair.

The total weight is $25q^2-15q+5$. Multiplying by the prefactor
and the volume of a class gives the stated value.
\end{proof}

\begin{proposition}
We have
\[
 h_4\bigl(\{1,2\},1,(3,1)\bigr)=10q^2(7q^2-5q+1).
\]
\end{proposition}
\begin{proof}
Put $I=\{1,2\}$ and write the matrix and its reduction as
\[
 B=\begin{pmatrix}
 1&0&0&0\\ a&1&0&0\\ b_1&c_1&1&0\\ b_2&c_2&0&1
 \end{pmatrix},\qquad
 E=\overline B=\begin{pmatrix}
 1&0&0&0\\ \alpha&1&0&0\\ \beta_1&\gamma_1&1&0\\
 \beta_2&\gamma_2&0&1
 \end{pmatrix}.
\]
Proposition~\ref{prop: expression for hn as an integral} gives
\begin{align*}
 h_4\bigl(\{1,2\},1,(3,1)\bigr)&=q^{13}\sum_E\tau_{I,E}V_E,\\
 V_E&=\int\limits_{\substack{\overline B=E,\ \pi\mid f(a)\\
       \pi^2\mid a(b_i-c_i),\ \pi^3\mid f(b_i)-a f(c_i)\ (i=1,2)\\
       \pi^3\mid b_1b_2-ac_1c_2}}|dB|.
\end{align*}
The condition on $f(a)$ restricts $\alpha$ to $0$ or $1$.
Write $\boldsymbol\beta=(\beta_1,\beta_2)$ and
$\boldsymbol\gamma=(\gamma_1,\gamma_2)$. For an elementary vector
$\boldsymbol\beta\in\{0,\e_1,\e_2\}$, let $\boldsymbol\epsilon$ be its
lift with entries in $\{0,1\}\subseteq\o$.
We first compute the lifting volumes for each residue class and then
sum their weights.

If $\alpha=1$, the conditions $\pi^2\mid a(b_i-c_i)$ require
$\boldsymbol\beta=\boldsymbol\gamma$. Set
$a=1+\pi A$ and $c_i=b_i+\pi^2 C_i$, giving a factor $q^{-5}$.
Expanding the remaining expressions modulo $\pi^3$ and dividing by $\pi$
gives the conditions
\[
 \pi^2\mid A b_1b_2+\pi(b_1C_2+b_2C_1),\qquad
 \pi^2\mid A f(b_i)+\pi(2b_i-1)C_i.
\]
When $A=\pi A'$, reduction after division by $\pi$ gives a homogeneous
system in $(A',C_1,C_2)$ with matrix
\[
 \begin{pmatrix}
 \beta_1\beta_2&\beta_2&\beta_1\\
 f(\beta_1)&2\beta_1-1&0\\
 f(\beta_2)&0&2\beta_2-1
 \end{pmatrix}.
\]
Its determinant is $\Delta(\boldsymbol\beta)=
\beta_1\beta_2(\beta_1+\beta_2-1)$. The last two rows are independent
unless both $2\beta_i-1$ vanish, since $f(\beta_i)\ne0$ whenever
$2\beta_i-1=0$. If both vanish, the first row supplies a direction
independent of the diagonal equations. Thus the rank $\rho$ is $2$
when $\Delta(\boldsymbol\beta)=0$ and
$\rho=3$ otherwise. Including the factor from $A=\pi A'$ and the two
residues of $b_i$, this branch has volume $q^{-8-\rho}$.

When $A$ is a unit, reduction of the same conditions modulo $\pi$
requires $\boldsymbol\beta$ to be elementary. Put $C_i=AC_i'$ and
$b_i=\epsilon_i+\pi B_i$. The diagonal congruences reduce to
$B_i+C_i'\equiv0\pmod\pi$, which also imply the off-diagonal congruence.
The substitutions in $b_i$ and these two residue conditions each give a
factor $q^{-2}$. The volume of this branch is $q^{-9}(1-q^{-1})$.
Consequently, for $\alpha=1$ and $\boldsymbol\beta=\boldsymbol\gamma$,
the total volume is $q^{-9}$ when $\boldsymbol\beta$ is elementary,
$q^{-10}$ for the other vectors with $\Delta(\boldsymbol\beta)=0$, and
$q^{-11}$ when $\Delta(\boldsymbol\beta)\ne0$.

For $\alpha=0$, suppose first that $a=\pi A$ with $A$ a unit.
Then $\boldsymbol\beta=\boldsymbol\gamma$, and we put $c_i=b_i+\pi C_i$.
These substitutions give a factor $q^{-3}$. Modulo $\pi^2$, the three
conditions on $f(b_i)-a f(c_i)$ and $b_1b_2-ac_1c_2$ reduce to divisibility
of $(1-\pi A)f(b_i)$ and $(1-\pi A)b_1b_2$ by $\pi^2$.
Since $1-\pi A$ is a unit, $\boldsymbol\beta$ is elementary and
$b_i=\epsilon_i+\pi^2 B_i$. The diagonal conditions modulo $\pi^3$
give $B_i\equiv AC_i\pmod\pi$, which imply the off-diagonal condition.
The factor $q^{-4}$ from the substitutions in $b_i$ and the factor
$q^{-2}$ from the residue conditions give volume $q^{-9}(1-q^{-1})$.

If $\pi^2\mid a$, set $a=\pi^2 A$. Reduction modulo $\pi^2$ again
forces $\boldsymbol\beta$ to be elementary and
$b_i=\epsilon_i+\pi^2 B_i$. After division by $\pi^2$, the remaining
congruences modulo $\pi$ are
\[
 (2\beta_i-1)\overline{B_i}=\overline A f(\gamma_i),\qquad
 \beta_2\overline{B_1}+\beta_1\overline{B_2}=\overline A\gamma_1\gamma_2.
\]
The two diagonal equations determine the residues of $B_i$.
Substituting them into the off-diagonal equation gives
$\overline A D_{\boldsymbol\beta}(\boldsymbol\gamma)=0$ over $\mathbb F_q$, where
\begin{equation}\label{eq: index-nine-D-beta}
 D_{\boldsymbol\beta}(\boldsymbol\gamma)=
 \gamma_1\gamma_2+\beta_2 f(\gamma_1)+\beta_1 f(\gamma_2)
 =
 \begin{cases}
 \gamma_1\gamma_2,&\boldsymbol\beta=0,\\
 \gamma_2(\gamma_1+\gamma_2-1),&\boldsymbol\beta=\e_1,\\
 \gamma_1(\gamma_1+\gamma_2-1),&\boldsymbol\beta=\e_2.
 \end{cases}
\end{equation}
The substitutions contribute $q^{-6}$, the fixed residues of $c_i$
contribute $q^{-2}$, and the diagonal equations contribute $q^{-2}$.
The volume is therefore $q^{-10}$ when $D_{\boldsymbol\beta}=0$ and
$q^{-11}$ otherwise. When $\boldsymbol\gamma=\boldsymbol\beta$,
$D_{\boldsymbol\beta}=0$ and the branch with $a=\pi A$ and $A$ a unit
raises the total volume to $q^{-9}$.

The non-zero fibres are summarized below. In the rows with $\alpha=1$,
we have $\boldsymbol\gamma=\boldsymbol\beta$; in the rows with $\alpha=0$,
$\boldsymbol\beta$ is elementary.
\[
\begin{array}{c|l|c}
 \alpha&\text{residue conditions}&V_E\\ \hline
 1&\boldsymbol\beta\text{ elementary}&q^{-9}\\
 1&\boldsymbol\beta\text{ non-elementary},\ \Delta(\boldsymbol\beta)=0&q^{-10}\\
 1&\Delta(\boldsymbol\beta)\ne0&q^{-11}\\[2pt] \hline
 0&\boldsymbol\gamma=\boldsymbol\beta&q^{-9}\\
 0&\boldsymbol\gamma\ne\boldsymbol\beta,\ D_{\boldsymbol\beta}(\boldsymbol\gamma)=0&q^{-10}\\
 0&D_{\boldsymbol\beta}(\boldsymbol\gamma)\ne0&q^{-11}
\end{array}
\]
All other residue classes have volume zero.

To sum the weights, put $t_\beta=3-\operatorname h(\boldsymbol\beta)$
and $t_\gamma=3-\operatorname h(\boldsymbol\gamma)$. If $\alpha=1$ and
$\boldsymbol\beta=\boldsymbol\gamma$, the pivot bounds in the proof of
Proposition~\ref{prop: h-12-small-index}
give $\tau_{I,E}=\binom{t_\beta+1}{2}$. The elementary vectors have total
weight $6+3+1=10$. The non-elementary vectors with $\Delta=0$ have one
non-zero entry different from $1$, or two non-zero entries summing to
$1$. They contribute $3(q-2)+(q-2)+(q-2)=5(q-2)$.
The remaining vectors contribute $(q-1)^2-(q-2)=q^2-3q+3$.

If $\alpha=0$, the two pivot orders give
$\tau_{I,E}=t_\beta t_\gamma+\min(t_\beta,t_\gamma)$. Indeed, reindexing
the later pivot by subtracting one gives $t_\beta t_\gamma$ pairs, with
each of the $\min(t_\beta,t_\gamma)$ ties counted in both orders.
Using \eqref{eq: index-nine-D-beta}, the sums of these weights are
\[
 \begin{array}{c|c|c|c}
 \boldsymbol\beta&\boldsymbol\gamma=\boldsymbol\beta
 &\boldsymbol\gamma\ne\boldsymbol\beta,\ D_{\boldsymbol\beta}=0
 &D_{\boldsymbol\beta}\ne0\\ \hline
 0&12&12(q-1)&4(q-1)^2\\
 \e_1&6&8+6(q-2)+3(q-1)&3(q-1)^2\\
 \e_2&2&4+2(q-2)+3+2(q-2)&2(q-1)^2+(q-2)
 \end{array}
\]
For example, with $\boldsymbol\beta=\e_1$, the middle column counts
$\boldsymbol\gamma=0$, then $\gamma_2=0$ and $\gamma_1\notin\{0,1\}$,
then $\gamma_2\ne0$ and $\gamma_1+\gamma_2=1$. With
$\boldsymbol\beta=\e_2$, the last of the corresponding cases separates
$\boldsymbol\gamma=\e_1$ from the $q-2$ pairs with two non-zero entries,
since their heights give different weights.

Thus the three $\alpha=0$ classes in the volume table have total weights
$20$, $25q-20$, and $9q^2-17q+7$, respectively. Adding the weights
$10$, $5(q-2)$, and $q^2-3q+3$ from the three $\alpha=1$ classes gives
combined weights $30$, $30(q-1)$, and $10(q-1)^2$ for the volumes
$q^{-9}$, $q^{-10}$, and $q^{-11}$. Hence
\[
 h_4\bigl(\{1,2\},1,(3,1)\bigr)
 =q^{13}\bigl(30q^{-9}+30(q-1)q^{-10}+10(q-1)^2q^{-11}\bigr).
\]
\end{proof}

\begin{proposition}
We have
\begin{equation*}
 h_4\bigl(\{1,2\},1,(1,2)\bigr)=10q^5+55q^4-15q^3.
\end{equation*}
\end{proposition}
\begin{proof}
We use the same matrices $B$ and $E$ as in the preceding calculation,
with coordinates $(a,b_1,b_2,c_1,c_2)$ and residues
$(\alpha,\beta_1,\beta_2,\gamma_1,\gamma_2)$. Proposition~\ref{prop: expression for hn as an integral}
gives
\begin{align*}
 h_4\bigl(\{1,2\},1,(1,2)\bigr)&=q^{11}\sum_E\tau_{I,E}V_E,\\
 V_E&=\int\limits_{\substack{\overline B=E\\
        \pi\mid f(c_i),\ c_1c_2\\
        \pi^2\mid f(b_i)-a f(c_i),\ b_1b_2-ac_1c_2\ (i=1,2)}}|dB|.
\end{align*}
The first set of conditions requires $\boldsymbol\gamma$ to be elementary. Reducing
the second set modulo $\pi$ then shows that $\boldsymbol\beta$ is elementary as well.

For such $\boldsymbol\beta$ and $\boldsymbol\gamma$, let
$\epsilon_i\in\{0,1\}\subseteq\o$ lift $\beta_i$ and substitute
$b_i=\epsilon_i+\pi b_i'$. Dividing the last three congruences by $\pi$ gives
\[
 (2\epsilon_i-1)b_i'\equiv a\frac{f(c_i)}{\pi}\pmod\pi,
 \qquad
 \epsilon_2b_1'+\epsilon_1b_2'\equiv a\frac{c_1c_2}{\pi}\pmod\pi.
\]
Since $2\epsilon_i-1=\pm1$, the two diagonal equations fix the residues of
$b_1'$ and $b_2'$. As $\epsilon_1\epsilon_2=0$, substituting these residues
into the cross equation leaves the condition $\pi^2\mid aD_{\boldsymbol\epsilon}(\mathbf c)$,
where $D_{\boldsymbol\epsilon}(\mathbf c)$ is the polynomial in
\eqref{eq: index-nine-D-beta} evaluated with $\boldsymbol\epsilon$ and
$\mathbf c$ in place of $\boldsymbol\beta$ and $\boldsymbol\gamma$.
The substitution and the
two fixed residues each contribute a factor $|\pi|^2$. The remaining
integral is over $a$ and $\mathbf c$ with $\overline a=\alpha$, $\overline{\mathbf c}=\boldsymbol\gamma$,
and $\pi^2\mid aD_{\boldsymbol\epsilon}(\mathbf c)$.

If $\alpha=0$, the remaining divisibility condition is automatic, because
$D_{\boldsymbol\epsilon}(\mathbf c)\in\pi\o$. The residues of $a$ and $\mathbf c$ therefore give a factor
$|\pi|^3$, and the volume of the original fibre is $|\pi|^7$.

If $\alpha\ne0$, the variable $a$ is a unit and its residue class contributes
$|\pi|$. For $\boldsymbol\beta=0$ we have $D_{\boldsymbol\epsilon}(\mathbf c)=c_1c_2$, and for $\boldsymbol\beta=\e_2$
we have $D_{\boldsymbol\epsilon}(\mathbf c)=c_1(c_1+c_2-1)$. The case $\boldsymbol\beta=\e_1$ follows by
interchanging the two indices in both $\boldsymbol\beta$ and $\boldsymbol\gamma$. We use this
interchange to evaluate the measures and retain the original labels when
counting the weights. Both factors of $D_{\boldsymbol\epsilon}(\mathbf c)$ are divisible by
$\pi$ exactly when $\boldsymbol\gamma=\boldsymbol\beta$. In the other elementary residue
classes one factor is a unit, so requiring $\pi^2\mid D_{\boldsymbol\epsilon}(\mathbf c)$ fixes
one further residue of a linear form in $c$. The measure in the $c$ variables
is thus $|\pi|^2$ for $\boldsymbol\gamma=\boldsymbol\beta$ and $|\pi|^3$ otherwise. The volume of
the original fibre is $|\pi|^7$ in the first case and $|\pi|^8$ in the second.

With rows indexed by $\boldsymbol\beta=0,\e_1,\e_2$ and columns by
$\boldsymbol\gamma=0,\e_1,\e_2$, the height conditions from the proof of
Proposition~\ref{prop: h-12-small-index} give the weights
\[
 \begin{pmatrix}12&8&4\\8&6&3\\4&3&2\end{pmatrix}
 \quad (\alpha=0),\qquad
 \begin{pmatrix}6&3&1\\5&3&1\\3&2&1\end{pmatrix}
 \quad (\alpha\ne0).
\]
The diagonal and off-diagonal sums are $20$ and $30$ for $\alpha=0$,
and $10$ and $15$ for each non-zero $\alpha$. Thus the sum of the weighted
volumes is $(10q+40)|\pi|^7+15(q-1)|\pi|^8$. Multiplying by the prefactor
$q^{3+8}$ from Proposition~\ref{prop: expression for hn as an integral} gives the stated value.
\end{proof}

\begin{corollary}\label{prop: g5-index-nine}
For every compact discrete valuation ring with residue field of size $q$,
\[
 g_5(q^9)=36q^5+126q^4-33q^3+12q^2+q+1.
\]
\end{corollary}
\begin{proof}
Substitute the four evaluations above into
\eqref{eq: index-nine-g5-decomposition}.
\end{proof}

\subsubsection*{The formula for $f_n(q^9)$}

\begin{theorem}\label{thm: f-index-nine}
For every $n\ge1$,
\[
 f_n(q^9)=\sum_{k=2}^{18}\binom nk P_k(q),
\]
where $\binom nk=0$ for $k>n$ and the polynomials $P_k$ are given below.

\[
\begin{array}{r|l}
k & P_k(q)\\ \hline
2 & 1\\
3 & q^{3} + 4 q^{2} + 4 q + 1\\
4 & 11 q^{4} + 30 q^{3} + 9 q^{2} + q + 25\\
5 & 36 q^{5} + 126 q^{4} - 33 q^{3} + 92 q^{2} + 201 q + 71\\
6 & q^{7} + 112 q^{6} + 133 q^{5} - 61 q^{4} + 623 q^{3} + 1047 q^{2} + 396 q + 571\\
7 & \begin{aligned}[t]&q^{9} + 22 q^{8} + 23 q^{7} + 59 q^{6} + 31 q^{5}\\&\quad {}+ 2047 q^{4} + 2173 q^{3} + 2473 q^{2} + 4950 q + 2486\end{aligned}\\
8 & \begin{aligned}[t]&q^{10} + q^{9} + 2 q^{8} + 2 q^{7} + 507 q^{6} + 955 q^{5}\\&\quad {}+ 3293 q^{4} + 5980 q^{3} + 23410 q^{2} + 17011 q + 13707\end{aligned}\\
9 & \begin{aligned}[t]&36 q^{8} + 37 q^{7} + 157 q^{6} + 1159 q^{5} + 8545 q^{4}\\&\quad {}+ 34997 q^{3} + 59371 q^{2} + 93649 q + 64019\end{aligned}\\
10 & 675 q^{6} + 795 q^{5} + 18960 q^{4} + 48927 q^{3} + 250632 q^{2} + 330657 q + 297103\\
11 & 990 q^{5} + 12540 q^{4} + 148830 q^{3} + 497640 q^{2} + 1157145 q + 1245992\\
12 & 13860 q^{4} + 97020 q^{3} + 1049895 q^{2} + 2961805 q + 4727041\\
13 & 135135 q^{3} + 1036035 q^{2} + 5900895 q + 15346045\\
14 & 945945 q^{2} + 7252245 q + 40500460\\
15 & 4729725 q + 80615535\\
16 & 110810700\\
17 & 91891800\\
18 & 34459425\\
\end{array}
\]

\end{theorem}
\begin{proof}
Specializing \eqref{eq: count by nonsingleton coordinates} to $e=9$
gives the asserted expansion with $P_k(q)=c_k(q^9)$.
Substituting the irreducible counts above into
\eqref{eq: coefficient formula for c} with $e=9$ gives the displayed
polynomials. For example,
when $k=18$, all the blocks are pairs and each contributes index $q$,
so $P_{18}(q)=18!/(2^9 9!)$.
\end{proof}

\section*{Acknowledgments}

The authors were supported by CONICET and Secyt (UNC). Part of this work was carried out while A. Pereda was an undergraduate student supported by an EVC fellowship (UNC).

The authors used GPT 6.1 Plus to assist with the wording and presentation of Sections~\ref{sec: finite-field systems} and~\ref{sec: applications to irreducible subalgebras}, and for language polishing in Sections~1--4. The tool was also used to review selected arguments and assist with computational checks, and provided suggestions for some minor steps in the proofs. The authors take full responsibility for the content of the paper.

The authors shared their own arguments for Section~\ref{sec: finite-field systems} with ChatGPT for editorial assistance in May 2026 and again during the revision of the manuscript. The May 2026 exchange includes their solution of the two-parameter system for $m=2$, and a record of this exchange is retained by the authors.

\bibliographystyle{abbrv}
\bibliography{reference}
\end{document}